\documentclass[11pt,twoside]{amsart}

\usepackage{latexsym}
\usepackage{amssymb}
 \usepackage{vmargin}
\usepackage{amscd}
\usepackage{stmaryrd}
\usepackage{euscript}
\usepackage{mathrsfs}
\usepackage[all]{xy}
  \usepackage{xr}

 \setmargins{32mm}{20mm}{14.6cm}{22cm}{1cm}{1cm}{1cm}{1cm}

\newcommand\da{\!\downarrow\!}

\newcommand\la{\leftarrow}
\newcommand\lra{\longrightarrow}

\newcommand\id{\mathrm{id}}

\newcommand\ten{\otimes}

\newcommand\Ru{\mathrm{R_u}}

\renewcommand\H{\mathrm{H}}

\newcommand\N{\mathbb{N}}
\newcommand\Z{\mathbb{Z}}
\newcommand\Q{\mathbb{Q}}
\newcommand\Ql{\mathbb{Q}_{\ell}}
\newcommand\Zl{\mathbb{Z}_{\ell}}
\newcommand\bFl{\mathbb{F}_{\ell}}

\newcommand\Cx{\mathbb{C}}

\newcommand\bA{\mathbb{A}}

\newcommand\bG{\mathbb{G}}
\newcommand\bH{\mathbb{H}}

\newcommand\bS{\mathbb{S}}
\newcommand\bT{\mathbb{T}}

\newcommand\cA{\mathcal{A}}
\newcommand\cB{\mathcal{B}}

\newcommand\cL{\mathcal{L}}
\newcommand\cM{\mathcal{M}}

\newcommand\cT{\mathcal{T}}

\newcommand\sF{\mathscr{F}}

\newcommand\sO{\mathscr{O}}

\newcommand\Y{\mathfrak{Y}}

\newcommand\n{\mathfrak{n}}

\newcommand\fu{\mathfrak{u}}

\newcommand\Ho{\mathrm{Ho}}

\newcommand\Hom{\mathrm{Hom}}

\newcommand\map{\mathrm{map}}

\newcommand\HHom{\underline{\mathrm{Hom}}}

\newcommand\Ext{\mathrm{Ext}}
\newcommand\EExt{\mathbb{E}\mathrm{xt}}
\newcommand\End{\mathrm{End}}

\newcommand\cocone{\mathrm{cocone}}

\newcommand\rec{\mathrm{rec}}

\newcommand\im{\mathrm{Im\,}}

\newcommand\Ob{\mathrm{Ob}}
\newcommand\ob{\mathrm{ob}}

\newcommand\CoLie{\mathrm{CoLie}}

\newcommand\Gp{\mathrm{Gp}}

\newcommand\mal{\mathrm{Mal}}

\newcommand\Spec{\mathrm{Spec}\,}
\newcommand\Dec{\mathrm{Dec}\,}
\newcommand\DEC{\mathrm{DEC}\,}

\newcommand\Set{\mathrm{Set}}

\newcommand\Com{\mathrm{Com}}

\newcommand\Lim{\varprojlim}
\newcommand\LLim{\varinjlim}
\DeclareMathOperator*{\holim}{holim}
\newcommand\ho{\mathrm{ho}\!}
\newcommand\into{\hookrightarrow}
\newcommand\onto{\twoheadrightarrow}
\newcommand\abuts{\implies}
\newcommand\xra{\xrightarrow}

\newcommand\pr{\mathrm{pr}}

\newcommand\HR{\mathrm{HR}}

\newcommand\bt{\bullet}
\newcommand\by{\times}

\newcommand\SL{\mathrm{SL}}

\newcommand\GL{\mathrm{GL}}

\newcommand\et{\acute{\mathrm{e}}\mathrm{t}}
\newcommand\Br{\mathrm{Br}}

\newcommand\nr{\mathrm{nr}}

\newcommand\Et{\acute{\mathrm{E}}\mathrm{t}}

\newcommand\Tot{\mathrm{Tot}\,}
\newcommand\diag{\mathrm{diag}\,}

\newcommand\pro{\mathrm{pro}}

\newcommand\pd{\partial}

\newcommand\ab{\mathrm{ab}}
\newcommand\cts{\mathrm{cts}}
\newcommand\gp{\mathrm{Gp}}
\newcommand\gpd{\mathrm{Gpd}}

\newcommand\red{\mathrm{red}}

\newcommand\Lie{\mathrm{Lie}}

\newcommand\cosk{\mathrm{cosk}}

\newcommand\op{\mathrm{opp}}

\newcommand\co{\colon\thinspace}

\newcommand\oR{\mathbf{R}}

\newcommand\uleft\underleftarrow
\newcommand\uline\underline
\newcommand\uright\underrightarrow

\newtheorem{theorem}{Theorem}[section]
\newtheorem{proposition}[theorem]{Proposition}
\newtheorem{corollary}[theorem]{Corollary}

\newtheorem{lemma}[theorem]{Lemma}
\newtheorem*{theorem*}{Theorem}
\newtheorem*{proposition*}{Proposition}
\newtheorem*{corollary*}{Corollary}
\newtheorem*{lemma*}{Lemma}
\newtheorem*{conjecture*}{Conjecture}

\theoremstyle{definition}
\newtheorem{definition}[theorem]{Definition}

\newtheorem*{definition*}{Definition}

\theoremstyle{remark}
\newtheorem{example}[theorem]{Example}
\newtheorem{examples}[theorem]{Examples}
\newtheorem{remark}[theorem]{Remark}

\newtheorem*{example*}{Example}
\newtheorem*{examples*}{Examples}
\newtheorem*{remark*}{Remark}
\newtheorem*{remarks*}{Remarks}
\newtheorem*{exercise*}{Exercise}
\newtheorem*{property*}{Property}
\newtheorem*{properties*}{Properties}

\begin{document}

\begin{abstract}
 We  reinterpret  Kim's non-abelian reciprocity maps for algebraic varieties as obstruction towers  of mapping spaces of \'etale homotopy types, removing technical hypotheses such as global basepoints and  cohomological constraints.
We then extend the theory by considering alternative natural series of extensions, one of which gives an obstruction tower whose first stage is the Brauer--Manin obstruction, allowing us to determine when Kim's maps recover the Brauer--Manin locus.
A tower based on  relative completions yields non-trivial reciprocity maps even for Shimura varieties; for the stacky modular curve, these 
 take values in Galois cohomology of modular forms, 
and 
give obstructions to an ad\'elic elliptic curve with global Tate module underlying a global elliptic curve. 
\end{abstract}

\title{Non-abelian reciprocity laws and higher Brauer--Manin obstructions}%{Non-abelian reciprocity laws for modular curves}%%{Non-abelian reciprocity laws beyond the anabelian setting}

\author{J. P. Pridham}
\thanks{The  author  was supported during this research by  the Engineering and Physical Sciences Research Council [grant number EP/I004130/2].}

%%msc 55Q05 sets of homotopy classes 55S35 obstruction theory 14F35 AG htpy th, $\pi_1$,  11D99 (Diophantine eqns, none of the above) 

\maketitle

\section*{Introduction}

In \cite{narec1}, Minhyong Kim  introduced a sequence of non-abelian reciprocity maps on the  ad\'elic points $X(\bA_F)$ of a variety $X$  over a number field $F$ equipped with a global point and satisfying certain cohomological conditions, with the global points contained within the kernel of all the maps. %; these thus refine the Hasse principle in an effective way. 
When $X=\bG_m$, this sequence just consists of a single map, the Artin reciprocity law 
\[
\rec\co \bA_F^{\by} \to G_F^{\ab}
\]
from the finite id\`eles of $F$ to the abelianisation of its Galois group, with the property that $\rec(F^{\by})=0$. 

In this paper, we give a topological construction of the non-abelian reciprocity maps, based on homotopical obstruction theory. These are defined under more  general hypotheses than those of \cite{narec1}. In particular, we do not need to assume existence of a global point in order to define the maps, so our reciprocity laws can be used to test the Hasse principle. For arbitrary varieties, the reciprocity maps exist as a tower of spaces over $ X(\bA_F)$, with the cohomological conditions of \cite{narec1} sufficing to ensure that the maps in the tower are injective.

Kim's non-abelian reciprocity laws are based on the lower central series of the geometric fundamental group, but other variants are possible with our approach. One variant produces a tower starting with the Brauer--Manin obstruction, allowing us to compare it with Kim's reciprocity laws. Another variant is based on relative completions, allowing us to study varieties whose geometric fundamental groups are perfect or nearly so.%, even though their lower central series are trivial or have small quotients.
%%quasi-perfect means perfect commutator. Presentation of $\SL_2(|Z)$ is $\<s,t~:~ s^3=t^2, \, t^4=1\>$. Actually, I'm worried that even $\PSL_2(\Z)$ isn't perfect, because surely $(C_2*C_3)^{\ab} \cong \Z/6$. On our abelianisation, gend by $2s,t$, as $s=3s-2s$ and $3s=2t$, so have $\Z/12$.

 For instance, the geometric fundamental group of the moduli stack $\cM_{1,1}$ of elliptic curves is the profinite completion $\widehat{\SL_2(\Z)}$ of $\SL_2(\Z)$. This has finite abelianisation, so  trivial pro-unipotent completion, which means  
the unipotent reciprocity maps of \cite{narec1} are identically zero. However, the Malcev completion of ${\SL_2(\Z)} $ relative to 
$\SL_2(\hat{\Z})$ (resp. $\SL_2(\Ql)$) is a pro-unipotent extension of $\SL_2(\hat{\Z})$  (resp. $\SL_2(\Ql)$) by a pro-unipotent group freely generated by duals of spaces of weight $2$  (resp. level $1$)
modular forms. 
Elements in Galois cohomology of these tensors then give non-trivial obstructions to an ad\'elic elliptic curve with global Tate module underlying a global elliptic curve. 

Our point of view is that the reciprocity maps of \cite{narec1} are obstruction towers in \'etale homotopy theory. The constructions of \cite{arma, fried} associate a pro-simplicial set $X_{\et}$ to any locally Noetherian simplicial scheme $X$. When $X$ is smooth and quasi-projective over a field 
$F$, with separable closure $\bar{F}$, 
\cite[Theorem 11.5]{fried} shows that for $\bar{X}:= X\ten_F\bar{F}$, the geometric homotopy type $(\bar{X})_{\et}$ is the homotopy fibre of $X_{\et}$ over $(\Spec F)_{\et}$, because the space $(\Spec \bar{F})_{\et}$ is contractible. Moreover, $(\bar{X})_{\et} $ is equivalent to the profinite completion of the homotopy type of the complex manifold $X(\Cx)$, for any embedding $F \into \Cx$, so $(\bar{X})_{\et} $ is a $K(\pi,1)$ whenever $X(\Cx)$ is so. 

We are interested in the simplicial set
\[
 \map_{(\Spec F)_{\et}}((\Spec F)_{\et}, X_{\et}),
\]
i.e. the mapping space (or function complex) of pro-simplicial sets over $(\Spec F)_{\et}$ (cf. Definition \ref{mapdef1}).
The space $(\Spec F)_{\et}$ is a $K(\pi,1)$, equivalent to the nerve $B G_F$ of the Galois group $G_F$. Since morphisms of schemes give rise to morphisms of \'etale homotopy types, there is then a natural map
\[
 X(F) \to \map_{BG_F}(BG_F, X_{\et}).
\]

When $X$ is a $K(\pi,1)$ (such as any hyperbolic curve, surface of general type, or abelian variety) over $F$, we have (ignoring issues with basepoints)
\[
 \pi_i\map_{BG_F}(BG_F, X_{\et})= \left\{\begin{matrix}
                                                                \H^{1-i}(F, \pi_1^{\et}(\bar{X})) & i \le 1\\
								 0 & i \ge 2.     
                                                               \end{matrix}\right.
\]
For smooth varieties $X$, $\pi_1^{\et}(\bar{X})$ will always be of strictly negative weights, so $\H^0(F, \pi_1^{\et}(\bar{X}))=0$, and we have
\[
 \map_{BG_F}(BG_F, X_{\et}) \simeq \H^{1}(F, \pi_1^{\et}(\bar{X})),
\]
a discrete set of points. This  non-abelian cohomology set is the main focus of \cite{narec1}, and for hyperbolic curves $X$, Grothendieck's section conjecture amounts to the prediction that the morphism
\[
 X(F) \to \map_{BG_F}(BG_F, X_{\et})
\]
is an equivalence.

%In Section \ref{obsthsn} 
In this paper, 
we  construct the reciprocity maps using obstruction theory analogous to \cite{bousfieldHtpySpectralObs}. The idea is to identify towers $\{X_{\et}(n)\}_n$ of quotients of $X_{\et}$ over $BG_F$ for which there exist non-abelian spectral sequences converging to $\map_{BG_F}(BG_F, X_{\et}(\infty))$, where $X_{\et}(\infty):= \ho\Lim_n X_{\et}(n)$. The crucial property making these spectral sequences special is that they incorporate fibre sequences
\[
 \pi_0\map_{BG_F}(BG_F, X_{\et}(n))\to \pi_0\map_{BG_F}(BG_F, X_{\et}(n-1)) \xra{\ob_n} \Ob_n
\]
giving obstructions to lifting homotopy classes of maps. 

We can also take more general spaces as the source, considering a profinite homotopy type  $BG_{ \bA_F^{\in \Sigma}}$ associated to the ad\`ele ring $\bA_F^{\in \Sigma}= \prod_{v \in \Sigma}' F_v$, for  a (possibly infinite) non-empty set $\Sigma$ of finite places. Reciprocity maps then arise in non-abelian spectral sequences converging to the homotopy groups of
\[
 X(\bA_F^{\in \Sigma} )\by^h_{ \map_{BG_F}(BG_{\bA_F^{\in \Sigma}}, X_{\et}(\infty))}\map_{BG_F}(BG_F , X_{\et}(\infty)),
\]
and
he spaces in the spectral sequence are compactly supported cohomology groups $\H_c^*(\sO_{F,\Sigma},-)$, which can be rewritten as duals of Galois cohomology groups by Poitou--Tate duality. Defining the tower $\{X_{\et}(n)\}_n$ in terms of the lower central series of the geometric fundamental group $\pi_1^{\et}(\bar{X})$ recovers Kim's reciprocity  maps \cite{narec1}. Subtler towers based on  relative completions give rise to reciprocity laws in more general situations. %, including for modular curves.

%I think we might want to cut this para, since first step should be effective. No, first step should only suffice for weight $2$, all levels.
Explicitly, for a modular curve $Y_{\Gamma}$ we can consider the set  $ Y_{\Gamma}(\bA_F^{\in \Sigma})_0$ of ad\'elic points $x$ for which the Tate module $T_{\ell}E_{\bar{x}}$ of the associated elliptic curve lifts to a  $G_{F,\Sigma}$-representation $\Lambda$. We then construct a sequence of subsets (glossing over subtleties related to potential higher automorphisms for now)
\[
       \ldots \to  Y_{\Gamma}(\bA_F^{\in \Sigma})_1 \to Y_{\Gamma}(\bA_F^{\in \Sigma})_0
\]
containing $Y_{\Gamma}(\sO_{F,\Sigma})$. These are defined inductively by $ Y_{\Gamma}(\bA_F^{\in \Sigma})_{n}= \ob_n^{-1}(0)$, for reciprocity maps
\[
  \ob_n \co  Y_{\Gamma}(\bA_F^{\in \Sigma})_{n-1} \to  \H^2_c( G_{F,\Sigma}, T_n),
\]
where the $\Ql$-vector  spaces $T_n$  are given by homogeneous factors of a Lie algebra  generated by 
\[
 T_1= \prod_m \H^1(\Gamma, V_m)^*\ten V_m,
\]
for irreducible $\SL_2$-representations $V_m$; via Eichler--Shimura, the groups $\H^1(\Gamma, V_m)$ can be interpreted as $\ell$-adic realisations  of motivic modular forms of weight $m+2$ and level $\Gamma$. If we instead assume that the Tate modules $T_pE_{\bar{x}}$  lift to $G_{F,\Sigma}$-representations $\rho_p$ for all primes $p$, then we have a similar sequence, but with $T_1$ now defined in terms of modular forms of all levels. In this case, \cite{HelmVoloch} shows that whenever there is an ad\'elic elliptic curve compatible with the representations $\rho_p$, there must exist a rational elliptic curve giving rise to them, but our obstructions should measure the difference between these elliptic curves.

We can even incorporate higher homotopical information in constructing reciprocity laws for Deligne--Mumford stacks $X$, by looking at completions of \'etale homotopy types   instead of  their fundamental groups. The first obstruction map in the spectral sequence is then just  the Brauer--Manin obstruction when we take the base $X_{\et}(0)$ of  the tower to be $BG_F$, with  refinements for (pro-)\'etale covers given by the subtler obstruction towers. 

The structure of the paper is as follows. Section \ref{obsthsn} lays the topological foundations for constructing  reciprocity laws, developing generalisations of Bousfield's obstruction theory \cite{bousfieldHtpySpectralObs}. The most general statement is Proposition \ref{abcofibregpd}, giving obstruction spaces for homotopy limits of abelian extensions of simplicial groupoids.

Section \ref{rationalsn} then applies this theory to give towers of obstructions to the existence of global points over a number field. The first such tower we consider  is Example \ref{nilpex1}. Writing $ \Pi_n:= \pi_1^{\et}(X,\bar{x})/[\pi_1^{\et}(\bar{X}, \bar{x})]_{n+1}$, $\bar{\pi}:= \pi_1^{\et}(\bar{X}, \bar{x})$, and
  $[\pi]_1:=\pi$, 

with $[\pi]_{k+1}$  the closure of  $[\pi, [\pi]_k]$, this gives a non-abelian spectral sequence 
\[
 E_1^{s,t}= \H^{1+s-t}(G_{F,\Sigma}, [\bar{\pi}]_s/[\bar{\pi}]_{s+1} ) \abuts  \pi_{t-s} \map_{BG_{F,\Sigma}}(BG_{F,\Sigma}, B\Pi_{\infty} ),
\] 
encoding Ellenberg's obstructions. 
There is a unipotent generalisation Example \ref{unipex1}, and further refinements for relative completion. Notably, Examples \ref{modularex2} and \ref{modularex2b} give obstructions, in terms of modular forms,
 to lifting a $G_F$-representation $\Lambda$ to an elliptic curve $E$ over $F$ with Tate module $\Lambda$.

In Section \ref{recsn}, this approach is refined to consider the difference between the obstruction towers for $F$ and $\bA_F^{\Sigma}$, yielding reciprocity laws in terms of Poitou--Tate duality. The main examples of resulting spectral sequences appear in \S \ref{recexamplessn}, with Examples \ref{nilpex3} and \ref{unipex3} recovering and generalising Kim's non-abelian reciprocity laws \cite{narec1}, while reciprocity laws for the stacky modular curve $\cM_{1,1}$ appear in Example \ref{modularex4}, giving obstructions to an ad\'elic elliptic curve being defined over $F$ when its Tate module is known to be a $G_F$-representation. 

Constructions in terms of higher homotopy types are then given in  Example \ref{exhtpy3}, with  
%Examples \ref{exhtpyBrauerManin} and \ref{exhtpyEtBrauerManin} 
\S \ref{BMsn}
showing how the spectral sequences for higher homotopy types start with the Brauer--Manin obstruction (or a pro-\'etale generalisation) as the first stage in the tower. Proposition \ref{MKBM} gives a sufficient condition for Kim's non-abelian reciprocity laws to recover the Brauer--Manin set. 
In \S \ref{alternative}, we then discuss more concrete ways to construct the reciprocity laws, with a fairly explicit description of the first obstruction for modular curves, and a discussion of the relation between higher Brauer--Manin obstructions and Massey products.

Appendix \ref{adelehtpy} contains the  technicalities needed to work with  higher \'etale homotopical 
invariants of ad\`ele rings, giving a morphism from  $(\Spec \bA_F^{\in \Sigma})_{\et}$ to the homotopy type  $BG_{ \bA_F^{\in \Sigma}}$ governing restricted products of local cohomology groups.

Readers unfamiliar with abstract homotopy theory are advised to skip \S \ref{obsthsn} entirely, starting with \S \ref{alternative} for an overview before reading the examples in \S\S \ref{rationalsn}, \ref{recsn}. We should warn at this stage that none of the examples exhibits explicit classes in Galois cohomology on which to evaluate the obstructions, but the weights of the Galois representations involved suggests they must exist in great generality. 

I would like to thank Minhyong Kim for many helpful discussions, Felipe Voloch for alerting me to relevant references,  Ambrus P\'al for a helpful observation, and the anonymous referee for catching several errors and suggesting improvements. 

\subsection*{Notation}

We will write $\cong$ for isomorphism and $\simeq$ for weak equivalence. 
Let $\bS$ denote the category of simplicial sets with the Kan model structure, and $s\bS$ the category of bisimplicial sets. We denote mapping spaces in model categories by $\map$; in the case of simplicial model categories, these simplicial sets are just given by derived functors of the simplicially enriched $\Hom$ bifunctor, and in general they are given by the function complexes of \cite[\S 5.4]{hovey}.

We fix a number field $F$, and a (possibly infinite) non-empty set $\Sigma$ of finite places of $F$. Then $G_F$ denotes the Galois group of $F$, and its quotient $G_{F,\Sigma}$ is the Galois group of the maximal extension of $\sO_F$ unramified outside $\Sigma$. We  write $\bA_F^{\in \Sigma}$ for the ad\`ele ring
\begin{align*}
 \bA_F^{\in \Sigma}:= \prod'_{v \in \Sigma} F_v 
= \LLim_{\substack{T \subset \Sigma \\ T \text{finite}}} (\prod_{v \in T} F_v \by \prod_{v \in \Sigma -T} \sO_{F,v}).
\end{align*} 
% and
% \[
%  \bA_{F,\Sigma}:= \bA_F^{\in \Sigma} \by \prod_{v \notin \sigma} \sO_{F,v}. %%probably cut this as unnecessary.
% \]

\tableofcontents

\section{Obstruction theory from abelian extensions}\label{obsthsn}

Given a fibration $f\co X \to Y$ of spaces with fibre $Z$, there is a long exact sequence 
\[
 \ldots \to \pi_2Z \to \pi_2X \to \pi_2Y \to  \pi_1Z \to \pi_1X \to \pi_1Y \to \pi_0Z \to \pi_0X \to \pi_0Y
\]
of homotopy groups and sets, where the final map need not be surjective (and at this stage we are being deliberately vague about basepoints).

Our primary goal in this section is to look for cases where this sequence extends one stage further, giving an  obstruction map from  $\pi_0Y$ to some pointed set %$\Ob_f$, 
such that the fibre %of $\pi_0Y \to \Ob_f$ 
over the basepoint is the image of $\pi_0X \to \pi_0Y$.  This will happen if there is some  space $B$ and a map $Y \to B$ in the homotopy category of spaces, with $X$ the homotopy fibre over a point $b \in B$, and in this case $Z$ above  is automatically the loop space $\Omega(B,b)$.

An obvious example of this phenomenon is when $X$ is a principal $G$-bundle over $Y$ for a topological group $G$, so arises as the homotopy fibre of a map $Y \to BG$. We then have a long exact sequence
\[
 \ldots \to \pi_1Y \to \pi_1BG \to \pi_0X \to \pi_0Y \to \pi_0BG,
\]
noting that $\pi_nBG=\pi_{n-1}G$. 

In this form, this statement is telling us nothing new, since $\pi_0X \to \pi_0Y$ is automatically surjective in such cases. However the characterisation of $X$ as a homotopy fibre also passes to homotopy limits of such diagrams. Given a small category $I$, together with $I$-diagrams $Y$ and $G$ in simplicial sets and simplicial groups, and a  principal $G$-bundle $X$ over $Y$, we can characterise $X$ as the homotopy fibre of a map $Y \to BG$ in the homotopy category, and then 
\[
 \ho\Lim_{i \in I} X(i) \to \ho\Lim_{i \in I} Y(i)\to \ho\Lim_{i \in I} BG(i).
\]
is a homotopy fibre sequence, so gives rise to  a long exact sequence of homotopy groups and sets of the desired form; this is essentially the content of Corollary \ref{obs2} below.

\subsection{Central and abelian extensions of simplicial groups}\label{sgpsn}

\subsubsection{Central extensions}

We now look at principal fibrations in the category of groups. First observe that an internal group object in the category of groups is an abelian group $A$ by the Eckmann--Hilton argument, with multiplication  $A \by A \xra{\cdot} A$ being a group homomorphism.

An $A$-space in groups is then a group $G$ equipped with a group homomorphism $\mu \co A \by G \to G$ such that
the diagram
\[
 \begin{CD}
  A \by A \by G @>{(\id_A, \mu)}>> A \by G \\
@V{(\cdot , \id_G)}VV  @VV{\mu}V\\
A \by G @>{\mu}>> G
\end{CD}
\]
commutes. In other words, $\mu(a,g)= \rho(a)g$, for the group homomorphism $\rho\co A \to Z(G)$ to the centre of $G$ given by $\rho(a)= \mu(a,1)$. The $A$-action is faithful if  $\rho$ is injective, and then $G$ is a principal $A$-space over $G/A$. 

Applying the nerve functor, we have a simplicial abelian group $BA$ (the group homomorphism $A \by A \xra{\cdot} A$ inducing a multiplication $BA \by BA \to BA$, and for every   principal $A$-space $G$ in groups over $H$, we get a principal $BA$-fibration $BG$ over $BH$.

\begin{definition}\label{barwdef}
Define  $\nabla:s\bS \to \bS$ to be the right adjoint to  Illusie's total $\Dec$ functor given by $\DEC(X)_{mn}= X_{m+n+1}$. Explicitly,
\[
 \nabla_p(X) = \{(x_0, x_1, \ldots, x_p) \in
 \prod^p_{i=0} X_{i,p-i} | \pd^v_0 x_i = \pd^h_{i+1}x_{i+1},\, \forall 0 \le i <p\}
\]
with operations
\begin{eqnarray*}
 \pd_i(x_0, \ldots, x_p) &=& (\pd^v_i x_0, \pd^v_{i-1}x_1, \ldots , \pd^v_1 x_{i-1}, \pd^h_ix_{i+1}, \pd^h_i x_{i+2}, \ldots, \pd^h_i x_p),\\
\sigma_i(x_0, \ldots, x_p) &=& (\sigma^v_i x_0,\sigma^v_{i-1}x_1, \ldots , \sigma^v_0 x_i, \sigma^h_i x_i, \sigma^h_i x_{i+1}, \ldots, \sigma^h_i x_p).
\end{eqnarray*}

Given a simplicial diagram $\Gamma$ of groupoids, the nerve $B\Gamma$ is a bisimplicial set, and we write $\bar{W}\Gamma:= \nabla B\Gamma$, noting that this agrees with the definition of \cite[\S V.7]{sht} when $\Gamma$ has constant objects. 
\end{definition}

Note that the loop space $\Omega\bar{W}G$  of $\bar{W}G$ is weakly equivalent to $G$, so in particular $\pi_i \bar{W}G \cong \pi_{i-1}G$, with $\pi_0G=*$. 

In \cite{CRdiag},  it is established that the canonical natural transformation
\[
\diag X \to \nabla X
\]
from the diagonal is a weak equivalence for all $X$. 
Thus $\nabla X$ is a model for the homotopy colimit
\[
 \holim_{\substack{\lra \\n \in \Delta^{\op}}} X_n,
\]
and in particular $\bar{W}\Gamma$ a model for $\ho\LLim_{n\in \Delta^{\op}} B(\Gamma_n)$.

\begin{proposition}\label{centralcofibre}
 Given a surjection $G \to H$ of simplicial groups with central kernel $A$, there is a simplicial set $Y'$ weakly equivalent to $\bar{W}H$ and a map $f\co Y' \to \bar{W}^2A$ with fibre $\bar{W}G$, which is also the homotopy fibre. Moreover, the space $Y'$ and weak equivalence $w\co Y' \to \bar{W}H$ can be chosen functorially.   
\end{proposition}
\begin{proof}
Writing $K=\bar{W}A$, the statement is essentially the well-known result (\cite[Theorem V.3.9]{sht}) that $\bar{W}K$ classifies principal fibrations.
The reasoning above applied to simplicial groups gives us a bisimplicial abelian group $BA$ and a principal $BA$-fibration $BG$ over $BH$. Applying the codiagonal functor $\nabla$ then gives us a simplicial abelian group $\bar{W}A$ and a principal $\bar{W}A$-fibration $\bar{W}G$ over $\bar{W}H$. The map $f$ then just comes by taking the homotopy quotient of $\bar{W}G \to \bar{W}H$ by the action of $\bar{W}A$.

Explicitly, we  set $Y'= \bar{W}[ \bar{W}G/\bar{W}A]$, for the simplicial groupoid $[\bar{W}G/\bar{W}A]$ with objects $\bar{W}G$ and morphisms given by $\bar{W}A$ acting on the right. Applying $\bar{W}$ twice to the map $[G/A] \to [H/1]$ of groupoids in groups gives the weak equivalence $Y' \to \bar{W}H$, since $\bar{W}[Y/1]= Y$ and the fibre $\bar{W}[A/A]$ is contractible. Similarly, the Kan fibration $Y' \to \bar{W}^2A$ comes from the map $[G/A] \to [1/A]$ of  groupoids in groups.
\end{proof}

\subsubsection{Abelian extensions}

More generally, given a group $H$, a group object $\Gamma$ in the comma category $\Gp \da H$ of groups over $H$ is of the form $\Gamma= H \ltimes A$, for an abelian group $A$ equipped with an $H$-action.

Then a $\Gamma$-space in groups over $H$ consists of a group $G$ and a surjection $G \to H$ together with an associative action $\Gamma\by_HG \to G$ (all maps being group homomorphisms). Equivalently, for the group $A$ above, we have a group homomorphism $G \ltimes A \to G$ over $H$, hence a $G$-equivariant map $A \to \ker(G\to H)$. 

The condition for $G$ to be a principal $\Gamma$-space is then just that the map $A \to \ker(G \to H)$ be an isomorphism. In other words, a pair $(\Gamma,G)$ is the same as an abelian group $A$ equipped with an $H$-action together with a surjective group homomorphism $G \to H$ with kernel $A$.

Given such a $G$, we can take the nerve, giving a surjective fibration $BG \to BH$ of simplicial sets with fibre $BA$ over the unique vertex of $BH$. The simplicial set $B(H \ltimes A)$ is a group object in  simplicial sets over  $BH$, and $BG$ is a principal $ B(H \ltimes A)$-bundle.  

\begin{proposition}\label{abcofibre}
 Take a surjection $G \to H$ of simplicial groups with abelian kernel $A$. Then there exists a fibration 
\[
 Y'\to \bar{W}[ H \ltimes \bar{W}A]
\]
for which the projection  $Y' \to \bar{W}H$ is a weak equivalence,  with
\[
 Y'\by_{ \bar{W}[H \ltimes\bar{W} A ]} \bar{W}H \cong \bar{W}G.
\]
Moreover, the space $Y'$ and weak equivalence $w\co Y' \to Y$ can be chosen functorially.   
 \end{proposition}
\begin{proof}
We adapt the proof of Proposition \ref{centralcofibre}. Set $Y'= \bar{W}[\bar{W}(G \ltimes A) \Rightarrow \bar{W}G]$, for the simplicial groupoid $[\bar{W}(G \ltimes A) \Rightarrow \bar{W}G]$ with objects $\bar{W}G$ and morphisms  $\bar{W}(G \ltimes A)$. Applying $\bar{W}$ twice to the map $[(G \ltimes A) \Rightarrow G] \to [H \Rightarrow H]$ of groupoids in groups gives the weak equivalence $Y' \to \bar{W}H$, since $\bar{W}[Y \Rightarrow Y]= Y$ and the fibre $\bar{W}[A/A]=\bar{W}[A\ltimes A \Rightarrow A]$ is contractible. Similarly, the Kan fibration $Y' \to \bar{W}[ H \ltimes \bar{W}A]$ comes from the map $[(G \ltimes A) \Rightarrow G] \to [(H \ltimes A)\Rightarrow H]$ of  groupoids in groups.
\end{proof}

\subsubsection{Groupoids}

The constructions above generalise to groupoids, and we will not concern ourselves with the full generality of internal groups in groupoids. We just observe that any abelian group is \emph{a fortiori} an internal group in groupoids with one object, and that for any groupoid $H$, an $H$-representation $A$ in abelian groups has associated  groupoid $H \ltimes A$, which is a group object in groupoids over $H$.

\begin{definition}
Say that a morphism $f\co G \to H$ is an abelian extension if it is an isomorphism on objects, surjective on morphisms, and the groups $A(x):=\ker (f\co G(x,x) \to H(fx,fx))$ are abelian for all objects $x$ of $G$. %We say that $f$ is a central extension if moreover $A(x)$ is central in $G(x,x)$ for all $x$.
\end{definition}

Thus for any abelian extension $G \to H$ of groupoids with kernel $A$, we get a surjective fibration $BG \to BH$ of simplicial sets, and the fibre over $fx \in (BH)_0$ is just $A(x)$. Moreover, $B(H \ltimes A)$ is a group object in  simplicial sets over  $BH$, and $BG$ is a principal $ B(H \ltimes A)$-bundle.  

\begin{proposition}\label{abcofibregpd}
 Given an abelian extension $G \to H$ of simplicial groupoids with abelian kernel $A$, there is  fibration $Y'\to \bar{W}[H \ltimes A]$ such that the projection $Y' \to \bar{W}H$ is a weak equivalence, with
\[
 Y'\by_{ \bar{W}[H \ltimes A ]} \bar{W}H \cong \bar{W}G.
\]
Moreover, the space $Y'$ and weak equivalence $w\co Y' \to Y$ can be chosen functorially.   
 \end{proposition}
\begin{proof}
 The proof of Proposition \ref{abcofibre} carries over, replacing groupoids in groups with groupoids in groupoids.
\end{proof}

%Group object in groupoids consists of a group $G$ of objects and stuff. See Ronnie Brown and Spencer, ``G-groupoids''

\subsection{Passage to homotopy limits}

For a small category $I$, we have a limit functor $\Lim_I \co \bS^I \to \bS$ from $I$-diagrams of simplicial sets to simplicial sets. Recall from \cite[\S VIII.2]{sht} or \cite[Ch. 18]{hirschhorn} that $\ho\Lim_I \co \Ho(\bS^I) \to \Ho(\bS)$ is the right-derived functor of $\Lim_I$; in other words, it is the universal functor under $\Lim_I$ preserving weak equivalences.

\begin{definition}
 Given a small category $I$ and simplicial group-valued functors $G,H \co I\to s\gp$, we say that a natural transformation $G \to H$ is a central (resp. abelian) extension if the maps $G(i) \to H(i)$ are so, for all $i \in I$.
\end{definition}

\begin{proposition}\label{centralholim}
Given a central extension $f\co G \to H$ of $I$-diagrams with kernel $A$,  there is a morphism $\ho\Lim_{i \in I} \bar{W}H(i)  \to \ho\Lim_{i \in I}\bar{W}^2A(i)$ in the homotopy category of simplicial sets with homotopy fibre $\ho\Lim_{i \in I} \bar{W}G(i)$ over the distinguished point $*$.
\end{proposition}
\begin{proof}
We just apply the derived functor $\ho\Lim_{i \in I}$ to the diagrams from Proposition  \ref{centralcofibre}.
\end{proof}
Note that when $I= \Delta$, the simplex category, this recovers a fairly general case of  Bousfield's obstruction maps from  \cite{bousfieldHtpySpectralObs}.

\begin{corollary}\label{obs1}
In the scenario of Proposition \ref{centralholim}, there is a  sequence 
\[
\pi_0 \ho\Lim_{i \in I} \bar{W}G(i) \xra{f_*} \pi_0\ho\Lim_{i \in I} \bar{W}H(i)  \xra{\delta_*} \pi_0\ho\Lim_{i \in I} \bar{W}^2A(i)
\]  
of sets, exact in the sense that the fibre of $\delta_*$ over $0$ is the image of $f_*$.

Moreover,  there is a group action of $\pi_0\ho\Lim_{i \in I} \bar{W}A(i) $ on $\pi_0 \ho\Lim_{i \in I} \bar{W}G(i)$ whose orbits are precisely the fibres of $f_*$. 

For any $x \in \ho\Lim_{i \in I} \bar{W}G(i)$, with $y=f_*x$, the homotopy fibre of $f$ over $y$ is weakly equivalent  to $\ho\Lim_{i \in I} \bar{W}A(i) $,
and the sequence above 
extends to a long exact sequence
\[
\xymatrix@R=0ex{
\cdots  \ar[r]^-{f_*}&\pi_n(\ho\Lim_{i \in I} \bar{W}H(i) ,y) \ar[r]^-{\delta}&  \pi_{n-1}\ho\Lim_{i \in I} \bar{W}A(i) \ar[r] &\pi_{n-1}( \ho\Lim_{i \in I} X(i) ,x)\ar[r]^-{f_*}&\cdots\\ 
\cdots \ar[r]^-{f_*}&\pi_1( \ho\Lim_{i \in I} \bar{W}H(i),y) \ar[r]^-{\delta}&  \pi_0\ho\Lim_{i \in I} \bar{W}A(i)  \ar[r]^-{-*x} &\pi_0 \ho\Lim_{i \in I} \bar{W}G(i).
}
\]
\end{corollary}
\begin{proof}
This is just the long exact sequence of a fibration  (\cite[Lemma I.7.3]{sht}) applied to $\delta\co \ho\Lim_{i \in I} Y'(i) \to \ho\Lim_{i \in I} \bar{W}^2A(i) $, and noting that 
\[
 \Omega \ho\Lim_{i \in I} \bar{W}^2A(i)\simeq  \ho\Lim_{i \in I} \Omega \bar{W}^2A(i)\simeq \ho\Lim_{i \in I} \bar{W}A(i),
\]
so 
\[
 \pi_n\ho\Lim_{i \in I} \bar{W}^2A(i)\cong  \pi_{n-1}\ho\Lim_{i \in I} \bar{W}A(i)
\]
for all $i>0$.
\end{proof}

\begin{remark}
 Were it not for the final term, Corollary \ref{obs1} would just be  the long exact sequence of homotopy for the map $\ho\Lim_{i \in I} \bar{W}G(i) \to \ho\Lim_{i \in I} \bar{W}H(i)$. The essential purpose of all our effort so far has thus been to incorporate the extra term $\pi_0\ho\Lim_{i \in I} \bar{W}^2A(i)$, giving an obstruction to lifting connected components.
\end{remark}

\begin{proposition}\label{abholim}
Given an abelian extension $G \to H$ of $I$-diagrams with kernel $A$,  there is a morphism $\delta\co \ho\Lim_{i \in I} \bar{W}H(i)  \to \ho\Lim_{i \in I}\bar{W}(H \ltimes \bar{W}A(i))$ in the homotopy category of simplicial sets over $\ho\Lim_{i \in I} \bar{W}H(i)$ with 
a homotopy pullback diagram
\[
\begin{CD}
 \ho\Lim_i  \bar{W}G(i)  @>>> \ho\Lim_i \bar{W}H(i) \\
@VVV @VV{\delta}V\\  
\ho\Lim_i \bar{W}H(i) @>0>>\ho\Lim_i \bar{W}(H(i) \ltimes \bar{W}A(i)). 
\end{CD}
\]

In particular, if the adjoint action of $H$ on $A$ factors through some quotient $\bar{H}$, then for any $\bar{y} \in \ho\Lim_i \bar{W}\bar{H}$, we have a fibration sequence
\[
 (\ho\Lim_i  \bar{W}G(i))_{\bar{y}} \to (\ho\Lim_i \bar{W}H(i))_{\bar{y}} \to \ho\Lim_i \bar{W}(\bar{H}(i) \ltimes \bar{W}A(i))_{\bar{y}}
\]
on homotopy fibres over $y$.
\end{proposition}
\begin{proof}
We just apply the derived functor $\ho\Lim_{i \in I}$ to the diagrams from Proposition  \ref{abcofibre}.
\end{proof}

Now, given an $I$-diagram $X$, write $\uleft{X}:= \ho\Lim_i  X(i)$. 

\begin{corollary}\label{obs2}
In the scenario of Proposition \ref{abholim},
an element $y$ lies in the image of
\[
\pi_0 \uleft{\bar{W}G} \xra{f_*} \pi_0\uleft{\bar{W}H}
\]  
if and only if $\delta_*(y)= 0 \in  \pi_0(\uleft{\bar{W}A}_{(y)} )$, where $\uleft{\bar{W}A}_{(y)}$ denotes the homotopy fibre of $ \uleft{\bar{W}(H \ltimes \bar{W}A)} \to \uleft{\bar{W}H}$ over $y$.

Moreover, for each such $y$  there is a transitive group action of $\pi_0(\uleft{\bar{W}A}_{(y)} )$ on the fibre of $f_*$.

For any $x \in  \uleft{\bar{W}G}$ with $y=f_*x$, the homotopy fibre of $f$ over $y$ is weakly equivalent  to $\uleft{\bar{W}A}_{(y)}$,
and the sequence above 
extends to a long exact sequence
\[
\xymatrix@R=0ex{
\cdots  \ar[r]^-{f_*}&\pi_n( \uleft{\bar{W}H} ,y) \ar[r]^-{\delta}&  \pi_{n-1}(\uleft{\bar{W}A}_{(y)}) \ar[r] &\pi_{n-1}( \uleft{\bar{W}G} ,x)\ar[r]^-{f_*}&\cdots\\ 
\cdots \ar[r]^-{f_*}&\pi_1( \uleft{\bar{W}H},y) \ar[r]^-{\delta}&  \pi_0(\uleft{\bar{W}A}_{(y)}) \ar[r]^-{-*x} &\pi_0 \uleft{\bar{W}G}.
}
\]
\end{corollary}
\begin{proof}
The proof of Corollary \ref{obs1} carries over. %%; note that $ \uleft{A}_{(y)}$ is the loop space of $ \uleft{ \bar{W}(H \ltimes A)}_y$. %%not defined!
\end{proof}

\section{Towers of Diophantine obstructions}\label{rationalsn}
 
Recall that we are fixing a number field $F$, and a (possibly infinite) non-empty set $\Sigma$ of finite places of $F$. 
When $\Sigma$ consists of all finite places, we %may regard these as mapping spaces over $(\Spec F)_{\et}$ via the 
have a weak equivalence $B G_{F} \simeq (\Spec F)_{\et}$; in general,  \cite[Appendix  A, Equation (1)]{cesnaviciusPT} combines with \cite[Corollary 6.5]{fried} to show  that the homotopy fibre of the surjective map $(\Spec \sO_{F,\Sigma})_{\et}\to B G_{F,\Sigma} $ becomes contractible on derived pro-$\Sigma$ completion,  % and  \cite[Proposition 1.29]{weiln} 
i.e. pro-$L$ completion (cf. \cite[Theorems 3.4 and 4.3]{arma}) for $L$ the set of integer primes  all of whose $F$-prime factors lie in  $\Sigma$.

% 
% \begin{definition}
% Given a set $L$ of primes, we say that an $L$-group is a finite group $G$ for which only primes in $L$ divide its order. Given a pro-group $G= \{G_i\}$, we define the pro-$L$-group
% $G^{\wedge_L}$ by  requiring that $G^{\wedge_L}$ be the system consisting of  all $L$-group quotients of each of the $G_i$. 
% \end{definition}
% We write $|\Sigma|$ for the set of primes under $\Sigma$, with  $G^{\wedge_{|\Sigma|}}$ defined accordingly.

Given any profinite group $\Pi$ and a pro-surjection $\Pi \to G_{F,\Sigma}$ (such as when $\Pi$ is the arithmetic fundamental group of an $\sO_{F,\Sigma}$-scheme), we have a fibration $B \Pi \to B G_{F,\Sigma}$ of pro-simplicial sets in the model structure of \cite{isaksen}. 
%Given any profinite group $\Pi$ and a surjection $\Pi \to G_F$ (such as when $\Pi$ is the arithmetic fundamental group of an $F$-scheme), we have a fibration $B \Gamma \to B G_F$ of pro-simplicial sets. 

Thus for any pro-simplicial set $Y$ over $BG_{F,\Sigma}$, we may 
consider the mapping space
\[
 \map_{B G_{F,\Sigma}}(Y, B\Pi)
\]
for the same model structure;
 when $Y= BG_{F,\Sigma}$, this is the space of sections of $ B\Pi \to BG_{F,\Sigma}$. 

Explicitly, \cite[Proposition 10.9]{isaksen} allows us to describe mapping spaces of pro-simplicial sets in terms of  the Edwards--Hastings  strict model structure \cite{EdwardsHastings}, reducing to the following description.
\begin{definition}\label{mapdef1}
 For pro-simplicial sets $X=\Lim_i X(i)$, $Y=\Lim_j Y(j)$ such that each $Y(j)$ has finitely many non-zero homotopy groups, we may define the simplicial set  $ \map(X, Y)$  in terms of  mapping spaces of simplicial sets as the homotopy limit
\[
 \map(X, Y):= \ho\Lim_j \LLim_i \map(X(i),Y(j)).
\]

For general pro-simplicial sets $X=\Lim_i X(i)$, $Y=\Lim_j Y(j)$, we may define $ \map(X, Y)$ as the homotopy limit
\[
 \map(X, Y):= \ho\Lim_{j,k} \LLim_i \map(X(i),P_kY(j)),
\]
where $P_k$ denotes a Postnikov tower.

For a diagram $X \xra{f} Z \la Y$ of pro-simplicial sets, the relative mapping space $\map_{Z}(X,Y)$ is the homotopy fibre of  $\map(X,Y) \to \map(X,Z)$ over $f$. 
\end{definition}

\subsection{Abelian extensions}\label{abext}
Assume that we have abelian extension $\Pi''\to \Pi'$ of profinite groups with kernel $A$, such that the conjugation action of $\Pi'$ on $A$ factors through some quotient $G$ of $\Pi'$. When working with nilpotent completions of geometric fundamental groups, we may take $G= G_{F,\Sigma}$, but for relative completions (as needed for modular curves), $G$ will be larger.

Writing $  B(G \ltimes BA):= \bar{W}(G \ltimes BA)$, we have:
\begin{proposition}\label{abprop}
 In the scenario above, and for any pro-simplicial set $Y$ over $BG$, there is a natural fibration sequence
\[
 \map_{B G}(Y,B\Pi'') \to  \map_{B G}(Y,B\Pi') \to \map_{BG }(Y, B(G \ltimes BA))
\]
of mapping spaces,  the fibre being taken over the zero map $Y \to BG \to B(G \ltimes BA)$. 
\end{proposition}
\begin{proof}
The idea behind this statement is that  the  extension  $\Pi''\to \Pi'$  defines an element of $\H^1(\Pi',A)$, which we can write as a morphism $\ob\co \Pi' \to G \ltimes BA$ in the homotopy category of simplicial profinite groups over $ G$. The proof of \cite[Proposition 1.19]{ddt1} adapts to any Artinian category, and in particular to finite groups, allowing us to regard  simplicial profinite groups as pro-objects in the category of  (bounded) finite simplicial groups. 
We can then recover $B\Pi''$ as a homotopy fibre product 
\[
 B\Pi' \by^h_{\ob, B(G \ltimes BA)}B\Pi',
\]
leading to the fibration sequence above. 

More formally,  we write $\Pi''= \Lim_{j\in J} \Pi''(j)$ as a filtered limit of finite quotient groups, inducing compatible expressions $ A= \Lim_j A(j)$,  $\Pi'(j)= \Pi''(j)/A(j)$ and $\Pi'(j) \onto G(j)$ with $G= \Lim_j G(j)$.

The mapping spaces   $\map(Y,B\Pi)$ are given by 
\[
 \map_{B G}(Y,B\Pi)\simeq \ho\Lim_{(n,j) \in \Delta \by J} \Hom_{\pro(\Set)}(Y_n, B\Pi(j)),
\]
so we apply  Proposition \ref{abholim} to the abelian extension
\[
  \Hom_{\pro(\Set)}(Y_n, \Pi''(j)) \to  \Hom_{\pro(\Set)}(Y_n, \Pi'(j))
\]
of $(\Delta \by J)$-diagrams in groups, and then take homotopy fibres over the canonical basepoint of  $\map(Y,BG)$.
\end{proof}

We think of the base $\map_{B G}(Y, B(G \ltimes BA))$ of the fibration as an obstruction space; via the description of Definition \ref{mapdef1}, its  homotopy groups   are given by equivariant cohomology groups
\[
 \pi_i\map_{B G}(Y, B(G \ltimes BA)) \cong \H^{2-i}_{G}(Y, A),
\]
so we have an  exact sequence 
\begin{align*}
  0 &\to \H^0_{G}(Y,A)\to \pi_1\map_{B G}(Y,B\Pi'')  \to \pi_1\map_{B G}(Y,B\Pi') \\
&\to \H^1_{G}(Y,A) \to \pi_0\map_{B G}(Y,B\Pi'') \to  \pi_0\map_{B G}(Y,B\Pi') \to \H^2_{G}(Y,A)
\end{align*}
% of homotopy groups and sets., or equivalently
% \begin{align*}
%  0 &\to \H^0(Y,A) \to \H^0_{G}(Y,\Pi'') \to \H^0_{G}(Y,\Pi')\\
%  &\to \H^1_{G}(Y,A) \to \H^1_{G}(Y,\Pi'') \to  \H^1_{G}(Y,\Pi') \to \H^2_{G}(Y,A)
% \end{align*}
% of cohomology groups and sets. %%bad notation, as inconsistent with $A$.
In particular, the obstruction to lifting a homotopy class of maps $Y \to B\Pi'$ to $B\Pi''$ lies in $ \H^2_{G}(Y,A)$, and the ambiguity in this lift is given by an action of $\H^1_{G}(Y,A) $ on the fibres.

% % \begin{definition}%%now just appears later.
% %  Say that a simplicial group $H$ is bounded if its Dold--Kan normalisation $NH$ (given by  $N_nH = H_n \cap \ker_{i>0} \ker \pd_i$) is so.
% % \end{definition}

\begin{remark}\label{abproprmk}
% The proof of Proposition \ref{abprop} generalises as follows. First observe that nay simplicial profinite group can be written as an inverse limit $\Pi=\Lim_n \cosk_n\Pi_n$ of bounded simplicial profinite groups, so $\Pi$ can be written as an inverse limit of bounded simplicial finite groups.
%
 Given an abelian extension $\Pi''\to \Pi'$ of %pro-(bounded simplicial finite groups) 
pro-simplicial groups
with kernel $A$, such that the conjugation action of $\Pi'$ on $A$ factors through some quotient $G$ of $\Pi'$, 
there is a natural fibration sequence
\[
 \map_{\bar{W} G}(Y,\bar{W}\Pi'') \to  \map_{\bar{W} G}(Y,\bar{W}\Pi') \to \map_{\bar{W}G }(Y, \bar{W}(G \ltimes \bar{W}A))
\]
of mapping spaces, for any pro-space $Y$ over $\bar{W}G$.
\end{remark}

\begin{example}\label{obsexample}
In order to understand the first obstruction map  $\ob\co \pi_0\map_{B G}(Y,B\Pi') \to \H^2_{G}(Y,A)$ explicitly, consider the case when $Y$ is reduced and connected, so $Y_0=*$ and an element of  $\pi_0\map_{B G}(Y,B\Pi')$ is a conjugacy class of pro-group homomorphisms $\alpha \co \pi_1(Y) \to \Pi'$ over $G$. Here, $\pi_1Y$ is a pro-group with generators $Y_1$ and relations $\pd_1y = \pd_0y \cdot\pd_2y$ for  $y \in Y_2$.  
 Since $\Pi'' \to \Pi'$ is surjective, we may lift $\alpha$ to a morphism $\tilde{\alpha} \co  Y_1 \to \Pi''$ of pro-sets. The obstruction $\ob(\alpha)$ then measures the failure of $\tilde{\alpha}$ to be a group homomorphism, in the form of the $2$-cocycle 
\[
(y \in Y_2) \mapsto \tilde{\alpha}(\pd_2y) \tilde{\alpha}(\pd_1y)^{-1}\tilde{\alpha}(\pd_0y).
\]
\end{example}

\subsection{Nilpotent obstruction towers}\label{nilpsn} 

We can of course iterate the construction of Remark \ref{abproprmk}, by considering towers  $\ldots \Pi_{n+1} \to \Pi_n \to \ldots \to \Pi_0 = G$ of surjections whose kernels are abelian $G$-representations. The motivating examples are given by the quotients of $\pi_1^{\et}(X)$ by the lower central series of  $\pi_1^{\et}(\bar{X})$, and by their pro-$p$ completions relative to $G_{F,\Sigma}$ for $p \in \Sigma$.

Writing $A_n$ for the kernel of $\Pi_n \to \Pi_{n-1}$ and $\Pi_{\infty}:= \Lim_n \Pi_n$, we then have an exact couple 
\[
 \xymatrix{
\ldots   \ar[r] & \pi_*\map_{B G}(Y,B\Pi_n)  \ar[r] &  \pi_*\map_{B G}(Y,B\Pi_{n-1}) \ar[r] \ar@{-->}[ld]^{\delta} & \ldots \ar@{-->}[ld]^{\delta} \ar[r] & \pi_*\map_{B G}(Y,B\Pi_1) %\ar[r] &\pi_*\map_{B G}(Y,B\Pi_0)\ar@{-->}[ld]^{\delta}
\\
 &		\H^{1-*}_{G}(Y,A_n)  \ar[u]  & \H^{1-*}_{G}(Y,A_{n-1}) \ar[u]& \ldots 
& \H^{1-*}_{G}(Y,A_{1}) \ar[u]^{\simeq} %&  0\ar@{=}[u]
 }
\]
% \[
%  \xymatrix{
% \ldots   \ar[r] & \H^*_{G}(Y,\Pi_n)  \ar[r] &  \H^*_{G}(Y,\Pi_{n-1}) \ar[r] \ar@{-->}[ld]^{\delta} & \ldots \ar@{-->}[ld]^{\delta} \ar[r] &\H^*_{G}(Y,\Pi_2) \ar[r]& \H^*_{G}(Y,\Pi_1) \ar@{-->}[ld]^{\delta}\\
%  &		\H^{*}_{G}(Y,A_n)  \ar[u]  & \H^{*}_{G}(Y,A_{n-1})\ar[u] & \ldots 
% &\H^{*}_{G}(Y,A_{2}) \ar[u] & \H^{*}_{G}(Y,A_{1}) \ar@{=}[u] &  
%  }
% \]
similar to that  in \cite[\S VI.2]{sht}, but with the extra final terms $ \H^2_{G}(Y,A_{n}) $. Here, the connecting homomorphism $\delta$ is of homological degree $-1$, so we have
\[
 \delta \co  \pi_i\map_{B G}(Y,B\Pi_{n-1})\to \H^{2-i}_{G}(Y,A_n).
\]

This induces a non-abelian spectral sequence 
\[
 E_1^{s,t}= \H^{1+s-t}_{G}(Y, A_s) \abuts  \pi_{t-s}\map_{B G}(Y,B\Pi_{\infty} )
\]
% \[
%  E_1^{s,t}= \H^{1+s-t}_{G}(Y, A_s) \abuts  \H^{1+s-t}_{G}(Y, \Pi_{\infty} )
% \]
of groups and sets,
where the terms $E_1^{s,t}$ are only defined for $t\ge \max(s-1,0)$, and the indexing convention follows  \cite[\S VI.2]{sht}, with $d_r\co E_r^{s,t}\to E_r^{s+r,t+r-1}$. Unlike the fringed Bousfield--Kan spectral sequence of  \cite[\S VI.2]{sht}, we have  terms $E_r^{t+1,t}$ ensuring that we can recover the images of %$\H^{1}_{G}(F, \Pi_{\infty} ) \to \H^{1}_{G}(F, \Pi_s )$ 
\[
 \pi_0\map_{B G}(Y,B\Pi_{\infty} )\to  \pi_0\map_{B G}(Y,B\Pi_s )
\]
from our spectral sequence. 

Explicitly,  writing 
\[
% \pi_i\map_{B G}(Y,B\Pi_s )^{(r)}
\pi_i M_s^{(r)} := \im(\pi_i\map_{B G}(Y,B\Pi_{s+r})\to \pi_i\map_{B G}(Y,B\Pi_{s})),  
\]
 there are  long exact sequences
\begin{align*}
 \ldots \to  E_r^{s-r+1, t-r+2}\to \pi_{t-s+1}M_{s-r+1}^{(r-1)}
\to &\pi_{t-s+1}M_{s-r}^{(r-1)}\\
  &\to E_r^{s,t} \to \pi_{t-s}M_s^{(r-1)}  \to \pi_{t-s}M_{s-1}^{(r-1)} \to \ldots 
\end{align*}
(as in \cite[Lemma VI.2.8]{sht}, but with  extra final terms $\pi_0 M_{t+1-r}^{(r-1)}  \to E_r^{t+1,t}$).

The first page just corresponds to the  exact sequences
\begin{align*}
0 \to \H^{0}_{G}(Y, A_s) \to \pi_{1}M_{s}^{(0)} \to  \pi_{1}M_{s-1}^{(0)}  \to \\
   \H^{1}_{G}(Y, A_s) \to \pi_{0}M_{s}^{(0)} \to  \pi_{0}M_{s-1}^{(0)}  \to  \H^{2}_{G}(Y, A_s).
\end{align*}

% Explicitly,  writing $\H^*_{G}(Y,\Pi_s)^{(r)}:= \im(\H^*_{G}(Y,\Pi_{s+r})\to \H^*_{G}(Y,\Pi_{s}))$,  there are  long exact sequences
% \[
%  \ldots %\to  E_r^{s-r+1, t-r+2}
% \to \H^{s-t}_{G}(Y,\Pi_{s-r})^{(r-1)}\to E_r^{s,t} \to  \H^{1+s-t}_{G}(Y,\Pi_s)^{(r-1)} \to \H^{1-s+t}_{G}(Y,\Pi_{s-1})^{(r-1)}\to \ldots 
% \]
% (as in \cite[Lemma VI.2.8]{sht}, but with  extra final terms $ \H^1(Y, \Pi_{t+1-r})^{(r-1)} \to E_r^{t+1,t}$).

\begin{example}[Nilpotent completion of $\pi_1^{\et}(\bar{X})$]\label{nilpex1}
If $X$ is a scheme over $F$, and $\bar{X}:=X\ten_{F}\bar{F}$, with some geometric point $\bar{x}$, then the simplest examples are given by taking lower central series 
\[
 \Pi_n:= \pi_1^{\et}(X,\bar{x})/[\pi_1^{\et}(\bar{X}, \bar{x})]_{n+1},
\]
where for a profinite group $\pi$  we 
define $[\pi]_{k+1}$ inductively to be the closure of  $[\pi, [\pi]_k]$, with $[\pi]_1:=\pi$.

Thus $\Pi_0=G_{F}$, and taking $Y=BG_{F}$, we get  the non-abelian spectral sequence 
\[
 E_1^{s,t}= \H^{1+s-t}(G_{F}, [\bar{\pi}]_s/[\bar{\pi}]_{s+1} ) \abuts  \pi_{t-s} \map_{BG_{F}}(BG_{F}, B\Pi_{\infty} )
\]
of groups and sets, where we write $\bar{\pi}:= \pi_1^{\et}(\bar{X}, \bar{x})$. If $\bar{x}$ lies over a point in $X(F)$, then $\Pi_{\infty}$ is just the semi-direct product of $G_{F}$ and the  pro-nilpotent completion of $\pi_1^{\et}(\bar{X}, \bar{x})$.

Since  points in $X(F)$ map to elements in $\pi_0 \map_{BG_{F}}(BG_{F}, B\Pi_{\infty} )$, this spectral sequence gives  obstructions to the existence of such rational points. The same constructions work when $X$ is a Deligne--Mumford stack instead of a scheme, in which case we have a morphism from the groupoid $X(\sO_{F})$ to the fundamental groupoid $\pi_f \map_{BG_{F}}(BG_{F}, B\Pi_{\infty} )$.

The maps   $d_r \co E^{1,1}_r \to E_r^{r+1,r}$ are just Ellenberg's obstructions, which can be described in terms of Massey products as in Wickelgren's thesis \cite{wickelgrenLCS}.

Another variant is given by taking a smooth scheme $X$  over $\sO_{F,\Sigma}$ admitting a smooth relative compactification, and setting $\bar{X}:=X\ten_{\sO_{F,\Sigma}}\sO_{\bar{F},\Sigma}$, with some geometric point $\bar{x}$. For a prime $\ell$ all of whose $F$-prime factors lie in $\Sigma$, we can  consider 
the relative pro-$\ell$ completion (or more generally pro-nilpotent pro-$L$ for a set $L$ of such primes) of $\pi_1^{\et}(X,\bar{x}) $ over $G_{F,\Sigma}$ in the sense of \cite{hainmatrelative}, which will have the effect of replacing $[\bar{\pi}]_s/[\bar{\pi}]_{s+1}$ with $\ell$-torsion groups --- the corresponding maps are described in \cite{wickelgrenNilpotent}. 

Alternatively, if we replaced $BG_{F}$ with the \'etale homotopy type of an $F$-scheme $Z$, we would instead obtain topological obstructions to the existence of a map $Z \to X$ over $F$.
\end{example}

%%when thinking about relative pro-$\ell$ completion of $\SL_2(\Z)$ over $\SL_2(\bFl)$, the issue is that we're looking for non-congruence subgroups of $\Gamma(\ell)$ of $\ell$-power index. 

\begin{example}[Relative completion of $\pi_1^{\et}(\bar{X})$ --- descent obstructions]\label{descentex1}
When the geometric fundamental group of $X$ is perfect, its nilpotent completion is trivial, so the construction of Example \ref{nilpex1} gives no information. However, we can remedy this by taking the completion relative to a larger group than $G_{F}$. We may take any quotient $P$ of $\pi_1^{\et}(X,\bar{x})$ bigger than $G_{F}$, then  write $K:= \ker (\pi_1^{\et}(X, \bar{x})\to P)$, and set  $\Pi_n:= \pi_1^{\et}(X,\bar{x})/[K]_{n+1}$.

This gives a non-abelian spectral sequence
\[
 E_1^{s,t}= \begin{cases}
             \H^{1+s-t}(G_{F}, [K]_s/[K]_{s+1} ) & s \ge 1\\
\pi_t\map_{BG_{F}}(BG_{F}, BP) & s=0
            \end{cases}
  \abuts  \pi_{t-s} \map_{BG_{F}}(BG_{F}, B\Pi_{\infty} )
\]
of groups and sets. Here, the Galois action on $[K]_s/[K]_{s+1}$ depends on the relevant section $\sigma\in\pi_0 \map_{BG_{F}}(BG_{F}, BP)$.

When $P$ is a finite extension of $G_{F}$, each section $\sigma$ as above gives a finite \'etale group scheme $P^{\sigma}%= \Spec \Hom_{\Set}(\ker(P \to G_{F}), \bar{\O}_{F})^{G_{F}}
$ 
over $F$ with $P^{\sigma}(\bar{F})\cong \ker(P \to G_{F})$, and hence
$BP^{\sigma}$ having \'etale homotopy type $BP$. Even when $P$ is not a finite extension of $G_{F} $, we can write it as a filtered limit $\Lim_{\alpha} P_{\alpha}$ of such finite extensions, with each section $\sigma$ giving a pro-(finite \'etale) group scheme $P^{\sigma}= \Lim_{\alpha} P^{\sigma}_{\alpha}$ over $F$. Maps $X_{\et} \to BP$  then correspond to  $P^{\sigma}$-torsors $f^{\sigma} \co Y^{\sigma} \to X$, and we may substitute  $K \cong \pi_1(\bar{Y}^{\sigma}$  in the spectral sequence above. 
\end{example}

\begin{example}[Relative completion of $\pi_1^{\et}(\bar{Y}_{\Gamma})$]\label{modularex1}
As a special case of Example \ref{descentex1}, take a  congruence subgroup $\Gamma \le {\SL_2(\Z)}$; we may then form a stacky modular curve $Y_{\Gamma}$   over  a  number field $F$. The geometric fundamental group $\pi_1^{\et}(Y_{\Gamma}, \bar{x})$ is  the profinite completion $\hat{\Gamma}$ of $\Gamma$, so a point $x \in Y_{\Gamma}(F)$  gives an isomorphism   $\pi_1^{\et}(Y_{\Gamma}, \bar{x})\cong \Gamma \rtimes G_{F}$. 
The Tate module of the universal elliptic curve over $Y_{\Gamma}$ gives rise to a $\hat{\Z}$-local system of rank $2$ on $Y_{\Gamma}$, and hence a map
\[
  \pi_1^{\et}(Y_{\Gamma}) \to \GL_2(\hat{\Z})                                                                                                                                                                                                                                                                                                                                                                                                                                                                                                                                                                                                                                                                                                                                                                                \]
(for any choice of basepoint).

Since the local system has determinant $\hat{\Z}(1)$, this induces a map
\[
 \pi_1^{\et}(Y_{\Gamma}) \to \GL_2(\hat{\Z}) \by_{\bG_m(\hat{\Z})}G_{F},
\]
and we may then take the relative pro-nilpotent completion over the image, %(isomorphic to $\hat{\Gamma} \rtimes G_{F}$).
or the relative pro-$\ell$ completion over the image in $\GL_2(\Zl) \by_{\bG_m(\Zl)}G_{F} $. Since the maps $\GL_2(\Zl) \to \GL_2(\bFl)$ are pro-$\ell$ extensions, completion relative to $\GL_2(\bFl)$ gives the same limit from a different tower.

For $\Gamma = \SL_2(\Z)$, with $\Gamma(N):= \ker({\SL_2(\Z)}\to \SL_2(\Z/N))$, the  spectral sequence resulting from the pro-nilpotent tower relative to $\GL_2(\Z/N) \by_{\bG_m(\Z/N)}G_{F}$ is
\[
 \H^{1+s-t}(G_{F}, \widehat{[\Gamma(N)]_s/[\Gamma(N)]_{s+1}} ) \abuts  \pi_{t-s} \map_{B(\GL_2(\Z/N) \by_{\bG_m(\Z/N)}G_{F})}(BG_{F}, B\Pi_{\infty} ),
\]
where $\Pi_{\infty}:= \Lim_s \pi_1^{\et}(Y_{\Gamma})/[\hat{\Gamma}(N)]_s$. 

The spectral sequence relative to $\GL_2(\hat{\Z}) \by_{\bG_m(\hat{\Z})}G_{F}$ instead has
\[
 \H^{1+s-t}(G_{F}, \Lim_N\widehat{[\Gamma(N!)]_s/[\Gamma(N!)]_{s+1}} ) \abuts  \pi_{t-s} \map_{B(\GL_2(\hat{\Z}) \by_{\bG_m(\hat{\Z})}G_F)}(BG_{F}, B\Pi_{\infty}),
\]
for $\Pi_{\infty}= \Lim_{s,N} \pi_1^{\et}(Y_{\Gamma})/[\hat{\Gamma}(N!)]_s $.

%There are variants of this example for a stacky modular curve $Y_{\Gamma}$   over a ring of integers $\sO_{F,\Sigma}$, in which case there is an isomorphism between the relative pro-$\Sigma$ completion of  $\pi_1^{\et}(Y_{\Gamma}, \bar{x})$ over $G_{F,\Sigma}$ and the semidirect product $ \Gamma^{\wedge_{\Sigma}} \rtimes G_{F,\Sigma}$. The Tate module then gives a representation $\pi_1^{\et}(Y_{\Gamma}) \to \prod_{\ell}\GL_2(\Zl)$, where the product runs over those primes which are units in $\sO_{F,\Sigma}$, and the constructions above adapt provided we take  relative pro-$\Sigma$ completion.
\end{example}

\subsection{Unipotent extensions}\label{unipsn}

We now look to  consider towers  $\ldots \Pi_{n+1} \to \Pi_n \to \ldots \to \Pi_0$ 
of  unipotent extensions of Lie  groups  over $\Ql$, where all $F$-prime factors of $\ell$ lie in of our set $\Sigma$ of places of $F$.%%, with $R$ pro-reductive and $\ker(\Pi_{n+1} \to \Pi_n)$ commutative and unipotent  $R$-representations (so the adjoint aciton of $\Pi_{n+1}$ factors through $R$).

\begin{definition} 
 Say that a simplicial group $H$ is bounded if its Dold--Kan normalisation $NH$ (given by  $N_nH = H_n \cap \bigcap_{i>0} \ker \pd_i$) is so.
\end{definition}

\begin{lemma}\label{admisslemma}
 If $U$ is a bounded simplicial unipotent algebraic group over $\Ql$, equipped with a  continuous action of a profinite group $G$, then $U(\Ql)$ is the filtered colimit of its bounded simplicial profinite $G$-equivariant subgroups. 
\end{lemma}
\begin{proof}
 This is a slight generalisation of \cite[Lemmas \ref{weiln-admissprop} and \ref{weiln-admissiblelattice}]{weiln}, which address the case where the $G$-action is semisimple. Standard arguments give a $G$-equivariant bounded simplicial $\Zl$-submodule $\Lambda$ of the Lie algebra $\fu$ of $U$, with $\Lambda$ of finite rank and $\Lambda\ten \Ql \onto \fu$. The closure $g(\Lambda)$ of $\Lambda$ under monomial operations in the Campbell--Baker--Hausdorff product is still bounded and of finite rank, as $\fu$ is nilpotent, and the groups $g(\ell^{-n}\Lambda)$ realise $U(\Ql)$ as a filtered colimit of the required form.
\end{proof}

\begin{corollary}\label{admisscor}
Take an affine algebraic group  $T$ over $\Ql$ and a surjection $\Pi \to T$ of simplicial affine group schemes, with   $U:= \ker(\Pi \to R)$ bounded unipotent. Then for any Zariski-dense profinite group $G \subset T(\Ql)$, the simplicial topological group
\[
 \Pi(\Ql)\by_{T(\Ql)}G
\]
is a filtered colimit of those simplicial profinite subgroups which are bounded nilpotent extensions of $G$.
\end{corollary}
\begin{proof}
Since $\Pi(\Ql)\by_{T(\Ql)}G$ is the fibre of $ \Pi(\Ql)\by_{T^{\red}(\Ql)}G \to T(\Ql)\by_{T^{\red}(\Ql)}G $, it suffices to prove this for $T$ reductive. As in \cite{htpy}, the simplicial unipotent extension $\Pi \to T$ then admits a section (i.e. a Levi decomposition), unique up to conjugation by $U(\Ql)$; this gives an isomorphism $\Pi\cong T \ltimes U$. Since $G$ is Zariski dense in the reductive group $T$, its action is semisimple so we may appeal to Lemma \ref{admisslemma}, writing 
\[
 G\by_{T(\Ql)}\Pi(\Ql)\cong  G\ltimes U \cong \LLim_{\alpha} G \ltimes N_{\alpha},
\]
for bounded $G$-equivariant simplicial profinite subgroups $N_{\alpha}$ of $U$.
\end{proof}

The nerve $\bar{W}(\Pi(\Ql)\by_{T(\Ql)}G)$ is then an ind-pro-simplicial set, and defining mapping spaces for these by the usual convention
\[
 \map(Y,\{Z_{\alpha}\}):= \LLim_{\alpha} \map(Y, Z_{\alpha}), 
\]
for  $Y, Z_{\alpha}$ profinite, we may  apply Proposition \ref{abprop} to unipotent extensions, by passing to filtered colimits:

 \begin{proposition}\label{abpropunip}
 Take a unipotent extension $\Pi''\to \Pi'$ of algebraic groups over $\Ql$ with commutative kernel $A$, such that the conjugation action of $\Pi'$ on $A$ factors through some quotient $\Pi$ of $\Pi'$. Then for any Zariski-dense map $G \to \Pi(\Ql)$ with $G$ profinite,  and for any pro-simplicial set $Y$ over $BG$, there is a natural fibration sequence
 \[
  \map_{B G}(Y,B(\Pi''\by_{\Pi}G)) \to  \map_{B G}(Y,B(\Pi'\by_{\Pi}G)) \to \map_{BG }(Y, B(G \ltimes BA))
 \]
 of mapping spaces,  the fibre being taken over the zero map $Y \to BG \to B(G \ltimes BA)$. 
 \end{proposition}

\begin{example}[Unipotent completion of $\pi_1^{\et}(\bar{X})$]\label{unipex1}
If $X$ is a smooth scheme over $\sO_{F,\Sigma}$ admitting a smooth relative compactification, and $\bar{X}:=X\ten_{\sO_{F,\Sigma}}\sO_{\bar{F},\Sigma}$, with some geometric point $\bar{x}$, then we may use Proposition \ref{abpropunip} to give a variant of Example \ref{nilpex1}. For simplicity, assume that we have a point $x \in X(\sO_{F,\Sigma})$ under $\bar{x}$ (if not, we can recover analogues of  the constructions below by taking a $G_{F,\Sigma}$-equivariant  set $\cB \subset X(\sO_{\bar{F},\Sigma})$, then consider the $G_{F,\Sigma}$-equivariant surjection from the groupoid $ \pi_1^{\et}(\bar{X},\cB)$ to the contractible groupoid on objects $\cB$).

Now,  \cite{friedlanderfib} shows that  $\bar{X}_{\et}$ has  equivalent derived pro-$\Sigma$  completion to the homotopy fibre of $X_{\et} \to (\Spec  \sO_{F,\Sigma})_{\et}$. Since the section $x$ splits the long exact  sequence of homotopy groups, this gives an isomorphism between
the relative pro-$\Sigma$ completion of $ \pi_1^{\et}(X,\bar{x})$  over $G_{F,\Sigma}$ and the semi-direct product $ G_{F,\Sigma} \ltimes \pi_1^{\et}(\bar{X}, \bar{x})^{\wedge_{\Sigma}}$.  We then consider the lower central series 
\[
 \Pi_n:= G_{F,\Sigma} \ltimes  (\pi_1^{\et}(\bar{X},\bar{x})\ten \Ql)/[\pi_1^{\et}(\bar{X}, \bar{x})\ten\Ql]_n,
\]
of the pro-unipotent Malcev completion $\pi_1^{\et}(\bar{X},\bar{x})\ten \Ql$.

Thus $\Pi_0=G_{F,\Sigma}$, and taking $Y=BG_{F,\Sigma}$ we get  a non-abelian spectral sequence 
\[
 E_1^{s,t}= \H^{1+s-t}(G_{F,\Sigma}, [\bar{\pi}\ten \Ql]_s/[\bar{\pi}\ten \Ql]_{s+1} ) \abuts  \pi_{t-s} \map_{BG_{F,\Sigma}}(BG_{F,\Sigma}, B\Pi_{\infty} )
\]
of groups and sets, where we write $\bar{\pi}:= \pi_1^{\et}(\bar{X}, \bar{x})$. 

Although this gives weaker obstructions than Example \ref{nilpex1}, the obstruction spaces are easier to calculate. The vector spaces $[\bar{\pi}\ten \Ql]_s/[\bar{\pi}\ten \Ql]_{s+1}$ are the graded pieces of a pro-nilpotent Lie algebra with generators $\H_1(\bar{X},\Ql)$ and relations non-canonically isomorphic to $\H_2(\bar{X},\Ql)$. 
Since  points in $X(\sO_{F,\Sigma})$ map to elements in $\pi_0 \map_{BG_{F,\Sigma}}(BG_{F,\Sigma}, B\Pi_{\infty} )$, this spectral sequence gives  obstructions to the existence of such rational points. 
\end{example}

\subsection{Pro-unipotent extensions}\label{prounipsn}

Relative Malcev completion was introduced by Hain in \cite{hainrelative} for discrete groups, and as in \cite{weight1} we consider the natural generalisation to profinite groups  as follows:
\begin{definition}
Given a topological group $\Gamma$, a reductive pro-algebraic group $R$ over $\Ql$, and a Zariski-dense continuous representation $\rho\co \Gamma \to R(\Ql)$, 
   Define the Malcev completion $(\Gamma)^{\rho,\mal}$  
to be the universal diagram
\[
\Gamma \to \Gamma^{\rho,\mal}(\Ql) \xra{p} R(\Ql),
\]
with $p\co \Gamma^{\rho,\mal} \xra{p} R$  a pro-unipotent extension, and the composition equal to $\rho$.

When the representation $\rho$ is clear from the context, we will write $\Gamma^{R,\mal}:=\Gamma^{\rho,\mal}$.
\end{definition}

\begin{remark}\label{relmalpropsrmk}
 The pro-unipotent radical $\Ru(\Gamma,\rho)^{\mal}$ is then given by $\exp(\fu)$ for a pro-(finite-dimensional nilpotent) Lie algebra $\fu$. For $O(R)$ the ring of algebraic functions on $R$ over $\Ql$, equipped with its left $R$-action, the abelianisation of $\fu$ is  dual to the continuous cohomology $\H^1(\Gamma, O(R))$, and there is a presentation of $\fu$ with relations dual to $\H^2(\Gamma, O(R))$. In particular, if $\H^2(\Gamma, O(R))=0$, then there are canonical isomorphisms
\[
 [\Ru(\Gamma,\rho)^{\mal}]_n/[\Ru(\Gamma,\rho)^{\mal}]_{n+1} \cong (\CoLie_n \H^1(\Gamma, O(R)))^{*},
\]
where  $\CoLie_n(V) = \Lie(n)^{*}\ten_{S_n}V^{\ten n}$ for the Lie operad $\Lie$. Explicitly, when $V$ is finite-dimensional, $(\CoLie_n V)^{*}$ is the subspace of the free Lie algebra on generators $V^{*}$ consisting of homogeneous terms of bracket length $n$. 

Also note that if $\Gamma$ is a discrete group and $\hat{\Gamma}$ its profinite completion, then for any representation $\rho$ of $\hat{\Gamma}$, the map $\Gamma^{\rho,\mal}\to \hat{\Gamma}^{\rho,\mal}$ is necessarily an isomorphism.
\end{remark}

\begin{examples}[$\widehat{\SL_2(\Z)}$]\label{SL2ex1}
Our main motivating example is to take $\Gamma= {\SL_2(\Z)}$ and its  profinite completion 
%of a congruence subgroup 
$\hat{\Gamma}$, with $R= \SL_2$ (regarded as a group scheme over $\Ql$) and $ \hat{\Gamma} \to \SL_2(\Ql)$ the natural map. 

Since the ring $O(\SL_2)$ of functions is given by $\bigoplus_m V_m\ten (UV_m)^{*}$, for $V_m$ the irreducible $\SL_2$-representation of dimension $m+1$ over $\Ql$ and $UV_m$ the underlying vector space, we have
\[
\H^*(\Gamma, O(\SL_2)) \cong \bigoplus_m \H^*(\Gamma, V_m)\ten V_m^{*}.
\]
Thus $\H^2(\Gamma, O(\SL_2))=0$, and  Eichler--Shimura gives a description of $\H^1(\Gamma, O(\SL_2))\ten \Cx$ in terms of the decomposition of $\H^1(\Gamma,V_m)\ten \Cx$ into  modular forms and cusp forms of weight $m+2$ and level $1$. 

Our groups of interest are $\H^1(\Gamma,V_m)$ 
We may think of the spaces $\H^1(\Gamma,V_m)$ as $\ell$-adic realisations of motives of modular forms, as in \cite{deligneFormesModulairesladiques}. These $\Ql$-vector spaces admit $G_{\Q}$-actions via the interpretation as summands of $\H^{m+1}_{\et}(\cM_{1,m+1}\ten\bar{\Q}, \Ql)(m)$, interpreting $\cM_{1,m+1}$ as the $m$-fold product of the universal elliptic curve $\cM_{1,2}$ over the moduli stack $\cM_{1,1}$ of elliptic curves (the Tate twists arise because we wish to regard $V_1$ as a Tate module rather than its dual).

%%decomposition $M_f$, as intersection of espaces of $T_p$ with evals $a_p$.

%As in \cite{deligneValeursFonctionsL,schollMotModular}, 
%%See http://mathoverflow.net/questions/247219/deligne-scholls-motives-for-modular-forms-hecke-operators-vs-transposed-hecke , and Deligne--Scholl.

%%Note that the kernel $K$ of $\widehat{\SL_2(\Z)}\to \SL_2(\Zl)$ maps onto $\prod_{p \ne \ell} \SL_2(\Z_p)$, which is (close to) perfect, so we know that $[K]_n$ also surjects onto it, meaning that our limit gives no other Tate modules.

More generally, we can take $\Gamma$ to be a congruence subgroup of ${\SL_2(\Z)}$, giving a similar expression involving modular forms of higher levels, but with relations coming from $\H^2(\Gamma, O(\SL_2))$ whenever it is non-zero.

Alternatively, we can look at the relative Malcev completion of the canonical morphism
\[
 {\SL_2(\Z)}  \to \SL_2(\hat{\Z}) \by \SL_2(\Ql)=:R,
\]
where we regard the profinite group $\SL_2(\hat{\Z})$ as an affine group scheme over $\Ql$. Then we still have $\H^2({\SL_2(\Z)}, O(R))\ten \Q=0$, and Leray--Serre gives
\[
 \H^*({\SL_2(\Z)},O(\SL_2(\hat{\Z}))\ten V) \cong \LLim_N \H^*(\Gamma(N!), V),
\]
 so 
\[
 \H^1({\SL_2(\Z)}, O(R))\cong \bigoplus_m \LLim_N \H^1(\Gamma(N!), V_m)\ten V_m^{*},
\]
giving generators of $\Ru \widehat{\SL_2(\Z)}^{(\SL_2(\hat{\Z}) \by \SL_2), \mal}=\Ru {\SL_2(\Z)}^{(\SL_2(\hat{\Z}) \by \SL_2), \mal}$ in terms of modular and cusp forms of all weights and levels.

We can also just look at the relative Malcev completion of the canonical morphism
$
 {\SL_2(\Z)}  \to \SL_2(\hat{\Z})
$,
again regarding $\SL_2(\hat{\Z})$ as an affine group scheme over $\Ql$. We then have 
\[
 \H^1({\SL_2(\Z)}, O(\SL_2(\hat{\Z})))\cong  \LLim_N \H^1(\Gamma(N!),\Ql)
\]
with the corresponding $\H^2$ vanishing, giving generators for $\Ru {\SL_2(\Z)}^{\SL_2(\hat{\Z}), \mal}$ in terms of modular and cusp forms of  weight $2$ and all levels.
\end{examples}

For our purposes, Proposition \ref{abprop} is now not quite general enough, as our group schemes might not be of finite type. Consider an affine group scheme $T$ over $\Ql$ and a surjection $\Pi \to T$ of simplicial affine group schemes, with   $U:= \ker(\Pi \to R)$ bounded pro-unipotent, together with a Zariski-dense profinite group $G \subset T(\Ql)$. We can then canonically write the morphism $\Pi \to T$ as a filtered limit of unipotent extensions $\Pi_a \to T_a$ of affine algebraic groups, with Corollary \ref{admisscor} giving that 
$
 \Pi_a(\Ql)\by_{T_a(\Ql)}G
$
is an ind-profinite group, so $\Pi(\Ql)\by_{T(\Ql)}G$ is naturally a pro-ind-profinite group.

The nerve $\bar{W}(\Pi(\Ql)\by_{T(\Ql)}G)$ is then a  pro-ind-pro-simplicial set, and defining mapping spaces for these by the usual convention
\[
 \map(Y,\{Z_a\}):= \ho\Lim_{a} \map(Y, Z_{a}), 
\]
for  $Y$ profinite and $Z_{a}$ ind-profinite,   Proposition \ref{abpropunip} extends verbatim to pro-unipotent extensions $\Pi''\to \Pi'$.

\begin{example}[Modular forms of level $1$]\label{modularex2}
 If $X= \cM_{1,1}$ is the stacky modular curve over $\sO_{F,\Sigma}$, and $x \in X(\sO_{F,\Sigma})$, then  the Tate module gives a surjective homomorphism $\pi_1^{\et}(\bar{X}, \bar{x}) \to \SL_2(\Zl)$ whose relative pro-$\ell$ completion (or equivalently that over  $\SL_2(\bFl)$) is the same as that of $\SL_2(\Z)$. 
In particular, there is  a natural action of $G_{F,\Sigma}$ on  the relative Malcev completion $\SL_2(\Z)^{\SL_2, \mal} \cong \pi_1^{\et}(\bar{X}, \bar{x})^{\SL_2, \mal}$, and  we may consider the pro-unipotent extension
\[
 G_{F,\Sigma} \ltimes( \SL_2(\Z)^{\SL_2, \mal}\by_{\SL_2(\Ql)}\SL_2(\Zl)) \to  G_{F,\Sigma} \ltimes \SL_2(\Zl),
\]
setting
\[
 \Pi_n :=  G_{F,\Sigma} \ltimes( (\SL_2(\Z)^{\SL_2, \mal}/[\Ru]_{n+1})\by_{\SL_2(\Ql)}\SL_2(\Zl)).
\]

As in Example \ref{modularex1}, for the representation $G_{F,\Sigma} \to \Zl^*$ given by the Tate motive $\Zl(1)$, we have $G_{F,\Sigma} \ltimes \SL_2(\Zl) \cong G_{F,\Sigma}\by_{\Zl^*}\GL_2(\Zl)$, so a section of the projection $G_{F,\Sigma} \ltimes \SL_2(\Zl) \to G_{F,\Sigma}$ is equivalent to giving a $G_{F,\Sigma}$-representation $\Lambda$ of rank $2$ over $\Zl$, with determinant $\Zl(1)$. 

For the universal elliptic curve $f\co E \to X$, we have the Tate module $\bT_{\ell}:= (\oR^1 f_*\Zl)^*\cong \oR^1 f_*\Zl(1)$, a lisse $\Zl$-sheaf of rank 
 $2$ on $X$, giving a $G_{F,\Sigma}$-action on $\H^1(\SL_2(\Z),V_m)$  by identifying it with $\oR^1q_*(S^m\bT_{\ell})\ten \Q$, for the structure morphism $q\co X\to \Spec \sO_{F,\Sigma}$.
% a $\Ql$-sheaf $\bV_1:=\oR^1 f_*\Ql$ of rank $2$ on $X$, giving a $G_{F,\Sigma}$-action on $\H^1(\SL_2(\Z),V_m) \cong \H^1(\bar{X}, S^m\bV_1)$. 

Write 
\begin{align*}
  L_s&:= \CoLie_s (\bigoplus_m \H^1(\SL_2(\Z), V_m) \ten S^m(\Lambda)^{*}),\\
&\cong \CoLie_s (\bigoplus_m \H^1(\SL_2(\Z), V_m) \ten S^m(\Lambda)(-m)).
\end{align*}
Adapting Example \ref{nilpex1}, the pro-unipotent generalisation of Proposition \ref{abpropunip} then combines with Examples \ref{SL2ex1} to  give a non-abelian spectral sequence 
\[
 E_1^{s,t}= \H^{1+s-t}(G_{F,\Sigma}, L_s^{*}) \abuts  \pi_{t-s} \map_{B(G_{F,\Sigma}\by_{\bG_m(\Zl)} \GL_2(\Zl))}(BG_{F,\Sigma}, B\Pi_{\infty} ),
\]
where the map $G_{F,\Sigma} \to \GL_2(\Zl)$ is given by $\Lambda$. Note that $\H^1(\SL_2(\Z),V_m)(-m)$ is mixed of weights $m+1$ (cusp forms and their conjugates) and $2m+2$ (Eisenstein series), and that $S^m\Lambda_{\Q}$ is pure of weight $-m$. Thus
$\H^1(\SL_2(\Z),V_m) \ten S^m(\Lambda)(-m)$ is mixed of weights $1$  and $m+2$, so $L_s$ is of strictly positive weights, and $E^1_{s,s+1}=0$. 

Now set $X_{(n)}:=\map_{B G_{F,\Sigma}}(BG_{F,\Sigma},B\Pi_n)$; thus $X_{(0)}$ consists of  representations $G_{F,\Sigma} \to \GL_2(\Zl)$ whose determinant is the Tate motive, conjugation by $\SL_2(\Zl)$ giving  equivalences, so $\pi_1(X_{(0)}, [\Lambda])$ consists of elements of $\SL_2(\Zl)$ commuting with the action of $G_{F,\Sigma}$ on $\Lambda$.  Since  $\pi_iX_{(n)}=0$ for $i>1$, we then have exact sequences
 \begin{align*}
  0 \to  \pi_1 X_{(n)} \to \pi_1 X_{(n-1)} \to
 \H^1(F, L_n^{*}) \to   \pi_0 X_{(n)} \to \pi_0 X_{(n-1)} \to \H^2(F, L_n^{*}),
 \end{align*}
with a map $X(\sO_{F,\Sigma}) \to  X_{(\infty)}$. Here, $X(\sO_{F,\Sigma})$ is the nerve of the groupoid of maps $\Spec \sO_{F,\Sigma} \to X$, so $\pi_0X(\sO_{F,\Sigma})$ is the set of isomorphism classes of elliptic curves over $\sO_{F,\Sigma}$, and $\pi_1(X(\sO_{F,\Sigma}),x)$ the group of automorphisms of the elliptic curve $E_x$ over $\sO_{F,\Sigma}$; the higher homotopy groups all vanish.

In other words, given a $G_{F,\Sigma}$-representation $\Lambda$ of rank $2$ over $\Zl$, with determinant $\Zl(1)$, these sequences give a tower of obstructions to lifting $\Lambda$ to an elliptic curve over $\sO_{F,\Sigma}$ with Tate module $\Lambda$, and characterise the ambiguity of the lift at each stage. % Beware that the choice of lift at each stage in the tower will not be unique unless  $\H^1(\sO_{F,\Sigma}, L_n^{*})$ vanishes.
 As in Examples \ref{SL2ex1}, there is an entirely similar treatment for profinite completions of congruence subgroups $\Gamma \le {\SL_2(\Z)}$, replacing $\cM_{1,1} $ with the modular curve $Y_{\Gamma}$.
\end{example}

\begin{example}[Modular forms of all levels]\label{modularex2b}
Again taking $X= \cM_{1,1}$ to be the stacky modular curve over a number field $F$ and $x \in X(F)$,  we may consider the pro-unipotent extension
\[
 G_{F} \ltimes( \SL_2(\Z)^{\SL_2 \by \SL_2(\hat{\Z}), \mal}\by_{\SL_2(\Ql)\by \SL_2(\hat{\Z})}\SL_2(\hat{\Z})) \to  G_{F} \ltimes \SL_2(\hat{\Z}),
\]
setting
\[
 \Pi_n :=  G_{F} \ltimes( (\SL_2(\Z)^{\SL_2\by \SL_2(\hat{\Z}), \mal}/[\Ru]_{n+1})\by_{\SL_2(\Ql)\by \SL_2(\hat{\Z})} \SL_2(\hat{\Z})).
\]

Choose a section of the projection $G_{F} \ltimes \SL_2(\hat{\Z}) \to G_{F}$; this  is equivalent to giving a $G_{F}$-representation $\Lambda$ of rank $2$ over $\hat{\Z}$, with determinant $\hat{\Z}(1)$. 
Write 
\[
M_s:= \CoLie_s ( \bigoplus_m \LLim_N \H^1(\Gamma(N!),V_m) \ten_{\hat{\Z}} S^m_{\hat{\Z}}(\Lambda)(-m)).
\] 
 As in Example \ref{modularex2}, we then have a non-abelian spectral sequence 
\[
 E_1^{s,t}= \H^{1+s-t}(G_{F}, M_s^{*}) \abuts  \pi_{t-s} \map_{B(G_{F}\by_{\bG_m(\hat{\Z})} \GL_2(\hat{\Z}))}(BG_{F}, B\Pi_{\infty} ),
\]
where the map $G_{F} \to \GL_2(\hat{\Z})$ is given by $\Lambda$. Since 
\[
O(\SL_2(\hat{\Z}))\ten\H^1(\Gamma(N!),V_m) \ten S^m(\Lambda)(-m)
\] 
is mixed of weights $1$ (cusp forms of all levels and their conjugates) and $m+2$ (Eisenstein series of all levels), so $M_s$ is of strictly positive weights, and $E^1_{s,s+1}=0$. 

Now set $X_{(n)}:=\map_{B G_{F}}(BG_{F},B\Pi_n)$; thus $X_{(0)}$ consists of  representations $G_{F} \to \GL_2(\hat{\Z})$ whose determinant is the Tate motive, conjugation by $\SL_2(\hat{\Z})$ giving  equivalences.  Since  $\pi_iX_{(n)}=0$ for $i>1$, we then have exact sequences
 \begin{align*}
  0 \to  \pi_1 X_{(n)} \to \pi_1 X_{(n-1)} \to
 \H^1(F, M_n^{*}) \to   \pi_0 X_{(n)} \to \pi_0 X_{(n-1)} \to \H^2(F, M_n^{*}),
 \end{align*}
with a map $X(F) \to  X_{(\infty)}$. 
\end{example}

\begin{example}[Modular forms of weight $2$]\label{modularex2c}
Again taking $X= \cM_{1,1}$ to be the stacky modular curve over $F$, and $x \in X(F)$,  we may consider the pro-unipotent extension
\[
 G_{F} \ltimes \SL_2(\Z)^{\SL_2(\hat{\Z}), \mal} \to  G_{F} \ltimes \SL_2(\hat{\Z}),
\]
setting
\[
 \Pi_n :=  G_{F} \ltimes (\SL_2(\Z)^{\SL_2(\hat{\Z}), \mal}/[\Ru]_{n+1}).
\]

As in Example \ref{modularex2c}, choose  a $G_{F}$-representation $\Lambda$ of rank $2$ over $\hat{\Z}$, with determinant $\hat{\Z}(1)$. 
Write 
\[
M_s:= \CoLie_s (  \LLim_N \H^1(\Gamma(N!),\Ql));
\] 
thus $M_1$ is related to weight $2$ modular forms; as a Galois representation it is mixed of weights $1$ and $2$.
We then have a non-abelian spectral sequence 
\[
 E_1^{s,t}= \H^{1+s-t}(G_{F}, M_s^{*}) \abuts  \pi_{t-s} \map_{B(G_{F}\by_{\bG_m(\hat{\Z})} \GL_2(\hat{\Z}))}(BG_{F}, B\Pi_{\infty} ),
\]
where the map $G_{F} \to \GL_2(\hat{\Z})$ is given by $\Lambda$. 

Set
%
% Since 
% \[
% O(\SL_2(\hat{\Z}))\ten\H^1(\SL_2(\Z),\Ql)
% \] 
% is mixed of weights $1$ (cusp forms of all levels and their conjugates) and $2$ (Eisenstein series of all levels), so $M_s$ is of strictly positive weights, and $E^1_{s,s+1}=0$. 
%
$X_{(n)}:=\map_{B G_{F}}(BG_{F},B\Pi_n)$; since  $\pi_iX_{(n)}=0$ for $i>1$, we then have exact sequences
 \begin{align*}
  0 \to  \pi_1 X_{(n)} \to \pi_1 X_{(n-1)} \to
 \H^1(F, M_n^{*}) \to   \pi_0 X_{(n)} \to \pi_0 X_{(n-1)} \to \H^2(F, M_n^{*}),
 \end{align*}
with a map $X(F) \to  X_{(\infty)}$. 
\end{example}

\begin{example}[\'Etale fundamental groups]\label{relmalex}
For any smooth Deligne--Mumford stack $X$ over $\sO_{F,\Sigma}$ admitting a smooth relative compactification, with $x \in X(\sO_{F,\Sigma})$, we can generalise the examples above by considering any $G_{F,\Sigma}$-equivariant Zariski-dense representation $\rho \co \pi_1^{\et}(\bar{X}, \bar{x})^{\wedge_{\Sigma}}\to R(\Ql)$ to a pro-reductive affine group scheme $R$ over $\Ql$. If there is no rational basepoint,  we can instead  take a  $G_{F,\Sigma}$-equivariant set  $\cB \subset X(\sO_{\bar{F},\Sigma})$ of basepoints, then consider the $G_{F,\Sigma}$-equivariant surjection from the groupoid $ \pi_1^{\et}(\bar{X},\cB)$ to the contractible groupoid on objects $\cB$, with relative Malcev completions as in \cite[\S \ref{htpy-malcev}]{htpy}).

We may then set
\[
 \Pi_n :=  G_{F,\Sigma} \ltimes(( \pi_1^{\et}(\bar{X}, \bar{x})^{R, \mal}/[\Ru]_{n+1})\by_{R(\Ql)}\rho(\pi_1^{\et}(\bar{X}, \bar{x})^{\wedge_{\Sigma}} )),
\]
with $P_n:= \ker(\Pi_n \to \Pi_{n-1})$ a quotient of $(\CoLie_n (\H^1(\bar{X},O(R))))^*$  described as in Remark \ref{relmalpropsrmk}. 

For any section $\sigma$ of the projection $G_{F,\Sigma} \ltimes \rho(\pi_1^{\et}(\bar{X}, \bar{x})^{\wedge_{\Sigma}}) \to G_{F,\Sigma}$, we then have a non-abelian spectral sequence 
\[
 E_1^{s,t}= \H^{1+s-t}(G_{F,\Sigma},P_s) \abuts  \pi_{t-s} \map_{B(G_{F,\Sigma}\ltimes\rho(\pi_1^{\et}(\bar{X}, \bar{x})^{\wedge_{\Sigma}} ) )}(BG_{F,\Sigma}, B\Pi_{\infty} ),
\] 
where the map $G_{F,\Sigma} \to G_{F,\Sigma} \ltimes \rho(\pi_1^{\et}(\bar{X}, \bar{x})^{\wedge_{\Sigma}})$ is given by $\sigma$. 
\end{example}

\begin{example}[\'Etale homotopy types]\label{exhtpy}
We may refine the previous example by considering \'etale homotopy types in place of fundamental groups. Take a locally Noetherian simplicial scheme $X$, and a geometric point $\bar{x}$. We can then form the \'etale  topological type  $X_{\et} \in \pro(\bS)$ as defined in \cite[Definition 4.4]{fried}. In particular, we can apply this to a simplicial scheme resolving a locally Noetherian algebraic stack (by \cite[Theorem 4.7]{stacks2}, such resolutions exist even for higher Artin stacks, and the description of \cite[Theorems 4.10  and Remark 4.11]{stacks2} ensures that the choice does not affect the homotopy type). 

Note that $(X_{\et})_0$ is the set of geometric points of $X_0$ (with some bound imposed on the cardinalities of the associated fields). Consider the reduced pro-simplicial set $(X_{\et}, \bar{x}) \subset  X_{\et}$ given by setting $(X_{\et}, \bar{x})_n$ to consist of $n$-simplices with fixed vertex $\bar{x}$. We may then apply the  simplicial loop groupoid functor of \cite{pathgpd}
to get a pro-simplicial groupoid $GX_{\et}$, and restricting to the vertex $\bar{x}$ gives a pro-simplicial groupoid $G(X_{\et},\bar{x})$ with $\pi_0G(X_{\et}, \bar{x})\cong \pi_{1}^{\et}(X, \bar{x})$. 

If $X$ is defined over $\sO_{F,\Sigma}$, with each $X_n$  admitting a smooth relative compactification, set $\bar{X}:=X\ten_{ \sO_{F,\Sigma}}\sO_{\bar{F},\Sigma}$. 
Now fix a Zariski-dense representation $\rho\co  \pi_1^{\et}(X,\bar{x}) \to S(\Ql)$ to a  pro-reductive pro-algebraic group $S$, and let $R$ be the Zariski closure of $\rho(\pi_1^{\et}(\bar{X},\bar{x}))$, and set $T:= S/R$. We now need to consider fibre sequences, because $G_{F,\Sigma}$ does not explicitly act on our model for $\bar{X}_{\et}$. If the  $G_{F,\Sigma}$-representation $\H^*(\bar{X}, V)$ is an extension of $T$-representations  for all $R$-representations $V$, then  \cite[Theorem \ref{weiln-lfibrations}]{weiln} combines with \cite{friedlanderfib} to give a fibre sequence
\[
 \bar{W}G(\bar{X}_{\et}, \bar{x})^{R,\mal} \to \bar{W}G(X_{\et}, \bar{x})^{S,\mal} \to BG(\sO_{F,\Sigma,\et})^{T,\mal} 
\]
of pro-algebraic homotopy types over $\Ql$. By \cite[Appendix  A, Equation (1)]{cesnaviciusPT}), $ BG_{F,\Sigma}$ and  $\sO_{F,\Sigma,\et}$ have isomorphic cohomology for $\ell$-torsion coefficients, so  $G(BG_{F,\Sigma})^{T,\mal}\simeq  G(\sO_{F,\Sigma,\et})^{T,\mal}$ and we have a long exact sequence 
\begin{eqnarray*}
\ldots \to \varpi_n(\bar{X} ,\bar{x})^{R,\mal} \to \varpi_n(X,\bar{x})^{S,\mal}\to \varpi_n(BG_{F,\Sigma})^{T,\mal} \to \varpi_{n-1}(\bar{X} ,\bar{x})^{R,\mal}\to \\
\ldots \to \pi_1^{\et}(\bar{X} ,\bar{x} )^{R,\mal} \to \pi_1^{\et}(X,\bar{x})^{S,\mal}\to G_{F,\Sigma}^{T,\mal} \to 1
\end{eqnarray*}
of pro-algebraic homotopy groups; in particular we will have an exact sequence of completed fundamental groups whenever $\varpi_2(BG_{F,\Sigma})^{T,\mal}=0$, i.e. if $G_{F,\Sigma}$ is $2$-good relative to $T$ in the sense of \cite[Definition \ref{weiln-relgood2}]{weiln} and \cite[\S \ref{heid-relgoodsn}]{heid}. %%somewhere here might give example of sth not good, $\Sp_g$?

We may then set $\tilde{\Pi}_n$ to be the simplicial topological group given by the homotopy fibre product
\[
\tilde{\Pi}_n:= (G(X_{\et}, \bar{x})^{S,\mal}/[U]_{n+1})\by^h_{G(BG_{F,\Sigma})^{T,\mal} }G_{F,\Sigma},
\]
where $U= \Ru G(\bar{X}_{\et}, \bar{x})^{R,\mal}$; 
in particular, $\tilde{\Pi}_0 = S\by_T G_{F,\Sigma}$.
Note that since  $B\tilde{\Pi}_{\infty}$ is equipped with a map from $\bar{W}G(\bar{X}_{\et}, \bar{x}) $, there is a canonical morphism
\[
 X(\sO_{F,\Sigma}) \to \map_{B(G_{F,\Sigma})}(BG_{F,\Sigma}, B\tilde{\Pi}_{\infty} )
\]
in the homotopy category of pro-ind-pro-simplicial sets.

We will then have a non-abelian spectral sequence 
\[
 E_1^{s,t} \abuts  \pi_{t-s} \map_{B(G_{F,\Sigma})}(BG_{F,\Sigma}, B\tilde{\Pi}_{\infty} ),
\]
with
\[ 
 E_1^{s,t}= \begin{cases}
             \bH^{1+s-t}(G_{F,\Sigma}, [U]_s/[U]_{s+1}) & s \ge 1 \\
 \pi_{t}\map_{BT}(BG_{F,\Sigma}, BS) & s=0,
            \end{cases}
\]
where $[U]_s/[U]_{s+1}$ is dual to  $ \CoLie_n ((\oR\Gamma(\bar{X},O(R))/\Ql)[1])$  (associating $O(R)$ with an ind-lisse $\Ql$ sheaf via $\rho$ and   Remark \ref{relmalpropsrmk}),  for $\CoLie$ now the cofree graded Lie coalgebra, and $\tilde{\H}^{\bt}= \H^{\bt}/\H^0$. Beware that the terms $ E_1^{s,t} $ depend on an element of $E_1^{0,0}$ to determine  the Galois action on $U$.

The filtration $[U]_n$ corresponds to the good truncation filtration on $\oR \Gamma(\bar{X},O(R))$, but there are variants for other filtrations, replacing  ${\H}^{1+\bt}(\bar{X},O(R))$ with the $E_1$ page of the associated spectral sequence. For the case of the weight filtration on a quasi-projective variety, with representations tamely ramified around the divisor, see  \cite[Corollary \ref{weiln-htpyleray}]{weiln}. 

Note that taking path components of simplicial groups $\tilde{\Pi}_n$ 
gives morphisms
\[
\map_{B(G_{F,\Sigma})}(BG_{F,\Sigma}, B\tilde{\Pi}_n ) \to \map_{B(G_{F,\Sigma})}(BG_{F,\Sigma}, B\Pi_n ) 
\]
for the groups $\Pi_n$ of Example \ref{relmalex}. When $\H^{\ge 2}(\bar{X},O(R))=0$ (such as for the stacky modular curve), the filtration $[U]_n$ is just equivalent to the lower central series filtration of Example \ref{relmalex}, so the morphisms are weak equivalences. 
For general $X$, the towers will be different, but whenever the higher relative Malcev homotopy groups of $X$ vanish, the towers will converge to the same limit.  
\end{example}

\begin{remark}\label{modex2fromexhtpyrmk}
 To recover Example \ref{modularex2} from Example \ref{exhtpy}, we take  $S$ to be the Zariski closure of the image of the representation
\[
\pi_1^{\et}(\cM_{1,1},\bar{x}) \to \GL(\bT_{\ell,\bar{x}}\ten \Q) \by \prod_m \GL(\H^1(\Gamma, V_m))
\]
given by combining the monodromy representation on $\bT_{\ell,\bar{x}}\ten \Q$ with the pullbacks of the $G_{F,\Sigma}$-representations $\H^1(\Gamma, V_m)$. Then the Zariski closure $R$ of the image of $ \pi_1^{\et}(\bar{X},\bar{x})$ is just $\SL_2 \by \{1\}$, and the quotient $T:= S/R$ is the Zariski closure of the representation $G_{F,\Sigma}\to \bG_m(\Ql) \by \prod_m \GL(\H^1(\Gamma, V_m))$. The conditions of \cite[Theorem \ref{weiln-lfibrations}]{weiln} are then satisfied by construction.
\end{remark}

\begin{remark}[Algebraic monoids and weighted completion]
A variant of Example \ref{exhtpy} is given by taking  a Zariski-dense representation $\rho\co  \pi_1^{\et}(X,\bar{x}) \to S(\Ql)$ to a  pro-reductive pro-algebraic monoid $S$, with $R$ the Zariski closure of $\rho(\pi_1^{\et}(\bar{X},\bar{x}))$, and $T:= S/R$. Then  the theory of relative Malcev completion still works to give a homotopy fibre sequence
\[
 G(\bar{X}_{\et}, \bar{x})^{R,\mal} \to G(X_{\et}, \bar{x})^{S,\mal} \to G(\sO_{F,\Sigma,\et})^{T,\mal}
\]
of simplicial pro-algebraic monoids, and we may restrict to invertible elements ($G(\bar{X}_{\et}, \bar{x})^{R,\mal}\by_R R^{\by}$ etc.)   and proceed as before. 

For Example \ref{modularex2}, that would mean adapting Remark \ref{modex2fromexhtpyrmk} by taking $S$ to be 
the Zariski closure of the image of the representation
\[
\pi_1^{\et}(\cM_{1,1},\bar{x}) \to \End(\bT_{\ell,\bar{x}}\ten \Q) \by \prod_m \End(\H^1(\Gamma, V_m)^*).
\]
The group $R$ would still be $\SL_2$, and the obstruction spaces would be the same, but this gives a smaller sequence deriving them by ignoring data from irrelevant representations.  

As noted in \cite[Remark \ref{heid-monoidrk}]{heid}, the weighted completions of \cite{hainmatweight} relative to a pro-reductive group $S$ with central cocharacter $\chi \co \bG_m \to S$ can also be regarded as completions relative to a monoid, namely $S\by_{\chi,\bG_m}\bA^1$. This has the effect of excluding some, but not all, irrelevant representations in Example \ref{modularex2} (analogous to the distinction between effective motives and motives of non-negative weight).

Again, weighted completions generate the same obstructions as unweighted completion. In particular, for a relative curve $C \to T$  in characteristic $0$ and a generic point $\eta$ of $T$, we may consider the fibre sequence $C_{\bar{\eta}} \to C_{\eta} \to \Spec k(T)$ as in \cite{hainUniversal}. If we let $S$ be the Zariski closure of the natural representation from $\pi_1^{\et}(C, \bar{\eta}_C)$ to $\GL(\H^1_{\et}(C_{\bar{\eta}}, \Ql)^*)$ (or to $\End(\H^1_{\et}(C_{\bar{\eta}}, \Ql)^*)$, or to any algebraic monoid in between), then we have a fibre sequence
\[
 G((C_{\bar{\eta}})_{\et})^{1,\mal} \to G( (C_{\eta})_{\et})^{S,\mal} \to G( \Spec k(T)_{\et} )^{S,\mal}.
\]
The associated  obstruction and lifting data of the same type as the unipotent obstructions we encountered in Example \ref{unipex1}, with relative completion in this case just providing an alternative description.
These obstructions  (and particularly the second stage of the tower) are the main technical ingredient of  \cite{hainUniversal}.
\end{remark}

\section{Non-abelian reciprocity laws as obstruction maps}\label{recsn}

\subsection{Ad\'elic mapping spaces and compact supports}\label{adelicmapsn}

\begin{definition}
In the category of pro-simplicial sets, we set
\[
BG_{ \bA_F^{\in \Sigma}}:= \Lim_{\substack{T \subset \Sigma \\ \text{finite}}}  (\coprod_{v \in T}  BG_v \sqcup \coprod_{v \in \Sigma-T}  B(G_v/I_v)),
\]
where $G_v= G_{F_v} \subset G_F$, and $I_v \lhd G_v$ is the inertia subgroup;
beware that both the coproduct and and the limit are taken in the category of pro-simplicial sets. 
% We then set
% \[
%  BG_{ \bA_{F,\Sigma}}:= BG_{ \bA_F^{\in \Sigma}} \sqcup \coprod_{v \notin \Sigma} B(G_v/I_v).
% \]
\end{definition}
Note that there is a natural map %$BG_{ \bA_{F,\Sigma}} \to B G_{F,\Sigma}$.
$BG_{ \bA_F^{\in \Sigma}} \to B G_{F,\Sigma}$.

\begin{definition}\label{adeliccohodef1}
Given a finite abelian  group $U$ equipped with a continuous $G_{F,\Sigma}$-action, define
\begin{align*}
 \oR\Gamma(G_{\bA_F^{\in \Sigma}},U)&:= \prod'_{v}\oR\Gamma(G_{v},U)\\ 
&= \LLim_T ( \prod_{v\in T }\oR\Gamma(G_{v},U)\by\prod_{v \in \Sigma -T}\oR\Gamma(G_{v}/I_{v},U)),
\end{align*}
where $\oR\Gamma(G,-)$ denotes the continuous cohomology complex and $T$ ranges over all finite subsets of $\Sigma$ containing the places at which the action on $U$ is ramified.
% We also write
% \[
%  \oR\Gamma(\bA_{F,\Sigma},U):=\oR\Gamma(G_{\bA_F^{\in \Sigma}},U) \by \prod_{v \notin \Sigma} \oR\Gamma(G_v/I_v, U).
% \]
\end{definition}

\begin{definition}\label{adeliccohodef2}
Given a continuous profinite   $G_{F,\Sigma}$-representation $U$, define
\begin{align*}
\oR\Gamma(G_{\bA_F^{\in \Sigma}},U)&:=\oR\Lim_i\oR\Gamma(G_{\bA_F^{\in \Sigma}},U_i),
%\\   \oR\Gamma(\bA_{F,\Sigma},U)&:=\oR\Lim_i\oR\Gamma(\bA_{F,\Sigma},U_i),
\end{align*}
where the $U_i$ range over the finite Galois-equivariant quotients of $U$.

Similarly, given a continuous discrete torsion   $G_{F,\Sigma}$-representation $U$, define
\begin{align*}
\oR\Gamma(G_{\bA_F^{\in \Sigma}},U)&:=\LLim_i\oR\Gamma(G_{\bA_F^{\in \Sigma}},U_i),
%\\   \oR\Gamma(\bA_{F,\Sigma},U)&:=\LLim_i\oR\Gamma(\bA_{F,\Sigma},U_i),
\end{align*}
where the $U_i$ range over the finite Galois-equivariant subgroups of $U$.
\end{definition}

\begin{definition}\label{adeliccohodef3}
Given a continuous $G_{F,\Sigma}$-representation $V$ in finite-dimensional vector spaces over $\Ql$, define
\begin{align*}
 \oR\Gamma(G_{\bA_F^{\in \Sigma}},V)&:= \LLim_j \oR\Gamma(G_{\bA_F^{\in \Sigma}},V_j),
%\\
 %\oR\Gamma(\bA_{F,\Sigma},V)&:= \LLim_j \oR\Gamma(\bA_{F,\Sigma},V_j),
\end{align*}
where the $V_j$ range over the filtered direct system of all profinite subrepresentations of $V$. 

Given a continuous $G_{F,\Sigma}$-representation $V=\Lim_{\alpha}V_{\alpha}$ in profinite-dimensional vector spaces over $\Ql$, define
\begin{align*}
 \oR\Gamma(G_{\bA_F^{\in \Sigma}},V)&:= \oR\Lim_{\alpha} \oR\Gamma(G_{\bA_F^{\in \Sigma}},V_{\alpha}).
%\\
%\oR\Gamma(\bA_{F,\Sigma},V)&:= \oR\Lim_{\alpha} \oR\Gamma(\bA_{F,\Sigma},V_{\alpha}).
\end{align*}
\end{definition}
Note that for any $G_{F,\Sigma}$-equivariant lattice $\Lambda$ in a finite-dimensional $G_{F,\Sigma}$-representation $V$ over $\Ql$, the system $\{\ell^{-n}\Lambda\}_n$ of profinite subrepresentations is cofinal, so 
\[
% \oR\Gamma(\bA_{F,\Sigma},V)\simeq \oR\Gamma(\bA_{F,\Sigma},\Lambda)\ten_{\Zl}\Ql.
\oR\Gamma(G_{\bA_F^{\in \Sigma}},V)\simeq \oR\Gamma(G_{\bA_F^{\in \Sigma}},\Lambda)\ten_{\Zl}\Ql.
\]

\begin{remark}\label{adeliccohormk}
  Given a finite abelian group $U$ equipped with a continuous $G_{F,\Sigma}$-action, observe that for all $r$,
\begin{align*}
 &\map_{B G_{F,\Sigma}}( BG_{ \bA_F^{\in \Sigma}}, B(G_{F,\Sigma} \ltimes B^rU)) \simeq \\
&\LLim_{\substack{T \subset \Sigma \\ \text{finite}}} (\prod_{v \in T} \map_{B G_{F,\Sigma}}( BG_v, B(G_{F,\Sigma} \ltimes B^rU)) \by \prod_{v \in \Sigma-T} \map_{B G_{F,\Sigma}}(  B(G_v/I_v), B(G_{F,\Sigma} \ltimes B^rU))),
\end{align*}
so $\pi_i\map_{B G_{F,\Sigma}}( BG_{ \bA_F^{\in \Sigma}}, B(G_{F,\Sigma} \ltimes B^rU)) \cong \H^{r+1-i}(G_{\bA_F^{\in \Sigma}},U)$. 

% \begin{align*}
%  \map_{B G_{F,\Sigma}}( BG_{ \bA_{F,\Sigma}}, B(G_{F,\Sigma} \ltimes B^rU)) &\simeq \LLim_{\substack{T \subset \Sigma \\ \text{finite}}} (\prod_{v \in T} \map_{B G_{F,\Sigma}}( BG_v, B(G_{F,\Sigma} \ltimes B^rU)) \by \prod_{v \notin T} \map_{B G_{F,\Sigma}}(  B(G_v/I_v), B(G_{F,\Sigma} \ltimes B^rU))),
% \end{align*}
% so $\pi_i\map_{B G_{F,\Sigma}}( BG_{ \bA_{F,\Sigma}}, B(G_{F,\Sigma} \ltimes B^rU)) \cong \H^{r+1-i}(\bA_{F, \Sigma},U)$. 

The convention of Definition \ref{adeliccohodef3} ensures that this equivalence extends to profinite groups or (pro-)finite-dimensional $\Ql$-vector spaces $U$ (regarding $BU$ as a pro-simplicial set  or a (pro-)ind-pro-simplicial set). 
\end{remark}

 Beware that %$  BG_{ \bA_{F,\Sigma}}$ and 
$ BG_{ \bA_F^{\in \Sigma}}$ 
%are 
is
not necessarily the same as the \'etale homotopy %types  $(\Spec  \bA_{F,\Sigma})_{\et}$ and 
type $(\Spec  \bA_F^{\in \Sigma})_{\et}$. However, there is a map from the former %spaces 
to the profinite %completions of the latter spacws 
completion of the latter
(see Corollary \ref{adelemapcor}); on the level of fundamental groups this is just the observation that a finite lisse \'etale sheaf on $\bA_F^{\in \Sigma}$ is only ramified at $F_v$ for finitely many places in $v \in \Sigma$.%, and unraified outside $\Sigma$.

We may now adapt all the examples from \S \ref{rationalsn} to consider ad\'elic points instead of rational points. In particular:

\begin{example}[Nilpotent completion of $ \pi_1^{\et}(\bar{X})$]\label{nilpex2} 
Using the pro-simplicial set $BG_{ \bA_F}$, we may adapt Example \ref{nilpex1}.
 If $X$ is a Deligne--Mumford stack over $F$, and $\bar{X}:=X\ten_{F}\bar{F}$, with some geometric point $\bar{x}$, again consider the lower central series 
\[
 \Pi_n:= \pi_1^{\et}(X,\bar{x})/[\pi_1^{\et}(\bar{X}, \bar{x})]_{n+1},
\]
where we write $[\pi]_1:=\pi$, $[\pi]_{k+1}:= [\pi, [\pi]_k]$. Thus $\Pi_0=G_F$, and taking $Y=BG_{ \bA_F} $ in the tower of \S \ref{nilpsn}, we get  the non-abelian spectral sequence 
\[
 E_1^{s,t}= \H^{1+s-t}(  \bA_F, [\bar{\pi}]_s/[\bar{\pi}]_{s+1} ) \abuts  \pi_{t-s} \map_{BG_{F}}(BG_{ \bA_F}, B\Pi_{\infty} )
\]
of groups and sets, where we write $\bar{\pi}:= \pi_1^{\et}(\bar{X}, \bar{x})$. 

The reasoning above (without recourse to Corollary \ref{adelemapcor}) gives  a morphism of groupoids from $X(\bA_F)$ to the fundamental groupoid $\pi_f \map_{BG_{F}}(BG_{ \bA_F}, B\Pi_{\infty} )$, so the spectral sequence gives  obstructions to the existence of such ad\'elic points.

A variant of this construction is given by taking  $X$ smooth over $\sO_{F,\Sigma}$  admitting a smooth relative compactification. For $\bar{X}:=X\ten_{\sO_{F,\Sigma}}\sO_{\bar{F},\Sigma}$, we can then  take $\Pi_{\infty}$ to be the relative pro-$\ell$ completion of $\pi_1^{\et}(X,\bar{x}) $ over $G_{F,\Sigma}$  (with our convention that all $F$-prime factors of $\ell$ lie in $\Sigma$), giving a non-abelian spectral sequence 
\[
 E_1^{s,t}= \H^{1+s-t}(  \bA_F^{\in \Sigma}, ([\bar{\pi}]_s/[\bar{\pi}]_{s+1})\ten_{\hat{\Z}}\Zl ) \abuts  \pi_{t-s} \map_{BG_{F,\Sigma}}(BG_{ \bA_F^{\in \Sigma}}, B\Pi_{\infty} )
\]
with a morphism $X(\bA_F^{\in \Sigma})\to \pi_f \map_{BG_{F,\Sigma}}(BG_{ \bA_F^{\in \Sigma}}, B\Pi_{\infty} )$.
\end{example}

\begin{example}[Unipotent completion of $\pi_1^{\et}(\bar{X})$]\label{unipex2}
For unipotent ad\'elic obstructions, we can adapt Example \ref{unipex1}, taking  a smooth scheme $X$ over $\sO_{F,\Sigma}$ admitting a smooth relative compactification, with $\bar{X}:=X\ten_{\sO_{F,\Sigma}}\sO_{\bar{F},\Sigma}$ and $\bar{x}$ a geometric point.  Assume that we have a point $x \in X(\sO_{F,\Sigma})$ under $\bar{x}$ (if not, there are analogous statements using a $G_{F,\Sigma}$-equivariant set $\cB\subset X(\sO_{\bar{F},\Sigma})$ of basepoints instead), and  consider the lower central series 
\[
 \Pi_n:= G_{F,\Sigma} \ltimes  (\pi_1^{\et}(\bar{X},\bar{x})\ten \Ql/[\pi_1^{\et}(\bar{X}, \bar{x})\ten\Ql]_n),
\]
of the pro-unipotent Malcev completion $\pi_1^{\et}(\bar{X},\bar{x})\ten \Ql$.

Thus $\Pi_0=G_{F,\Sigma}$, and taking $Y=BG_{\bA_F^{\in \Sigma}}$, we get  a non-abelian spectral sequence 
\[
 E_1^{s,t}= \H^{1+s-t}(G_{\bA_F^{\in \Sigma}}, [\bar{\pi}\ten \Ql]_s/[\bar{\pi}\ten \Ql]_{s+1} ) \abuts  \pi_{t-s} \map_{BG_{F,\Sigma}}(BG_{\bA_F^{\in \Sigma}}, B\Pi_{\infty} )
\]
of groups and sets, where we write $\bar{\pi}:= \pi_1^{\et}(\bar{X}, \bar{x})$. As in Example \ref{nilpex2}, there is a natural  morphism  $X(\bA_F^{\in \Sigma})\to \pi_f \map_{BG_{F,\Sigma}}(BG_{F,\Sigma}, B\Pi_{\infty} )$ of groupoids, but the obstruction spaces are easier to calculate in this setting.
\end{example}

\begin{example}[Modular forms of level $1$]\label{modularex3}
As in Example \ref{modularex2}, let
  $X= \cM_{1,1}$ be the stacky modular curve, take $x \in X(\sO_{F,\Sigma})$, and consider the resulting pro-unipotent extension
\[
 G_{F,\Sigma} \ltimes( \SL_2(\Z)^{\SL_2, \mal}\by_{\SL_2(\Ql)}\SL_2(\Zl)) \to  G_{F,\Sigma} \ltimes \SL_2(\Zl),
\]
then set
\[
 \Pi_n :=  G_{F,\Sigma} \ltimes( (\SL_2(\Z)^{\SL_2, \mal}/[\Ru]_{n+1})\by_{\SL_2(\Ql)}\SL_2(\Zl)).
\]

Using Example \ref{modularex1}, a lift $BG_{ \bA_F^{\in \Sigma}} \to G_{F,\Sigma} \ltimes \SL_2(\Zl)$ of the homomorphism $BG_{ \bA_F^{\in \Sigma}} \to G_{F,\Sigma}$ is equivalent to giving $G_v$-representations $\Lambda_v$ of rank $2$ over $\Zl$ for $v \in \Sigma$, with determinant $\Zl(1)$, such that %$\Lambda_v$ is unramified for $v \notin S$ and 
for each $n$, there are only finitely many $v \in \Sigma$ with $\Lambda_v/\ell^n$ ramified. Write $\Lambda$ for the system $\{\Lambda_v\}_v$. 

%For the universal elliptic curve $f\co E \to X$, we have a $\Ql$-sheaf $\bV_1:=\oR^1 f_*\Ql$ of rank $2$ on $X$, giving a $G_{F,\Sigma}$-action on $\H^1(\SL_2(\Z),V_m) \cong \H^1(\bar{X}, S^m\bV_1)$. 

As in Example \ref{modularex1}, write $L_s:= \CoLie_s (\bigoplus_m \H^1(\SL_2(\Z),V_m) \ten S^m(\Lambda)(-m))$.  
The pro-unipotent generalisation of Proposition \ref{abpropunip} then combines with Examples \ref{SL2ex1} to  give a non-abelian spectral sequence 
\[
 E_1^{s,t}= \H^{1+s-t}(G_{\bA_F^{\in \Sigma}}, L_s^{*}) \abuts  \pi_{t-s} \map_{B(G_{F,\Sigma}\by_{\bG_m(\Zl)} \GL_2(\Zl))}(BG_{ \bA_F^{\in \Sigma}}, B\Pi_{\infty} ),
\]
where the map $BG_{ \bA_F^{\in \Sigma}}\to \GL_2(\Zl)$ is given by $\Lambda$. Note that  $E^1_{s,s+1}=0$ as $L_s$ is of non-zero weights. 

Now set $X_{(n)}:=\map_{B G_{F,\Sigma}}(BG_{ \bA_F^{\in \Sigma}},B\Pi_n)$; thus $X_{(0)}$ consists of  sets $\{\Lambda_v\}_v$ as above, conjugation by $\SL_2(\Zl)$ giving  equivalences, so $\pi_1(X_{(0)}, [\Lambda])$ consists of elements of $\SL_2(\Zl)$ commuting with the actions of the $G_v$ on $\Lambda$.  Since  $\pi_iX_{(n)}=0$ for $i>1$, we then have exact sequences
 \begin{align*}
  0 \to  \pi_1 X_{(n)} \to \pi_1 X_{(n-1)} \to
 \H^1(G_{\bA_F^{\in \Sigma}} , L_n^{*}) \to   \pi_0 X_{(n)} \to \pi_0 X_{(n-1)} \to \H^2(G_{\bA_F^{\in \Sigma}}, L_n^{*}),
 \end{align*}
with a map $X(\bA_F^{\in \Sigma}) \to  X_{(\infty)}$. Here,  $\pi_0X(\bA_F^{\in \Sigma})$ is the set of isomorphism classes of elliptic curves over $\bA_F^{\in \Sigma}$, and $\pi_1(X(\bA_F^{\in \Sigma}),x)$ the group of automorphisms of the elliptic curve $E_x$ over $\bA_F^{\in \Sigma}$.  

In other words, given a system $\Lambda=\{\Lambda_v\}_{v\in \Sigma}$ of rank $2$ local Galois representations  over $\Zl$ as above,  these sequences give a tower of obstructions to lifting $\Lambda$ to an elliptic curve over $ \bA_F^{\in \Sigma}$ with Tate module $\Lambda$, and characterise the ambiguity of the lift at each stage. As in Examples \ref{SL2ex1}, there is an entirely similar treatment for profinite completions of congruence subgroups $\Gamma \le {\SL_2(\Z)}$, replacing $\cM_{1,1} $ with the modular curve $Y_{\Gamma}$.
 \end{example}

\begin{example}[\'Etale homotopy types]\label{exhtpy2}
We now consider \'etale homotopy types in place of fundamental groups,  as in Example \ref{exhtpy}. Take a smooth Deligne--Mumford stack $X$ over $\sO_{F,\Sigma}$ admitting a smooth relative compactification,  and set $\bar{X}:=X\ten_{ \sO_{F,\Sigma}}\sO_{\bar{F},\Sigma}$. For  a geometric point $\bar{x}$ and a Zariski-dense representation $\rho\co  \pi_1^{\et}(X,\bar{x}) \to S(\Ql)$ to a  pro-reductive pro-algebraic group $S$, let $R$ be the Zariski closure of $\rho(\pi_1^{\et}(\bar{X},\bar{x}))$, and set $T:= S/R$. 

We then look at the pro-simplicial group $G(X_{\et}, \bar{x})$  associated to  the \'etale  topological type  $X_{\et} \in \pro(\bS)$.  If the  $G_{F,\Sigma}$-representation $\H^*(\bar{X}, V)$ is an extension of $T$-representations  for all $R$-representations $V$, then we may again set $\Pi_n$ to be the simplicial topological group given by the homotopy fibre product
\[
\Pi_n:= (G(X_{\et}, \bar{x})^{S,\mal}/[U]_{n+1})\by^h_{G(BG_{F,\Sigma})^{T,\mal} }G_{F,\Sigma},
\]
where $U= \Ru G(\bar{X}_{\et}, \bar{x})^{R,\mal}$. Note that since  $B\Pi_{\infty}$ is equipped with a map from $\bar{W}G(\bar{X}_{\et}, \bar{x}) $, Corollary \ref{adelemapcor}  gives a canonical  morphism 
\[
 X(\bA_F^{\in \Sigma}) \to \map_{B(G_{F,\Sigma})}(BG_{ \bA_F^{\in \Sigma}}, B\Pi_{\infty}).
\]
in the homotopy category of pro-ind-pro-simplicial sets.

We  then have a non-abelian spectral sequence 
\[
 E_1^{s,t} \abuts  \pi_{t-s} \map_{B(G_{F,\Sigma})}(BG_{ \bA_F^{\in \Sigma}}, B\Pi_{\infty} ),
\]
with
\[ 
 E_1^{s,t}= \begin{cases}
             \bH^{1+s-t}( \bA_F^{\in \Sigma}, [U]_s/[U]_{s+1}) & s \ge 1 \\
 \pi_{t}\map_{BT}( BG_{ \bA_F^{\in \Sigma}}, BS)  %\pi_{t}\map_{BG_{F,\Sigma}}( BG_{ \bA_F^{\in \Sigma}}, B(S(\Ql)\by_{T(\Ql)}G_{F,\Sigma})  ) 
& s=0,
            \end{cases}
\]
where $[U]_s/[U]_{s+1}$ is dual to  $ \CoLie_n ((\oR\Gamma(\bar{X},O(R))/\Ql)[1])$.
\end{example}

\subsection{Reciprocity laws}\label{recipsn}

The idea behind non-abelian reciprocity laws is to compare the towers of obstructions for rational and ad\'elic points, giving a relative obstruction tower for rational points over ad\'elic points.

\begin{definition}\label{cptCC}
Given a continuous $G_{F,\Sigma}$-representation $U$, we set
\[
 \oR\Gamma_c( G_{F,\Sigma},U):= \cocone( \oR\Gamma( G_{F,\Sigma}, U)\to \oR\Gamma(G_{\bA_F^{\in \Sigma}}, U)),
\]
where $U$ can be any of the types of representation considered in Definitions \ref{adeliccohodef1}--\ref{adeliccohodef3}.

\end{definition}

\subsubsection{Abelian Poitou--Tate duality}\label{abPT}

\begin{definition}
 Define a contravariant functor $(-)^{\vee}$ on the category of abelian groups by
\[
 A^{\vee}:= \Hom_{\Z}(A,\Q/\Z).
\]
\end{definition}

\begin{definition}
Define a contravariant functor $(-)^{\vee}(1)$ on the category of continuous $G_{F,\Sigma}$-representations in locally compact topological torsion abelian groups (in the sense of \cite{HoffmannSpitzweckLCA}) by
\[
 A^{\vee}(1):= \Hom_{\Z,\cts}(A,\mu_{\infty}).
\]
\end{definition}
Note that $(-)^{\vee}$ preserves the subcategory of finite representations, and interchanges profinite and discrete representations.

\begin{lemma}\label{PTfinitelemma}
 If $\Sigma$ is a finite set of finite places containing all primes dividing $\ell$, and $U$ a continuous pro-$\ell$ $G_{F,\Sigma}$-representation,   then there is a canonical equivalence
\[
 \oR\Gamma_c( G_{F,\Sigma},U)\simeq \oR\Gamma( G_{F,\Sigma},U^{\vee}(1))^{\vee}[-3].
\]

If $V$ is  a continuous $G_{F,\Sigma}$-representation in finite-dimensional vector spaces over $\Ql$, then we also have
\[
 \oR\Gamma_c( G_{F,\Sigma},V)\simeq \oR\Gamma( G_{F,\Sigma},V^*(1))^*[-3].
\]
\end{lemma}
\begin{proof}
The first statement is the formulation of Poitou--Tate duality given in \cite{limPTduality}, refining a homological isomorphism from \cite{nekovarSelmerComplexes}. For the second statement, take a $G_{F,\Sigma}$-equivariant lattice $\Lambda\subset V$, and then (writing $\Lambda^*:=\Hom_{\Zl}(\Lambda,\Zl)$),
\begin{align*}
 \oR\Gamma_c( G_{F,\Sigma},V)&\simeq \oR\Gamma_c( G_{F,\Sigma},\Lambda)\ten\Q \\
&\simeq \oR\Gamma( G_{F,\Sigma},\Lambda^{\vee}(1))^{\vee}[-3]\ten \Q\\
&\simeq  \oR\HHom_{\Zl}(\oR\Gamma( G_{F,\Sigma},\Lambda^*(1))\ten \Ql/\Zl, \Ql/\Zl)[-3]\ten \Q\\
&\simeq  \oR\HHom_{\Zl}(\oR\Gamma( G_{F,\Sigma},\Lambda^*(1)), \Zl)[-3]\ten \Q\\
&\simeq  \oR\HHom_{\Zl}(\oR\Gamma( G_{F,\Sigma},\Lambda^*(1)), \Ql)[-3],
\end{align*}
the last isomorphism following because $\H^*( G_{F,\Sigma},\Lambda^*(1))$ has finite rank, $\Sigma$ being finite. The result now follows because $V^*\cong \Lambda^*\ten \Ql$.
\end{proof}

\begin{lemma}\label{PTinfinitelemma}
  If $\Sigma$ is a possibly infinite set of finite places, and $U$ a continuous  $G_{F,\Sigma}$-representation in profinite abelian groups whose order is a unit outside $\Sigma$,   then there is a canonical equivalence
\[
 \oR\Gamma_c( G_{F,\Sigma},U)\simeq \oR\Gamma( G_{F,\Sigma},U^{\vee}(1))^{\vee}[-3],
\]
following the continuous cohomology conventions of Definition \ref{adeliccohodef2}.
\end{lemma}
\begin{proof}
When $U$ is finite, this is essentially the  Poitou--Tate duality of   \cite[1.4.10]{milneArithDuality}. In general, writing $U=\Lim_{\alpha}U_{\alpha}$ for $U_{\alpha}$ finite, we have
\begin{align*}
 \oR\Gamma_c( G_{F,\Sigma},U) &\simeq \oR\Lim_{\alpha}\oR\Gamma_c( G_{F,\Sigma},U_{\alpha})\\
&\simeq \oR\Lim_{\alpha} \oR\Gamma( G_{F,\Sigma},U^{\vee}_{\alpha}(1))^{\vee}[-3]\\
&\simeq (\LLim_{\alpha} \oR\Gamma( G_{F,\Sigma},U^{\vee}_{\alpha}(1)))^{\vee}[-3]\\
&= \oR\Gamma( G_{F,\Sigma},U^{\vee}(1))^{\vee}[-3].
\end{align*}
 %%This is essentially the statement of \cite{cesnaviciusPT}. When $\Sigma$ is the set of all finite places, this is the classical Poitou--Tate duality of  \cite[II.6.3]{galoisienne}.
\end{proof}

\begin{remark}
 If we wanted to extend Lemma \ref{PTinfinitelemma} to more general coefficients, we would have to pass to a larger category than the category $\cT\cT$ of locally compact topological torsion groups. The category $\cT\cT$ precisely consists of the Tate objects over the category of finite abelian groups in the sense of \cite{BraunlingGroechenigWolfsonTate}. Since $\oR\Gamma(G_{F,\Sigma},-)$ and $\oR\Gamma_c(G_{F,\Sigma},-)$ are functors from finite groups to complexes of Tate objects, their natural extension to  coefficients in $\cT\cT$ will take values in complexes of $2$-Tate objects over finite abelian groups (or equivalently Tate objects over $\cT\cT$), and Poitou--Tate duality will extend formally to that category.
\end{remark}

\subsubsection{Non-abelian reciprocity laws}\label{recexamplessn}

We may now adapt all the examples from \S \ref{rationalsn} to obtain obstructions to  ad\'elic points being rational points, with terms in the spectral sequence given by Galois cohomology $\H^*_c(G_{F,\Sigma},-)$ with compact supports. Since the coefficients we consider have negative weights,  the lower cohomology groups with compact supports tend to be small; when they vanish,  the obstruction towers have no ambiguity in the lift at each stage.

\begin{example}[Nilpotent completion of $\pi_1^{\et}(\bar{X})$]\label{nilpex3} 
 If $X$ is a Deligne--Mumford stack over $F$, and $\bar{X}=X\ten_{F}\bar{F}$, with some geometric point $\bar{x}$, then as in Examples \ref{nilpex1} and \ref{nilpex2} we may consider the lower central series 
\[
 \Pi_n:= \pi_1^{\et}(X,\bar{x})/[\pi_1^{\et}(\bar{X}, \bar{x})]_{n+1},
\]
where we write $[\pi]_1:=\pi$, and $[\pi]_{k+1}$ for the closure of $[\pi, [\pi]_k]$. Write $\Pi_{\infty}= \Lim_n \Pi_n$.

We then define  the tower $ \ldots X(\bA_F )_n \to X(\bA_F)_0= X(\bA_F)$ by the homotopy fibre products
\[
X(\bA_F)_n:= X(\bA_F)\by^h_{\map_{BG_{F}}(BG_{ \bA_F}, B\Pi_n ) }\map_{BG_{F}}(BG_{F}, B\Pi_n ),
\]
defined using the morphism $X(\bA_F)\to {\map_{BG_{F}}(BG_{ \bA_F}, B\Pi_{\infty} ) } $ from \S \ref{adelicmapsn}.

Taking homotopy fibres of the fibration sequences in    \S \ref{nilpsn}, we then get  a non-abelian spectral sequence 
\[
 E_1^{s,t}= \begin{cases}
             \H^{1+s-t}_c( G_{F}, [\bar{\pi}]_s/[\bar{\pi}]_{s+1} )& s \ge 1 \\
 \pi_{t}X(\bA_F) & s=0
            \end{cases}
  \abuts  \pi_{t-s} X(\bA_F)_{\infty} 
\]
of groups and sets, where we write $\bar{\pi}:= \pi_1^{\et}(\bar{X}, \bar{x})$. This comes from the exact couple
\[
  \xymatrix{
 \ldots   \ar[r] & \pi_*X( \bA_F)_s  \ar[r] &  %\pi_*X(\bA_F)_{s-1} \ar[r] \ar@{-->}[ld]^{\delta} &
 \ldots \ar@{-->}[ld]^{\delta} \ar[r] & \pi_*X(\bA_F)_1 \ar[r] &\pi_*X(\bA_F)_0\ar@{-->}[ld]^{\delta}\\
  &		\H^{1-*}_c(G_{F}, [\bar{\pi}]_s/[\bar{\pi}]_{s+1} )  \ar[u]  & %\H^{1-*}_c(G_{F} , [\bar{\pi}]_{s-1}/[\bar{\pi}]_{s})  \ar[u]& 
\ldots 
 & \H^{1-*}_c( G_{F}, \bar{\pi}/[\bar{\pi}]_2) \ar[u] &  \pi_*X(\bA_F) \ar@{=}[u],
  }
 \]
with $\delta$ of cohomological degree $+1$. 

 As in \S \ref{rationalsn}, we have a map $X(F) \to X(\bA_F)_{\infty}$, so the spectral sequence gives  obstructions to an ad\'elic point being rational. When $X$ is a scheme (or algebraic space), $\pi_0X(\bA_F)= \pi_*X(\bA_F)$ and $\pi_i (\bA_F)=0$ for $i>0$. 

By Lemma \ref{PTinfinitelemma}, 
 $\H^{1+s-t}_c( G_{F}, [\bar{\pi}]_s/[\bar{\pi}]_{s+1} )$ is isomorphic  to $\H^{2+t-s}(G_{F}, ([\bar{\pi}]_s/[\bar{\pi}]_{s+1} )^{\vee}(1))^{\vee}$. Thus  elements of $ \H^{1}(G_{F}, ([\bar{\pi}]_s/[\bar{\pi}]_{s+1} )^{\vee}(1))$ give obstructions to lifting points in $\pi_0X(\bA_F)$ to $ X(F)$, and the ambiguities of the lifts at each stage are dual to the groups $\H^{2}(G_{F}, ([\bar{\pi}]_s/[\bar{\pi}]_{s+1} )^{\vee}(1))$, which are often finite for weight reasons as in \cite{jannsenweights}. The higher homotopy groups $\pi_{\ge 2} X(\bA_F)_n$ are necessarily $0$, by vanishing of $\H^{\le 0}_c$.
\end{example}

\begin{remark}
 Since $[\bar{\pi}]_{n}/[\bar{\pi}]_{n+1}$ is contained in the centre of $ \bar{\pi}/[\bar{\pi}]_{n+1}$, it seems that the spectral sequence in Example \ref{nilpex3} can alternatively be obtained as an inverse  limit of the non-abelian Poitou--Tate exact sequence of  \cite[Theorem 168]{stixRationalPoints}.
\end{remark}

\begin{example}[Unipotent completion of $\pi_1^{\et}(\bar{X})$]\label{unipex3}
Examples \ref{unipex1} and \ref{unipex2} adapt along the lines of Example \ref{nilpex3}. 
Take  a smooth Deligne--Mumford $X$ over $\sO_{F,\Sigma}$ admitting a smooth relative compactification, with $\bar{X}:=X\ten_{\sO_{F,\Sigma}}\sO_{\bar{F},\Sigma}$ and $\bar{x}$ a geometric point.  Assume that we have a point $x \in X(\sO_{F,\Sigma})$ under $\bar{x}$ (if not, there are analogous statements using a $G_{F,\Sigma}$-equivariant set $\cB$ of basepoints instead).

Now set
\[
 \Pi_n:= G_{F,\Sigma} \ltimes  (\pi_1^{\et}(\bar{X},\bar{x})\ten \Ql)/[\pi_1^{\et}(\bar{X}, \bar{x})\ten\Ql]_n),
\]
and 
\[
X(\bA_F^{\in \Sigma})_n:= X(\bA_F^{\in \Sigma})\by^h_{\map_{BG_{F,\Sigma}}(BG_{ \bA_F^{\in \Sigma}}, B\Pi_n ) }\map_{BG_{F,\Sigma}}(BG_{F,\Sigma}, B\Pi_n ),
\]
to give a   non-abelian spectral sequence 
\[
 E_1^{s,t}= \begin{cases} \H^{1+s-t}_c(G_{F,\Sigma}, [\bar{\pi}\ten \Ql]_s/[\bar{\pi}\ten \Ql]_{s+1} ) & s \ge 1 \\
 \pi_{t}X(\bA_F^{\in \Sigma}) & s=0
            \end{cases}
\abuts  \pi_{t-s}X(\bA_F^{\in \Sigma})_{\infty}
\]
of groups and sets, where we write $\bar{\pi}:= \pi_1^{\et}(\bar{X}, \bar{x})$.

Lemma \ref{PTinfinitelemma} shows that $ \H^{1+s-t}_c(G_{F,\Sigma}, [\bar{\pi}\ten \Ql]_s/[\bar{\pi}\ten \Ql]_{s+1} )$ is isomorphic to $\H^{2+t-s}(G_{F,\Sigma}, ([\bar{\pi}]_s/[\bar{\pi}]_{s+1} )^{\vee}(1))^{\vee}\ten\Ql$. Since $X$ is smooth, the space $[\bar{\pi}\ten \Ql]_s/[\bar{\pi}\ten \Ql]$ is a pro-finite-dimensional Galois $\Ql$-representation of negative weights, %$\le -s$, 
so  the local monodromy weight conjectures (as in the Poitou--Tate dual form of \cite[Conjecture 6.3]{jannsenweights}) 
would imply $E^1_{s,s}=0$ for $s>0$, 
with the exact couple yielding the spectral sequence then degenerating to   exact sequences
 \begin{align*}
% %   0 \to  \pi_1 X_{n} \to \pi_1 X_{n-1} \to \H^2( G_{F,\Sigma}, L_n(1))^* \H^1_c(G_{F, \Sigma}, L_s^*)
0 \to   \pi_0 X(\bA_F^{\in \Sigma})_{(n)} \to \pi_0 X(\bA_F^{\in \Sigma})_{(n-1)} \to \H^2_c(G_{F, \Sigma}, [\bar{\pi}\ten \Ql]_s/[\bar{\pi}\ten \Ql]_{s+1}), %%\H^1(G_{F,\Sigma}, L_n(1))^*,
 \end{align*}
so the tower becomes a sequence of subsets.
\end{example}

\begin{example}[Modular forms of level $1$]\label{modularex4}
As in Examples \ref{modularex2} and \ref{modularex3}, let
  $X= \cM_{1,1}$ be the stacky modular curve, take $x \in X(\sO_{F,\Sigma})$, and 
set
\[
 \Pi_n :=  G_{F,\Sigma} \ltimes( (\SL_2(\Z)^{\SL_2, \mal}/[\Ru]_{n+1})\by_{\SL_2(\Ql)}\SL_2(\Zl)),
\]
where $\SL_2$ is here regarded as an algebraic group over $\Ql$.

We now write 
\[
 X(\bA_F^{\in \Sigma})_n:= X(\bA_F^{\in \Sigma})\by^h_{\map_{BG_{F,\Sigma}}(BG_{ \bA_F^{\in \Sigma}}, B\Pi_n ) }\map_{BG_{F,\Sigma}}(BG_{F,\Sigma}, B\Pi_n ).
\]
 Since $\Pi_0 =  G_{F,\Sigma} \ltimes \SL_2(\Zl)$, the space $ X(\bA_F^{\in \Sigma})_0$ consists of pairs $(x, \Lambda)$ with $x$ an ad\'elic point and $\Lambda$ a $G_{F,\Sigma}$-representation of rank $2$ over $\Zl$ with determinant $\Zl(1)$, together with an isomorphism $T_{\ell}E_{\bar{x}} \cong \Lambda$ of $BG_{ \bA_F^{\in \Sigma}}$-representations. 

Writing $L_s:= \CoLie_s (\bigoplus_m \H^1(\SL_2(\Z), V_m ) \ten S^m(\Lambda)(-m))$, Proposition \ref{abpropunip} and Examples \ref{SL2ex1} then  give a non-abelian spectral sequence 
\[
 E_1^{s,t}= \begin{cases}
             \H^{1+s-t}_c(G_{F, \Sigma}, L_s^*) &  s \ge 1 \\
 \pi_{t}X(\bA_F^{\in \Sigma})_0 & s=0
            \end{cases}
 \abuts  \pi_{t-s} X(\bA_F^{\in \Sigma})_{\infty}.
\]

In other words, given a global Galois representation $\Lambda$ and, for each $v \in \Sigma$,  a  local elliptic curve $E_v$ lifting each underlying $G_v$-representation, with constraints on ramification,      these sequences give a tower of obstructions to lifting $(\Lambda, \{E_v\}_{v\in \Sigma})$ to an elliptic curve $E$ over $ \sO_{F,\Sigma}$ with Tate module $T_{\ell}E(\bar{F})\ten\Q \cong \Lambda\ten \Q$ and localisations $E_v$; the sequences also characterise the ambiguity of the lift at each stage.

 As in Example \ref{SL2ex1}, the group $\H^1(\Gamma,V_m)(-m)$ consists of modular forms and cusp forms of weight $m+2$ and level $1$. 
Thus  $L_s$ is  a Galois $\Ql$-representation of weights $\ge s$, so it follows that $L_s^*$ is a pro-finite-dimensional Galois $\Ql$-representation of weights $\le -s$.  As in Example \ref{unipex3}, the local monodromy weight conjectures would cause the 
 exact couple yielding the spectral sequence to degenerate to the  exact sequences
 \begin{align*}
% %   0 \to  \pi_1 X_{n} \to \pi_1 X_{n-1} \to \H^2( G_{F,\Sigma}, L_n(1))^* \H^1_c(G_{F, \Sigma}, L_s^*)
0 \to   \pi_0 X(\bA_F^{\in \Sigma})_{(n)} \to \pi_0 X(\bA_F^{\in \Sigma})_{(n-1)} \to \H^2_c(G_{F, \Sigma}, L_s^*) %%\H^1(G_{F,\Sigma}, L_n(1))^*,
 \end{align*}
equipped with a map $X(\sO_{F,\Sigma}) \to  X(\bA_F^{\in \Sigma})_{(\infty)}$. 
\end{example}

As in Examples \ref{SL2ex1}, there is an entirely similar treatment for congruence subgroups $\Gamma \le {\SL_2(\Z)}$, replacing $\cM_{1,1} $ with the modular curve $Y_{\Gamma}$.
If we instead started from a representation $\Lambda$ over $\hat{\Z}$, relative Malcev completion of ${\SL_2(\Z)}$ over  $\SL_2 \by \SL_2(\hat{\Z})$ as in Example \ref{modularex2b} would give rise to reciprocity laws associated to modular forms of all levels. Meanwhile, relative Malcev completion of ${\SL_2(\Z)}$ over  $\SL_2(\hat{\Z})$ as in Example \ref{modularex2c} gives rise to reciprocity laws associated to weight $2$ modular forms of all levels. 

\begin{remark}\label{modularBM}
 We may write 
\[
 \H^{i}_c(G_{F, \Sigma}, L_s^*) \cong  \Lie(n)\ten^{S_n} \H^{i}_c(G_{F, \Sigma},((\bigoplus_m \H^1(\SL_2(\Z), V_m ) \ten S^m(\Lambda)(-m))^{\ten n})^*)
\]
 As in Example \ref{modularex2}, we may then consider the sheaf $\bT_{\ell}$ of relative Tate modules on $Y_{\Gamma}$, with $\oR q_* \bT_{\ell}\ten \Q[1] \simeq \oR^1 q_*\bT_{\ell}\ten \Q \cong \H^1(\SL_2(\Z), V_m )$, for the structure map $q \co Y_{\Gamma} \to \Spec \sO_{F,\Sigma}$. 
Applying Poitou--Tate duality in the form of Lemma \ref{PTinfinitelemma} to this $\Zl$-lattice then gives
\begin{align*}
  &\oR\Gamma_c(G_{F, \Sigma},((\bigoplus_m \H^1(\SL_2(\Z), V_m ) \ten S^m(\Lambda)(-m))^{\ten s})^*)\\
% &= 
% (\oR\Lim_M\oR\Gamma_c(G_{F, \Sigma},((\bigoplus_{m\le M} \oR^1 q_*S^m\bT_{\ell} \ten S^m(\Lambda)^{*})^{\ten s})^*))\ten \Q\\
% &\simeq (\LLim_M\oR\Gamma(G_{F, \Sigma}, (\bigoplus_{m\le M} \oR^1 q_*S^m\bT_{\ell} \ten S^m(\Lambda)^{*})^{\ten s}\ten \mu_{\ell^{\infty}}))^{\vee}\ten \Q[3]\\
% &\simeq \oR\Gamma(G_{F, \Sigma}, (\bigoplus_{m} \oR^1 q_*S^m\bT_{\ell} \ten S^m(\Lambda)^{*})^{\ten s}\ten \mu_{\ell^{\infty}})^{\vee}\ten \Q[3]\\
&\simeq \oR\Gamma(G_{F, \Sigma}, (\bigoplus_{m\ge 1} \oR q_*S^m\bT_{\ell} \ten S^m(\Lambda)(-m))^{\ten s}\ten \mu_{\ell^{\infty}})^{\vee}\ten \Q[3+s]\\
&\simeq \oR\Gamma(X^s, \bigotimes_{i=1}^s (\bigoplus_{m\ge 1}\pr_i^*S^m\bT_{\ell}\ten q^*S^m(\Lambda)(-m))\ten \mu_{\ell^{\infty}})^{\vee}\ten\Q[3+s],
\end{align*}
providing an expression for $E_1^{s,t}$ as a summand of $\H^{2+t}(X^s, \cdots\ten \mu_{\ell^{\infty}} )^{\vee}\ten \Q$. As we will see in Example \ref{exhtpyBrauerManin}, the $s=1$ case is a part of the Brauer--Manin obstruction, divisible elements in cohomology giving rise to obstructions.

By \cite{HelmVoloch}, for $\Sigma$ cofinite and $\rho\co  G_{F, \Sigma} \to \SL_2(\hat{\Z})$, non-emptiness of $\cM_{1,1}(\bA_{\Q}^{\in \Sigma})_{\rho}$ implies non-emptiness of $X(\Z_{\Sigma})_{\rho}$. The variants of Example \ref{modularex4} for relative Malcev completions of ${\SL_2(\Z)}$ over  $\SL_2(\hat{\Z})$ or over $\SL_2(\Ql) \by \SL_2(\hat{\Z})$ should then help to identify $X(\Z_{\Sigma})_{\rho} \subset \cM_{1,1}(\bA_{\Q}^{\in \Sigma})_{\rho} $. %Following \cite{stollFiniteDescentObs}, we would expect the first map in this tower (a form of pro-\'etale Brauer--Manin obstruction) to be effective in cutting out the rational points.  
\end{remark}

%%\cite{HelmVoloch} say that for $\Sigma$ cofinite and $\rho\co \Gamma  \to \SL_2(\hat{\Z})$, nonemptiness of $X(\bA_{\Q}^{\in \Sigma})_{\rho}$ implies non-emptiness of $X(\Z_{\Sigma})_{\rho}$. For non-CM curves, they would even have a uniqueness result, so we could compare given data $E_p$ with the localisations $E_{\Q_p}$ of the global curve produced. For CM, $E,E'$ with two isogenies of coprime degree from $E$ to $E'$ would give the same data. In the non-CM case you get uniqueness using \hat{Z}. If you use only one prime l, then one isogeny is enough to break uniqueness and therefore it's more common.

%%also see \cite{PatrikisVolochZarhin}: an attempt at a generalisation of the previous paper to abelian varieties but depends on a bunch of big conjectures.

\begin{example}[Relative Malcev \'etale homotopy types]\label{exhtpy3}
As in Examples \ref{exhtpy} and  \ref{exhtpy2}, we may consider \'etale homotopy types in place of fundamental groups. Take a smooth Deligne--Mumford stack $X$  over $\sO_{F,\Sigma}$ admitting a smooth relative compactification, a geometric point $\bar{x}$ and a Zariski-dense representation $\rho\co  \pi_1^{\et}(X,\bar{x}) \to S(\Ql)$ to a  pro-reductive pro-algebraic group $S$, let $R$ be the Zariski closure of $\rho(\pi_1^{\et}(\bar{X},\bar{x}))$, and set $T:= S/R$. 

Now set $\Pi_n$ to be the simplicial topological group given by the homotopy fibre product
\[
\Pi_n:= (G(X_{\et}, \bar{x})^{S,\mal}/[U]_{n+1})\by^h_{G(BG_{F,\Sigma})^{T,\mal} }G_{F,\Sigma},
\]
where $U= \Ru G(\bar{X}_{\et}, \bar{x})^{R,\mal}$. 

The formula of Example \ref{modularex4} 
% \[
%  X(\bA_F^{\in \Sigma})_n:= X(\bA_F^{\in \Sigma})\by^h_{\map_{BG_{F,\Sigma}}(BG_{ \bA_F^{\in \Sigma}}, B\Pi_n ) }\map_{BG_{F,\Sigma}}(BG_{F,\Sigma}, B\Pi_n ).
% \]
then gives a  tower of spaces $\{X(\bA_F^{\in \Sigma})_n\}_n$
and an associated non-abelian spectral sequence 
\[
 E_1^{s,t} \abuts  \pi_{t-s} (X(\bA_F^{\in \Sigma})\by^h_{\map_{B(G_{F,\Sigma})}(BG_{ \bA_F^{\in \Sigma}}, \bar{W}\Pi_{\infty} )} \map_{B(G_{F,\Sigma})}(BG_{F,\Sigma}, \bar{W}\Pi_{\infty} )),
\]
with
\[ 
 E_1^{s,t}= \begin{cases}
             \bH^{1+s-t}_c( G_{F,\Sigma}, [U]_s/[U]_{s+1}) & s \ge 1 \\
 \pi_{t} X(\bA_F^{\in \Sigma})_0 
& s=0,
            \end{cases}
\]
where $[U]_s/[U]_{s+1}$ is dual to  $ \CoLie_s ((\oR\Gamma(\bar{X}, O(R))/\Ql)[1])$
 and 
\[
 X(\bA_F^{\in \Sigma})_0=(X(\bA_F^{\in \Sigma})\by^h_{\map_{BT}( BG_{\bA_F^{\in \Sigma}}, BS)}\map_{BT}(B(G_{F,\Sigma}), BS).
\]
\end{example}

\begin{remark}
As in \cite[Theorem \ref{weiln-laff}]{weiln}, Lafforgue's theorem and Esnault--Kerz (\cite[Theorem VII.6 and Corollary VII.8]{La} and \cite{EsnaultKerz}) imply that the ind-lisse sheaf 
 $\rho^{-1}O(R)$ on $\bar{X}$ is pure of weight $0$. If $\bar{X}$ is smooth and proper, \cite[Corollary \ref{weiln-wgtexistspin}]{weiln} then implies that the group $ \H^{-i}([U]_s/[U]_{s+1})$ in Example \ref{exhtpy3} is pure of weight $-i-s$. 

The obstruction spaces for \'etale homotopy sections $\pi_0X(\bA_F^{\in \Sigma})_{(\infty)}$ are given in the spectral sequence by the terms $E_1^{s,s-1}$.
Assuming that $\rho$ is of geometric origin, the local monodromy weight conjectures (as in \cite[Conjecture 6.3]{jannsenweights}) 
 would imply that the groups $\H^1_c$ vanish, so the  only non-trivial contributions to $E_1^{s,s-1}$ come from  
\[
 \H^{2}_c( G_{F,\Sigma}, (\CoLie_s\H^{1}(\bar{X},O(R)))^*)
\]
 as in Example \ref{unipex3}, and from 
\[
\H^{3}_c( G_{F,\Sigma},  (\H^{2}(\bar{X},O(R))\ten \CoLie_{s-1}\H^{1}(\bar{X},O(R)))^*).
\]

The latter group can only be only non-zero for $s=1$, when $\H^2(\bar{X},O(R))$ contains copies of the Tate motive, in which case the reciprocity map is detecting the Brauer--Manin obstruction of a pro-\'etale covering whose geometric fibres are $\rho(\pi_1^{\et}(\bar{X},\bar{x}))$-torsors as in Example \ref{exhtpyEtBrauerManin} below. These copies of the Tate motive then generate a large contribution $\H^{2}_c( G_{F,\Sigma}, \H^2(\bar{X},O(R))^*)$ to the  $E_1^{1,1}$ term, producing an ambiguity in the lift much larger than the new obstruction, meaning the map $X(\bA_F^{\in \Sigma})_1 \to X(\bA_F^{\in \Sigma})$ would then be far from  injective.
%%Extreme example is given by projective space $\bP^n$, with $\H^2(\bar{X},\Ql) \cong \Ql(1)$. Since the Brauer group is $0$, the resulting  obstruction $X(\bA_F) \to \H^3_c(F, \Ql(1))\cong \Ql$ is $0$, and $X(\bA_F)_{\infty} becomes an $\H^2_c(F, \Ql(y21))$-torsor over $ X(\bA_F)$ ($\H^1_c$ vanishes for weight reasons), so there are very many more rational homotopy points than global points. Effective for non-trivial Brauer--Severi varieties, though, as $X(\bA_F)_1=\emptyset$. I think $\H^2(X, \mu_{\infty})=\Q/\Z$ for all, but image of $\Pic$ different, so differnet Brauer. 

\end{remark}

%%for reference, $\H^2(F)$ should have no divisible elts  for weights $\ge -1$, so $\H^j(X,\mu_{\infty})^{\vee}\ten \Q=\H^1(F, \H^{j-1}(\bar{X}))^{\vee}\ten \Q$ for $j> 2$. Low degree terms: $\H^2(X,\mu_{\infty})$ has contribs from $\H^2(F, \H^0(\bar{X}, \mu_{\infty})$, $\H^0(F, \H^2(\bar{X}, \mu_{\infty})$. [Never have to worry about $\H^1(X,\mu_{\infty})=\H^1(F,\H^0(\bar{X},\mu_{\infty})$ and $\H^0=0$.

\subsubsection{Brauer--Manin obstructions}\label{BMsn}

We now look at  Example \ref{exhtpy3} and analogous completions of \'etale homotopy types, giving rise to obstruction towers refining the non-abelian reciprocity laws by incorporating higher homotopical information.
A common feature is that the first obstruction map in the tower is just the Brauer--Manin obstruction, or related (pro-)\'etale refinements in the case of relative completion. Because the higher obstructions induce non-abelian reciprocity laws, they will be non-trivial in any case where the higher reciprocity maps of \cite{narec1} are non-zero on the relevant Brauer--Manin set. %\ref{exhtpyBrauerManin}  Example \ref{nilpex3} 

If $O(R)_{\Zl}$ is a $\pi_1^{\et}(\bar{X})$-equivariant $\Zl$-form for the ring $O(R)$ of functions on the reductive group featuring in Example \ref{exhtpy3}, then we may use Poitou--Tate duality to rewrite the term $E_1^{1,t}$ as  
\begin{align*}
 \bH^{2-t}_c( G_{F,\Sigma}, [U]_1/[U]_2 ) &\cong \bH^{2+t}( G_{F,\Sigma},\oR\Gamma(\bar{X}, O(R)_{\Zl}\ten \mu_{\ell^{\infty}})/\mu_{\ell^{\infty}} )^{\vee}\ten \Q\\
 &\cong \begin{cases}
	  \H^{2}_{\et}( X, O(R)_{\Zl}\ten \mu_{\ell^{\infty}})/\H^2(G_{F,\Sigma},\mu_{\ell^{\infty}})^{\vee} & t=0\\
         \H^{2+t}_{\et}( X, O(R)_{\Zl}\ten \mu_{\ell^{\infty}})^{\vee} & t>0;
        \end{cases}
 \end{align*}
when $R=1$ (unipotent completion of the geometric fibre), we have $O(R)_{\Zl}=\Zl$, and the first obstruction map $d_1 \co E_1^{0,0} \to E_1^{1,0}$ is the rationalised Brauer--Manin obstruction
\[
 \pi_0X(\bA_F^{\in \Sigma})\to (\H^{2}_{\et}( X, \mu_{\ell^{\infty}})/\H^2(G_{F,\Sigma},\mu_{\ell^{\infty}}))^{\vee}\ten \Q.
\]

\begin{remark}\label{LieBMremark}
 We may write $E_1^{s,t}$ as cohomology of a complex defined in terms of the Lie operad and the complexes $ \oR\Gamma( X^n, \mu_{\ell^{\infty}})^{\vee}\ten\Q$ for $n \le s$. 
In particular, 
\[
E_1^{2,t} \cong \H^{2+t}\Tot(\oR\Gamma(X^2, \mu_{\ell^{\infty}})^{\vee}\ten \Q/S_2 \xra{\pr_{1*}-\pr_{2*}} \oR\Gamma(X, \mu_{\ell^{\infty}})^{\vee}\ten \Q), 
\]
where $S_2$ acts by switching the factors in $X^2$. For $R \ne 1$,  the expression in Remark \ref{modularBM} for modular curves generalises whenever $\H^{>0}(\bar{X},\Ql)=0$, but usually there are extra factors reflecting the difference between reduced and non-reduced cohomology.
\end{remark}

Taking nilpotent completion instead of unipotent completion gives the following:

\begin{example}[\'Etale homotopy types and the Brauer--Manin obstruction]\label{exhtpyBrauerManin}
%As in Examples \ref{exhtpy}, \ref{exhtpy2} and \ref{exhtpy3}, we may %refine Example \ref{nilpex3}  by  consider \'etale homotopy types in place of fundamental groups. 
Take a smooth Deligne--Mumford stack $X$  over $\sO_{F,\Sigma}$ admitting a smooth  relative compactification, and a geometric point $\bar{x}$. 
Applying  relative pro-$\Sigma$  completion over $G_{F,\Sigma}$ levelwise (cf. \cite[\S \ref{weiln-profinitesn}]{weiln}) 
 to  the pro-simplicial group $G(X_{\et}, \bar{x})$   of Example \ref{exhtpy} gives a  pro-(finite simplicial group) $\hat{G}(X_{\et}, \bar{x})$ as in the proof of Proposition \ref{abprop}; up to homotopy, this is independent of the choices made, by \cite[Proposition \ref{weiln-Lcohoweak}]{weiln}. When  $\Sigma$ is the set of all primes (corresponding to $\sO_{F,\Sigma}=F$), note that $\hat{G}(X_{\et}, \bar{x})$ is just the profinite completion of $G(X_{\et}, \bar{x})$. 

We  now refine Example \ref{nilpex3} by  considering  relative   pro-nilpotent completions of the whole profinite homotopy type $\hat{G}(X_{\et}, \bar{x})$ instead of the fundamental group. For completions relative to  $G_{F,\Sigma}$,  we set  $K:= \ker(\hat{G}(X_{\et}, \bar{x})\to G_{F,\Sigma})$ and 
\[
\Pi_n:= \hat{G}(X_{\et}, \bar{x})/[K]_{n+1}, 
\]
which is a pro-(finite simplicial group). 

We then construct a tower $ \ldots \to X(\bA_F^{\in \Sigma} )_1 \to  X(\bA_F^{\in \Sigma})_0= X(\bA_F^{\in \Sigma})$ of homotopy fibre products
\[
X(\bA_F^{\in \Sigma})_n:= X(\bA_F^{\in \Sigma})\by^h_{\map_{BG_{F,\Sigma}}(BG_{ \bA_F^{\in \Sigma}}, \bar{W}\Pi_n ) }\map_{BG_{F,\Sigma}}(BG_{F,\Sigma}, \bar{W}\Pi_n ),
\]
defined using the morphism $X(\bA_F^{\in \Sigma})\to {\map_{BG_{F,\Sigma}}(BG_{ \bA_F^{\in \Sigma}}, B\Pi_{\infty} ) } $ from \S \ref{adelicmapsn} and Corollary \ref{adelemapcor}.

This gives a non-abelian spectral sequence 
\[
 E_1^{s,t}= \begin{cases}
             \bH^{1+s-t}_c( G_{F,\Sigma}, [K]_s/[K]_{s+1} )& s \ge 1 \\
 \pi_{t}X(\bA_F^{\in \Sigma}) & s=0
            \end{cases}
  \abuts  \pi_{t-s} X(\bA_F^{\in \Sigma})_{\infty}, 
\]
where we regard the simplicial abelian groups $ [K]_s/[K]_{s+1}$ as chain complexes. 

Nielsen--Schreier implies that the simplicial group $K$ is given levelwise by profinite completions of free groups, so the $s=1$ term is given by 
$
 [K]_1/[K]_2 \simeq (G(\bar{X}_{\et}, \bar{x})^{\wedge_{\Sigma}})^{\ab}, 
$
which is just the reduced homology complex of $\bar{X}$ with $\prod_{\ell}\Zl$ coefficients, where the product runs over those primes $\ell$ which are units in $\sO_{F,\Sigma}$.
Poitou--Tate duality in the form of Lemma \ref{PTinfinitelemma} applied to the complexes $ \hat{G}(\bar{X}_{\et}, \bar{x})^{\ab}$ thus gives 
\begin{align*}
 \bH^{2-t}_c( G_{F,\Sigma}, [K]_1/[K]_2 ) %&\cong \bH^{2+t}( G_{F,\Sigma},\oR\Gamma(\bar{X}, \mu_{\infty})/\mu_{\infty} )^{\vee}\\
 &\cong \begin{cases}
	  \prod_{\ell}(\H^{2}_{\et}( X,  \mu_{\ell^{\infty}})/\H^2(G_{F,\Sigma}, \mu_{\ell^{\infty}}))^{\vee} & t=0\\
        \prod_{\ell} \H^{2+t}_{\et}( X,  \mu_{\ell^{\infty}})^{\vee} & t>0,
        \end{cases}
 \end{align*}
where $\ell$ runs over all primes which are units in $\sO_{F,\Sigma}$ and we follow the usual convention for continuous cohomology, regarding $\mu_{\ell^{\infty}}$ as the ind-sheaf $\LLim_{n\in \N}  \mu_{\ell^n}$.

If we set $\Br_{\Sigma}(X):= \im(\bigoplus_{\ell}\H^{2}_{\et}( X,  \mu_{\ell^{\infty}}) \to \H^{2}_{\et}( X, \bG_m))$ to be the $\Sigma$-torsion cohomological Brauer group, then 
 the first obstruction map $d_1 \co E_1^{0,0} \to E_1^{1,0}$ is thus the map
\[
 \pi_0X(\bA_F^{\in \Sigma})\to \prod_{\ell}(\H^{2}_{\et}( X,  \mu_{\ell^{\infty}})/\H^2(G_{F,\Sigma}, \mu_{\ell^{\infty}}))^{\vee},
\]
induced by the natural map 
% 
% which sends $x \in X(\bA_F^{\in \Sigma})$ to the composition 
% \[
%  \H^{2}_{\et}( X, \bigoplus_{\ell} \mu_{\ell^{\infty}}) \xra{x^*} \H^{2}(\bA_F^{\in \Sigma},\bigoplus_{\ell} \mu_{\ell^{\infty}}) \xra{\pd} \H^3_c(G_{F,\Sigma}, \bigoplus_{\ell} \mu_{\ell^{\infty}})\cong \bigoplus_{\ell}\Ql/\Zl
% \]
% (so vanishes on $\pi_0X(\sO_{F,\Sigma})$). This is essentially 
$
 \mathrm{BM}_{\Sigma} \co \pi_0X(\bA_F^{\in \Sigma})\to \Br_{\Sigma}(X)^{\vee},
$
which is just the Brauer--Manin obstruction 
of \cite{ManinBrauer} when  $\sO_{F,\Sigma}= F$.

% % which sends $x \in X(\bA_F^{\in \Sigma})$ to the composition 
% % \[
% %  \H^{2}_{\et}( X, \bG_m) \xra{x^*} \H^{2}(\bA_F^{\in \Sigma},\bG_m) \xra{\pd} \H^3_c(G_{F,\Sigma}, \bG_m)\cong \Q/\Z,
% % \]
% % and thus vanishes on $\pi_0X(\sO_{F,\Sigma})$.

Writing $\pi_0X(\bA_F^{\in \Sigma})^{\Br_{\Sigma}}$ for the kernel of $\mathrm{BM}_{\Sigma}$, we thus have
\[
 \pi_0X(\bA_F^{\in \Sigma})^{\Br_{\Sigma}}= \im(\pi_0X(\bA_F^{\in \Sigma} )_1 \to \pi_0X(\bA_F^{\in \Sigma} ) )
\]
for the tower above, and the later pages of the spectral sequence give obstructions to lifting further up the tower. Beware, however, that when $E_1^{s,s}\ne 0$, the lifts are not unique at each stage; in particular if a point lies in the kernel of $\mathrm{BM}_{\Sigma}$, we have a  $\prod_{\ell} \H^{3}_{\et}( X, \mu_{\ell^\infty})^{\vee}$-torsor of possible choices on which to apply the secondary obstruction. 

 When $X$ is an algebraic space rather than a stack, we have $\pi_0X(\bA_F^{\in \Sigma} )=X(\bA_F^{\in \Sigma} )$, and may simply write $X(\bA_F^{\in \Sigma})^{\Br_{\Sigma}} $ for the image of $\pi_0X(\bA_F^{\in \Sigma} )_1$.
\end{example}

\begin{remark}\label{LieBMrmk2}
Because the simplicial pro-group $K$ of Example \ref{exhtpyBrauerManin} is given levelwise by pro-$\Sigma$ completions of free groups, the Magnus embedding (applied to profinite groups as in \cite{wickelgrenLCS}) gives an isomorphism $[K]_s/[K]_{s+1} \cong \hat{\Lie}_s(K^{\ab})$, where $\bigoplus_{s\ge 1} \Lie_s$ is the free Lie algebra functor, graded by bracket length, and $\hat{\Lie}_s$ the profinite completion of $\Lie_s$, applied levelwise to the simplicial abelian group. These functors are homotopy invariant when applied to chain complexes of projective modules via the Dold--Kan correspondence, but are not easy to calculate; they give the terms arising in the unstable Adams spectral sequence.

Over $\Q$, the functor $\bigoplus_s \Lie_s$ corresponds via the Dold--Kan correspondence to the free Lie algebra functor on chain complexes.
Thus the spaces $E_1^{s,t}\ten \Q$ are much simpler to describe in terms of free Lie algebras, but they correspond to the obstructions for the unipotent completion of Example \ref{exhtpy3} (with $R=1$). 
\end{remark}

We are now in a position to compare Kim's non-abelian reciprocity laws with the Brauer--Manin obstruction. Restricting to a single prime $\ell$ would give a similar statement for the $\ell$-torsion part of the Brauer--Manin obstruction.  
\begin{proposition}\label{MKBM}
If the natural maps 
\[
 \H^2_{\cts}(\pi_1^{\et}(\bar{X})/[\pi_1^{\et}(\bar{X})]_{n+1}, \Ql/\Zl)\to \H^{2}_{\et}( \bar{X},\Ql/\Zl)
\]
is surjective for all primes $\ell$ which are units in $\sO_{F,\Sigma}$, then
the image of the map $X(\bA_F^{\in \Sigma})_n \to X(\bA_F^{\in \Sigma})$ from Example \ref{nilpex3} is contained in the Brauer--Manin set $X(\bA_F^{\in \Sigma})^{\Br_{\Sigma}}$.
%%prob sufficient to have $K(\pi,1)$ and gend by relevant Massey prods, except that $\Q/\Z$ coeffs complicate things.
\end{proposition}
\begin{proof}
Take a free pro-simplicial resolution $\tilde{P}$ of $P:=\pi_1^{\et}(\bar{X})^{\wedge_{\Sigma}}/[\pi_1^{\et}(\bar{X})^{\wedge_{\Sigma}} ]_{n+1}$, and observe that the cofibrancy of $G(\bar{X}_{\et})$ ensures that the natural map $G(\bar{X}_{\et})\to P$ lifts to a map $ G(\bar{X}_{\et})\to\tilde{P}$, unique up to homotopy.

Since a point of $X(\bA)_n$ incorporates the datum of a $P$-valued Galois representation, the
composite map
\[
 X(\bA)_n \to X(\bA) \to \bH^2_c(G_{F,\Sigma}, (G(\bar{X}_{\et})^{\wedge_{\Sigma}})^{\ab}) \to \H^2_c(G_{F,\Sigma} , P^{\ab})
\]
is necessarily $0$. The kernel of the middle map is the $\Sigma$-torsion Brauer--Manin set as in Example \ref{exhtpyBrauerManin}, and 
via Poitou--Tate duality we can rewrite the final map as
\[
 \prod_{\ell}\bH^2(G_{F,\Sigma}, \oR \tilde{\Gamma}_{\et}(\bar{X},  \mu_{\ell^{\infty}}))^{\vee} \to \prod_{\ell}\bH^2(G_{F,\Sigma}, \oR \tilde{\Gamma}_{\cts}(P,  \mu_{\ell^{\infty}}))^{\vee}, %%2 dual to 1, but group coho a degree higher, so 2=(3-2)+1.
\]
where $\oR\tilde{\Gamma}$ denotes the reduced cohomology complex.

It suffices to show that this map is injective, or equivalently that its dual is surjective. This will follow from the Leray spectral sequences provided the maps
\[
 \H^i_{\cts}(P, \Ql/\Zl)\to \H^{i}_{\et}( \bar{X},\Ql/\Zl)
\]
are  isomorphisms for $i=1$ and surjective for $i=2$. The first condition is automatic and the second is our hypothesis.
\end{proof}

Considering the relative merits of the higher Brauer--Manin obstructions of Example \ref{exhtpyBrauerManin} and the non-abelian reciprocity laws of Example \ref{nilpex3}, the latter generally avoid ambiguity of lifts to the higher stages of the tower, but converge more slowly.

\begin{example}[\'Etale Brauer--Manin obstructions]\label{exhtpyEtBrauerManin}
While Example \ref{exhtpyBrauerManin} considered completions of the \'etale  homotopy type $\hat{G}(X_{\et}, \bar{x})$ relative to  $G_{F,\Sigma}$, it also makes sense to consider completions with respect to larger quotients  $P$ of $\pi_0\hat{G}(X_{\et}, \bar{x})$ over $G_{F,\Sigma}$ (i.e. relative pro-$\Sigma$ quotients $P$ of  $\pi_1^{\et}(X, \bar{x})$ over $G_{F,\Sigma}$). We can write $K:= \ker (\hat{G}(X_{\et}, \bar{x})\to P)$, and set  $\Pi_n:= \hat{G}(X_{\et}, \bar{x})/[K]_{n+1}$.
% For each section $\sigma $ of $P \to G_{F,\Sigma}$, we write 
% \[
%  X(\bA_F^{\in \Sigma})_{\sigma} := X(\bA_F^{\in \Sigma})\by^h_{\map_{BG_{F,\Sigma}}(BG_{ \bA_F^{\in \Sigma}},BP)}\{\sigma\}.
% \]

%The simplicial abelian profinite groups $[K]_s/[K]_{s+1}$ have natural $P$-actions, and we write $([K]_s/[K]_{s+1})_{\sigma} $ for the simplicial  $G_{F,\Sigma}$-representation then resulting from the section $\sigma$. 

As before, we define a tower $\{X(\bA_F^{\in \Sigma})_n\}_n$ by
\[
X(\bA_F^{\in \Sigma})_n:= X(\bA_F^{\in \Sigma})\by^h_{\map_{BG_{F,\Sigma}}(BG_{ \bA_F^{\in \Sigma}}, \bar{W}\Pi_n ) }\map_{BG_{F,\Sigma}}(BG_{F,\Sigma}, \bar{W}\Pi_n );
\]
%Substituting our definition of $\Pi_n$ in the definition of the tower $X(\bA_F^{\in \Sigma})_n$ in  
%Example \ref{exhtpyBrauerManin}, this 
note that points in $ X(\bA_F^{\in \Sigma})_0$ now include the data of sections of $P \to G_{F,\Sigma}$, because $\Pi_0=P$.
The reasoning of Example \ref{nilpex3} again gives a non-abelian spectral sequence 
\[
 E_1^{s,t}= \begin{cases}
             \bH^{1+s-t}_c( G_{F,\Sigma}, [K]_s/[K]_{s+1})& s \ge 1 \\
 \pi_{t}X(\bA_F^{\in \Sigma})_0 & s=0
            \end{cases}
  \abuts  \pi_{t-s} X(\bA_F^{\in \Sigma})_{\infty}
\]
of groups and sets. 
The terms $\bH^{1+s-t}_c( G_{F,\Sigma}, [K]_s/[K]_{s+1})$ depend on the section $\sigma$ of $P \to G_{F,\Sigma}$ induced by the relevant element of $ \pi_0X(\bA_F^{\in \Sigma})_0$,  
the Galois action then coming from the natural $P$-action on $[K]_s/[K]_{s+1}$.

As in Example \ref{descentex1}, each section $\sigma$ above  gives a pro-(finite \'etale $\Sigma$-torsion) group scheme $P^{\sigma}%= \Spec \Hom_{\Set}(\ker(P \to G_{F,\Sigma}), \bar{\O}_{F,\Sigma})^{G_{F,\Sigma}}
$ 
over $\sO_{F,\Sigma}$ with $BP^{\sigma}$ having \'etale homotopy type $BP$, and  maps $X_{\et} \to BP$ 
correspond to  $P^{\sigma}$-torsors $f^{\sigma} \co Y^{\sigma} \to X$. %%$Y^{\sigma}(\bA_F^{\in \Sigma})$ the homotopy fibre of $X(\bA_F^{\in \Sigma})_0 \to \map_{BG_{F,\Sigma}}(BG_{F,\Sigma}, BP )$ over $\sigma$
The first obstruction map $d_1 \co E_1^{0,0} \to E_1^{1,0}$ in the spectral sequence above is the disjoint union, over inner automorphism classes of sections  $\sigma$, %$\in \H^1(G_{F,\Sigma},P)$ 
of the  Brauer--Manin obstructions
\[
 \pi_0 Y^{\sigma}(\bA_F^{\in \Sigma})/P^{\sigma}(\sO_{F,\Sigma}) \to \prod_{\ell} (\H^{2}_{\pro(\et)}( Y^{\sigma}, \mu_{\ell^{\infty}})/\H^2(G_{F,\Sigma}, \mu_{\ell^{\infty}}))^{\vee}
\]
 of the $Y^{\sigma}$ (defined as derived limits unless $\ker(P \to G_{F,\Sigma})$ is finite),  
so we have
\[
\im(\pi_0X(\bA_F^{\in \Sigma})_1 \to \pi_0X(\bA_F^{\in \Sigma})) = \bigcup_{\substack{\sigma\co G_{F,\Sigma} \to P \\ \text{a section}}}  f^{\sigma}(\pi_0 Y^{\sigma}(\bA_F^{\in \Sigma})^{\Br_{\Sigma}})
\]
(when $X$ is an algebraic space, we can drop the $\pi_0$'s). When $\sO_{F,\Sigma}=F$, combining these for  all finite extensions $P$ of $G_{F} $ will thus give Skorobogatov's \'etale Brauer--Manin obstruction \cite{skorobogatovBeyondManin}.

%Since an inverse limit of non-empty sets  can be empty, it seems that considering pro-\'etale covers in this way might give a stronger obstruction than \'etale Brauer--Manin in general. However, 
For  smooth  proper varieties,  the space of ad\'elic points is compact,
and by  Tychonoff's theorem the inverse limit of non-empty compact spaces is non-empty,  so  considering pro-\'etale covers in this way %$X(\bA_F^{\in \Sigma})^{\pro(\et),\Br_{\Sigma}}= X(\bA_F^{\in \Sigma})^{\et,\Br_{\Sigma}}$ 
will just recover the \'etale Brauer--Manin obstruction 
in this case. 

When $F=\sO_{F,\Sigma}$, the universal case to consider would take $P= \pi_1^{\et}(X, \bar{x})$, with the spectral sequence then detecting exclusively higher homotopical information, and $\bar{Y}^{\sigma} $ being a universal cover $\tilde{X}$ of $X$. For this choice of $P$, we may therefore set
\begin{align*}
 \pi_0X(\bA_F)^{\pro(\et),\Br} :&= \im(\pi_0X(\bA_F)_1 \to \pi_0X(\bA_F))\\
&= \bigcup_{\substack{\sigma\co G_{F} \to \pi_1^{\et}(X, \bar{x}) \\ \text{a section}}}  f^{\sigma}(\pi_0 Y^{\sigma}(\bA_F)^{\Br})
\end{align*}
(again, we can drop the $\pi_0$'s when $X$ is an algebraic space).

%their universal covers in profinite \'etale homotopy theory can be represented by schemes for which the Brauer--Manin obstruction detects the failure of the Hasse principle. 

Since $G_{F}$ has cohomological dimension $2$, the higher homotopy groups $\pi_{\ge 2}( [K]_s/[K]_{s+1} )$ never contribute to  the obstruction spaces $E_1^{s,s-1}$ for $\pi_0X(\bA_F)$ in the non-abelian spectral sequence above. 
For the universal case $P= \pi_1^{\et}(X, \bar{x})$, we have $\pi_iK= \pi^{\et}_{i+1}(\tilde{X})$, and $\pi_1[K]_2=0$ (the Hurewicz map for $\pi_2$ being an isomorphism). Thus $E_1^{s,s-1}=0$ for $s>1$, meaning   all  higher obstructions vanish and 
\[
 \pi_0X(\bA_F)^{\pro(\et),\Br}= \im(\pi_0X(\bA_F)_{\infty} \to \pi_0X(\bA_F)). 
\]

Moreover the sequence $[K]_n$ is increasingly connected, so $\Pi_{\infty}\simeq \hat{G}(X_{\et}, \bar{x})$. Together, these phenomena imply that vanishing of the pro-\'etale Brauer--Manin obstruction alone implies the existence of a compatible section of  the map $X_{\et}^{\wedge} \to (\Spec \sO_{F})_{\et}^{\wedge}$ of profinite \'etale homotopy types when $X$ is geometrically connected. 
This is not nearly as impressive as it might seem, since the construction of the pro-\'etale Brauer--Manin obstruction assumes a compatible section of $\pi_1^{\et}(X) \to G_{F}$. 

%%\cite{schlankBrauerManinRamified} shows taht we'd need to pass to ramified covers to account for Poonen's counterexample. Might combine nicely with \cite{cesnaviciusPT}
\end{example}

\begin{remark}[Relation to Harpaz--Schlank]
Our spaces $X(\bA_F)_n$ in this section are closely related to those of \cite{harpazschlank2}, which (after including Archimedean places) considers spaces $X(\bA_F)^h$ broadly of the form 
\[
X(\bA_F)\by^h_{\map_{BG_{F}}( (\Spec \bA_F)_{\et}, X_{\et} }\map_{BG_{F}}(BG_{F}, X_{\et} ),
\]
 as well as variants $ X(\bA_F)^{\Z h}$,  $X(\bA_F)^{h,n}$ and $X(\bA_F)^{\Z h,n}$. In our terms, $X(\bA_F)^{\Z h} $ corresponds to replacing $X_{\et}$ with $\bar{W}(G(X_{\et})^{\ab})$ above;  the others are given by taking Postnikov towers. 

Rather than imposing smoothness hypotheses and appealing to \cite{friedlanderfib} as we have done, \cite{harpazschlank2} constructs a $G_F$-equivariant homotopy type $\Et_{/K}(X)$, and effectively works with the homotopy quotient  $\Et_{/K}(X)/^hG_F$ in place of $X_{\et}$ above. In \cite[Theorem 11.1]{harpazschlank2}, 
the \'etale Brauer set is shown to correspond to   the set $X(\bA_F)^h$, which is a somewhat stronger statement than our final  observation in Example \ref{exhtpyEtBrauerManin}.

The main new ingredient in our constructions and comparisons is that by modelling profinite homotopy types as simplicial profinite groups and groupoids following \cite[\S \ref{weiln-profinitesn}]{weiln} and \cite[Proposition 1.19]{ddt1}, we are able to work systematically  with much more general towers than the Postnikov tower.
\end{remark}

\subsection{Alternative characterisations of the reciprocity laws}\label{alternative}

We now give a more pedestrian interpretation of the obstruction maps from \S \ref{obsthsn}, and show how this can give rise to a more explicit description of the first obstruction map in cases of interest. This first obstruction map seems to be well-known to experts, but we are not aware of a reference.

\subsubsection{Cohomological obstruction classes}

Extensions    $e\co 0 \to A \to  \Pi'' \to \Pi' \to 1$ of a group $\Pi'$  by an abelian $\Pi'$-representation $A$ are classified by
\[
 \H^2(\Pi',A),
\]
 by which we mean continuous cohomology when considering extensions of topological groups. 

Given a group homomorphism $\psi \co G \to \Pi'$, the obstruction to lifting $\psi$ to a homomorphism $\tilde{\psi} \co G \to \Pi''$ is then given by
\[
 \psi^*[e] \in \H^2(G,A).
\]
If $\psi^*[e]=0$, then  the difference between two choices for $\tilde{\psi}$ is a derivation, so the set of choices is a torsor for the group
\[
 \H^1(G,A).
\]

Taking $\Pi',\Pi''$ to be suitable quotients of the arithmetic fundamental group of a scheme $X$ over  $\sO_{F,\Sigma}$, 
the Diophantine obstruction maps on spaces of sections
\[
 \pi_0\map_{B G_{F,\Sigma}}(BG_{F,\Sigma}, B\Pi') \to \H^2(G_{F,\Sigma},A)
\]
of \S \ref{rationalsn} are all of this form.  The ad\'elic obstruction maps of \S \ref{adelicmapsn} are a slight variant coming from looking at restricted products 
\[
 \prod'_{v \in \Sigma}\pi_0\map_{B G_{F,\Sigma}}(BG_v, B\Pi')\to \prod'_{v \in \Sigma} \H^2(G_v,A).
\]

The reciprocity maps associated to an $\bA_F^{\in \Sigma}$-point in  \S \ref{recipsn} then effectively look at the difference between these obstructions, yielding an obstruction in $ \H^2_c(G_{F,\Sigma},A)$ via the  exact sequence
\[
 \prod'_{v \in \Sigma} \H^1(G_v,A) \xra{\pd} \H^2_c(G_{F,\Sigma},A) \to \H^2(G_{F,\Sigma},A)\to \prod'_{v \in \Sigma} \H^2(G_v,A).
\]

In general, this is not very easy to work with, but when the extension $e$ splits, so $\Pi'' = \Pi' \ltimes A$, the ad\'elic point defines a derivation in $\alpha \in \prod'_{v \in \Sigma} \H^1(G_v,A)$, with associated abelian obstruction $\pd(\alpha) \in \H^2_c(G_{F,\Sigma},A)$   to lifting the ad\'elic point to a rational point.

\begin{example}\label{abnilpex}
In nilpotent or unipotent settings such as Example \ref{unipex3}, the first stage in the tower is a split extension
\begin{align*}
G_{F} \ltimes  \pi_1^{\et}(\bar{X},\bar{x})^{\ab}  &\to G_{F},\\
 G_{F,\Sigma} \ltimes  (\pi_1^{\et}(\bar{X},\bar{x})\ten \Ql)^{\ab} \cong G_{F,\Sigma} \ltimes  \H^1(\bar{X}, \Ql)^* &\to G_{F,\Sigma}.
\end{align*}
 
Then an $\bA_{F}^{\in \Sigma}$-point $y$ defines a class in $ \H^1(\bA_{F}^{\in \Sigma},\H^1(\bar{X}, \Ql)^*)$ whose image in $\H^2_c(G_{F,\Sigma},\H^1(\bar{X}, \Ql)^*)$ is the first unipotent obstruction to $y$ being a rational point.  
\end{example}

\begin{example}\label{abrelmalex}
Relative Malcev completions as in Example \ref{modularex4} are a little more complicated. For  $X= \cM_{1,1}$  the stacky modular curve, take $x \in X(\sO_{F,\Sigma})$, giving rise to a $ G_{F,\Sigma}$-representation $V$ of dimension $2$ over $\Ql$. We then set
set $P_0= G_{F,\Sigma} \ltimes \SL_2(\Ql)$, and
\begin{align*}
 P_1 &:=  G_{F,\Sigma} \ltimes( \SL_2(\Z)^{\SL_2, \mal}/[\Ru]_{2}),\\
&=  G_{F,\Sigma} \ltimes( \H^1(\SL_2(\Z), O(\SL_2))^* \rtimes  \SL_2(\Ql) ).
\end{align*}
with $\Pi_i= P_i\by_{\SL_2(\Ql)}\SL_2(\Zl)$, where we are writing $O(\SL_2)$ for the ring of algebraic functions on the scheme $\SL_2$ over $\Ql$.

Now, $P_1$ is an extension of $P_0$ by $\H^1(\SL_2(\Z), O(\SL_2)\ten \Ql)^* $, so is given by a class in $\H^2(P_0, \H^1(\SL_2(\Z), O(\SL_2))^*)$, where we may regard $\SL_2(\Ql)$ as an algebraic group. Since $\SL_2$ is reductive, the Leray--Serre spectral sequence then gives
\[
 \H^2(P_0, \H^1(\SL_2(\Z), O(\SL_2))^*) \cong \H^2(G_{F,\Sigma}, (\H^1(\SL_2(\Z), O(\SL_2))^*)^{\SL_2}),
\]
which vanishes because $\H^1(\SL_2(\Z), \Ql)=0$. 

We therefore have a split extension $\Pi_1 \cong \Pi_0 \ltimes \H^1(\SL_2(\Z), O(\SL_2))^*$. (For more general relative Malcev completions, a similar conclusion will still hold by combining Leray--Serre with the splitting of the extension $\Pi_1 \to G_{F,\Sigma}$.)

Thus  an ad\'elic elliptic curve $E$ defines a class in $\H^1(\bA_{F}^{\in \Sigma},\H^1(\SL_2(\Z), O(\SL_2))^*)$, whose image in $\H^2_c(G_{F,\Sigma},\H^1(\SL_2(\Z), O(\SL_2))^*)$  is the first obstruction to $E$ being defined over $\sO_{F,\Sigma}$ with Tate module $T_{\ell}(E(\bar{F}))\ten \Q \simeq V$.
\end{example}

\subsubsection{The first obstruction for modular curves} %%no need ot worry about relations here, as only first obstruction
We now give an explicit description of the abelian obstruction of Example \ref{abrelmalex}, seeking elliptic curves with given Tate module. 

On the modular curve $q \co Y_{\Gamma} \to \Spec \sO_{F,\Sigma}$, the Tate module of the universal elliptic curve $f \co E \to Y_{\Gamma}$ gives a lisse $\Zl$-sheaf $\bT_{\ell}$ of rank $2$, and we write $\bT_{\Ql}:=\bT_{\ell}\ten \Q$.
%
%we have a canonical lisse sheaf $\bV_1$ of rank $2$ over $\Ql$ with determinant $\Ql(-1)$, given by $\bV_1:= \oR^1f_*\Ql$ for the universal elliptic curve $f \co E \to Y_{\Gamma}$. We set $\bV_m:=S^m\bV_1$ and 
%
On pulling back to $\bar{Y}_{\Gamma}$, the sheaves $S^m\bT_{\Ql}$ correspond to the irreducible representations $V_m$ of $\SL_2$, and we 
consider the Galois representations $\H^1(\Gamma,V_m):= \oR^1q_*S^m\bT_{\Ql} $. For each $m$,  the adjunction $q^* \dashv \oR q_*$ defines a  class 
\[
\eta_m \in \Ext^1_{Y_{\Gamma},\Ql}(q^* \H^1(\Gamma,V_m), S^m\bT_{\Ql}).
\]

Now take an ad\'elic point $x \in Y_{\Gamma}(\bA_F^{\in \Sigma})$, and assume that there is a $G_{F,\Sigma}$-representation $\Lambda$ with $\det \Lambda= \Zl(1)$ and an isomorphism  $\alpha \co \Lambda\ten\Q \cong\bT_{\Ql,x}$ which is $G_v$-equivariant for all $v \in \Sigma$. A necessary condition for $x$ to lie in $Y_{\Gamma}(\sO_{F,\Sigma})$ compatibly with $\alpha$  is that the class $x^*\eta_m \in \prod_{v \in \Sigma}\Ext^1_{G_v}(\H^1(\Gamma,V_m), S^m\Lambda\ten\Q)
$ lies in the image of $ \Ext^1_{G_{F,\Sigma}}(\H^1(\Gamma,V_m), S^m\Lambda\ten\Q)$. Following the conventions of \S \ref{abPT} to replace the product with a suitable restricted product, we get an obstruction
\[
 \pd(x^*\eta_m) \in \H^2_c(G_{F,\Sigma},\H^1(\Gamma,V_m)^* \ten S^m\Lambda).
\]

Combining these gives a map
\[
 \H^1(G_{F,\Sigma}, \GL_2(\Ql))\by_{\H^1(\bA_F^{\in\Sigma}, \GL_2(\Ql))}Y_{\Gamma}(\bA_F^{\in \Sigma}) \to \prod_{m\ge 1} \H^2_c(G_{F,\Sigma}, \H^1(\Gamma, V_m)^{*}\ten S^m\Lambda),
\]
which is the first reciprocity map associated to the relative completion of $\Gamma \to \SL_2(\Ql)$ in Example \ref{modularex4}, via the isomorphism $O(\SL_2)\ten \Ql \cong \bigoplus_m V_m\ten V_m^*$. We may then use Poitou--Tate duality as in Example \ref{exhtpyBrauerManin} to rewrite the target of the map as
\[
 \prod_{m\ge 1} (\H^2_{\et}(Y_{\Gamma}, S^m\bT_{\ell} \ten q^*S^m\Lambda^*\ten \mu_{\ell^{\infty}})^{\vee}\ten\Q);
\]
adapting Example \ref{exhtpyEtBrauerManin}, this can be recovered from the Brauer--Manin obstruction of an inverse system of finite \'etale covers of $Y_{\Gamma}$, which in this case correspond to twisted level structures associated to the  $G_{F,\Sigma}$-representations $\Lambda/\ell^n$.

\begin{remark}
 An intermediate step in the construction above associates to each  elliptic curve $E$ over $F$ a class in 
\[
 \Ext^1_{G_{F,\Sigma}}(\H^1(\Gamma,V_m), S^mT_{\ell}(E(\bar{F})\ten\Q)).
\]
The corresponding construction for complex elliptic curves and mixed Hodge structures is given in   \cite[Remark 13.3]{hainHodgeDRmodular} (evaluating the section at the point $[E]$). The extension arises geometrically as the relative cohomology group $\H^1(\bar{Y}_{\Gamma},[E];S^m\bT_{\Ql})$.%, and hence as a summand of  $\H^{m+1}(\cM_{1,m+1}(\Gamma)\ten \bar{F}, E^m\ten \bar{F}; \Ql)$. 

\end{remark}

% % To relate this to the Brauer--Manin obstruction, we might embed $O(\SL_2)\ten \Zl$ in the ring of continuous functions from $\SL_2(\Zl)$ to $\Zl$, and consider the associated pro-\'etale cover of $Y_{\Gamma}$ as in Example \ref{exhtpyEtBrauerManin}. To do this, we first associate a pro-(finite \'etale) group scheme $L=\Lim_n L_n$ over $\sO_{F,\Sigma}$ to 
% %  the Galois representation $\rho \co G_{F,\Sigma} \to  \GL(\Lambda)$,  with $L_n(\sO_{\bar{F},\Sigma})\cong\Lambda/\ell^n$. This then determines an inverse system of \'etale covers  $\tilde{Y}_{\Gamma}^{\rho}:= \{Y_{\Gamma(\ell^n)}^{\rho}\}_n$ of $Y_{\Gamma}$, parametrising elliptic curves  with twisted level structure  in the form of a map from $L_n$. There is then a natural injective map from $S^m\bT_{\ell} \ten q^*S^m\Lambda^*/\ell^n$ to the direct image of the constant sheaf $\Z/\ell^n$ on $Y_{\Gamma(\ell^r)}^{\rho}$ for $r \ge n$, so the map above factors through the Brauer--Manin obstruction
% % \[
% %  \tilde{Y}_{\Gamma}^{\rho}(F) \to  (\H^2_{\pro(\et)}( \tilde{Y}_{\Gamma}^{\rho},\mu_{\ell^{\infty}})^{\vee}.
% % \]

\subsubsection{Higher Brauer--Main obstructions via cochain algebras}\label{highercochain}

The unipotent obstructions which we have considered were formulated in terms of morphisms of simplicial pro-unipotent groups, so could be thought of as a form of Quillen homotopy type \cite{QRat}. An equivalent alternative formulation would be to look at morphisms of Sullivan homotopy types \cite{Sullivan}, which are just algebras of cochains.

Taking a Deligne--Mumford stack $X$ over $\sO_{F,\Sigma}$ and writing $\bar{X}:=X\ten \sO_{\bar{F},\Sigma}$, the cochain complex $\oR\Gamma(\bar{X},\Ql)$ carries a natural cup product, and  is in fact naturally quasi-isomorphic to a commutative differential graded algebra over $\Ql$. Equivalently this means that $\oR\Gamma(\bar{X},\Ql)$ carries the structure of a unital $\Com_{\infty}$-algebra (or strongly homotopy commutative algebra): it has a symmetric bilinear multiplication $m_2$, which is associative up to a homotopy $m_3$, and there is a hierarchy of higher homotopies $m_n$ formulated in terms of the Lie operad. In the $R=1$ case, Example \ref{exhtpy3} looks at the morphism 
\[
 \oR\Gamma(\bar{X},\Ql) \to \Ql
\]
defined by an ad\'elic point, and studies obstructions to lifting it to a $\Com_{\infty}$-morphism $\{f_n\}_{n\ge 1}$  which is equivariant for the global Galois group $G_{F,\Sigma}$, rather than just the pro-groupoid 
\[
 G_{\bA_F^{\in \Sigma}}:=  \Lim_{\substack{T\subset \Sigma\\ T \text{finite}}}( \coprod_{v\in T }G_{v}\sqcup\coprod_{v \in \Sigma -T}G_{v}/I_{v})
\]
formed from local Galois groups.

\begin{enumerate}
 \item The first reciprocity law seeks just to lift this as a morphism of complexes, fixing $\Ql \subset \oR\Gamma(\bar{X},\Ql) $,  so the first obstruction lies in 
\[
 \EExt^1_{G_{F,\Sigma},c}( \oR\Gamma(\bar{X},\Ql)/\Ql, \Ql) \cong (\H^2(X, \mu_{\ell^{\infty}})/\H^2(G_{F,\Sigma},\mu_{\ell^{\infty}}))^{\vee}\ten \Q;  
\]
this is just the rational Brauer--Manin obstruction.

\item The secondary obstruction of \S \ref{BMsn} depends on a choice $f_1\co \oR\Gamma(\bar{X},\Ql) \to \Ql $  of  $G_{F,\Sigma}$-equivariant chain map, together with a  homotopy $h_1$ of $G_{\bA_F^{\in \Sigma}}$-representations making $f_1$ compatible with our chosen ad\'elic point.  Such a lift exists whenever the rational $\ell$-torsion Brauer--Manin obstruction vanishes, and we now need to look at whether it respects the cup product. We thus ask whether the diagram
\[
 \xymatrix{
  \oR\Gamma(\bar{X},\Ql)\ten \oR\Gamma(\bar{X},\Ql) \ar[rd]_-{f_1\ten f_1} \ar[r]^-{m_2}  &   \oR\Gamma(\bar{X},\Ql) \ar[d]^{f_1}\\
&\Ql
 }
\]
commutes, up to a  homotopy $f_2$, in the derived category of $G_{F,\Sigma}$-representations,  with a further $G_{\bA_F^{\in \Sigma}}$-equivariant homotopy $h_2$ between $f_2$ and 
% its known $G_{\bA_F^{\in \Sigma}}$-equivariant  commutativity provided by 
the homotopy $f_1\ten h_1 + h_1 \ten f_1 + (h_1d) \ten h_1 - h_1\circ m_2$ providing the known $G_{\bA_F^{\in \Sigma}}$-equivariant  commutativity of $f_1$. The resulting obstruction lies in 
\[
 \EExt^0_{G_{F,\Sigma},c}( \oR\Gamma(\bar{X}^2,\Ql),\Ql)\cong \H^3(X^2, \mu_{\ell^{\infty}})^{\vee}\ten \Q,
\]
but this restricts to the finer obstruction  described in Remark \ref{LieBMremark} when we take symmetry and the unit into account.

\item  The third obstruction is more complicated, measuring obstructions to choosing the next component $(f_3,h_3)$ of a $\Com_{\infty}$-morphism. If we choose a model $A$ of  $\oR\Gamma(\bar{X},\Ql)$ which is strictly (graded-)commutative, this means we seek a map $f_3 \co A^{\ten 3} \to A[-1]$ satisfying
\begin{align*}
% &df_3(a,b,c) \mp f_3(da,b,c) \mp f_3(a,db,c) \mp f_3(a,b,dc)\\
 (d\circ f_3 \mp f_3\circ d)(a,b,c)  = f_2(ab,c) \pm f_2(a,bc) \mp f_1(a)f_2(b,c) \mp f_2(a,b)f_2(c),
\end{align*}
which must vanish on the unit $1 \in A$ and  on  shuffle products. The %second line 
right-hand side and associated $G_{\bA_F^{\in \Sigma}}$-equivariant homotopy in terms of $h_2$
give rise to  an obstruction class in 
\[
 \EExt^{-1}_{G_{F,\Sigma},c}( \oR\Gamma(\bar{X}^3,\Ql),\Ql)\cong \H^4(X^3, \mu_{\ell^{\infty}})^{\vee}\ten \Q,
\]
which is closely related to Massey triple products

\item Explicit  descriptions for the higher obstructions follow from the formulae for $\Com_{\infty}$-morphisms as in \cite[\S\S 10.2.2, 13.1.13]{lodayvalletteOperads} (take the expression for $A_{\infty}$-morphisms in \cite[Proposition 10.2.12]{lodayvalletteOperads} and replace $\cA s$ with $\cL ie$ by taking invariants under shuffle permutations). These are related to higher Massey products. 
\end{enumerate}

To express Example \ref{exhtpy3} in these terms beyond the $R=1$ case, we may reformulate via \cite[Proposition 3.15 and Corollary 4.41]{htpy}  to seek $G_{F,\Sigma} \ltimes R$-equivariant morphisms
\[
 \oR\Gamma(\bar{X},O(R)) \to O(R),
\]
for a pro-reductive algebraic groupoid  $R$ over $\Ql$ and a Zariski dense Galois-equivariant homomorphism $\pi_1(\bar{X}, \cB) \to R(\Ql)$ with a Galois-equivariant set of basepoints $\cB$. The descriptions above adapt, with the sheaf $O(R)_{\Zl}\ten \mu_{\ell^{\infty}}$ (regarded as a  $\pi_1(\bar{X}, \cB) \by R$-representation via the left and right actions) replacing $\mu_{\ell^{\infty}}$. 
 
\begin{remark}
 If we wished to construct obstructions in the nilpotent, rather than unipotent setting,  we should seek Galois-equivariant morphisms $\oR\Gamma(\bar{X},\hat{\Z}) \to \hat{\Z}$ of cosimplicial commutative rings. The first obstruction is just Brauer--Manin, but the torsion in the  higher obstructions is very difficult to describe,  as discussed in Remark \ref{LieBMrmk2}. 
\end{remark}

\begin{remark}
 The description in terms of cochain algebras will readily adapt to more general cohomology theories with cup product. For instance, a motivic analogue of \S \ref{rationalsn} would be given by seeking $\Com_{\infty}$-morphisms $M(Y_{\Gamma}) \to M(F)$ of cohomological $F$-motives, assuming existence of a suitable $\Com_{\infty}$-structure enriching the cup product on motivic cohomology. The obstruction tower just depends on a filtration on the $\Com_{\infty}$-operad, whereas a Postnikov-type filtration in terms of motivic homotopy groups \cite[\S \ref{HHtannaka2-motsn}]{HHtannaka2} would require a suitable $t$-structure. This approach could also be used to construct motivic obstructions to ad\'elic points being global, along the lines of this section, but it is not obvious what the motivic analogue of Poitou--Tate duality should be.
\end{remark}

\appendix

\section{Pro-finite homotopy types for ad\`eles}\label{adelehtpy}

\begin{definition}
Write $s^{\flat}\gpd$ for the category consisting of simplicial groupoids $G$ for which
\begin{enumerate}
\item The simplicial set $\Ob G$ of $G$ is constant and finite;
 \item each $G_i(x,y)$ is finite;
\item the group $N_iG(x,x):= G_i(x,x) \cap \bigcap_{j>0} \ker \pd_j$ is trivial for all but finitely many $i$. %%implies at most one map $x \to y$ for $i\gg 0$.
\end{enumerate}
Note that the second condition is equivalent to saying that the map $G \to \cosk_nG$ to the $n$-coskeleton is an isomorphism for sufficiently large $n$. 
\end{definition}

\begin{lemma}\label{profgplemma}
The functor $U$ from $\pro(s^{\flat}\gpd)$ to simplicial profinite groupoids given by $(U\{G(\alpha)\}_{\alpha})_n := \{G(\alpha)_n\}_{\alpha}$ 
is an equivalence of categories; moreover, we may restrict to inverse systems in which all morphisms are surjective.
\end{lemma}
\begin{proof}
 Since $s^{\flat}\gp$ is an Artinian category, the proofs of \cite[Proposition \ref{ddt1-cSp}]{ddt1} (which dealt with Artinian local rings rather than finite groupoids) and of \cite[ Lemma \ref{weiln-levelwiseproL}]{weiln} carry over to this generality.
\end{proof}

\begin{definition}
 Given a simplicial scheme $Y$, define $\Gamma^{\bS}(Y,-)$ to be the global sections functor from simplicial \'etale presheaves on $Y$ to simplicial sets. Write $\oR\Gamma^{\bS}_{\et}(Y,-)$ for its right-derived functor with respect to the model structure for \'etale hypersheaves. Explicitly,
\[
 \oR\Gamma^{\bS}_{\et}(Y,\sF)\simeq \ho\LLim_{Y'_{\bt}} \ho\Lim_{n \in \Delta} \Gamma(Y'_n,\sF),
\]
where $Y'_{\bt}$ runs over simplicial \'etale hypercovers of $Y$.

Given an inverse system $\sF=\{\sF_i\}_i$, set
\[
 \oR\Gamma^{\bS}_{\et}(Y,\sF):= \ho\Lim_i \oR\Gamma^{\bS}_{\et}(Y,\sF_i).
\]
 \end{definition}

\begin{lemma}\label{adelemap}
There is a canonical morphism
\[
 \oR\Gamma^{\bS}_{\et,\cts}(\Spec \bA_F^{\in \Sigma},\bar{W}G)\to\map( BG_{\bA_F^{\in \Sigma}} , \bar{W}G)
\]
in $\Ho(\bS)$,
functorial in  simplicial  profinite groupoids $G$. 
\end{lemma}
\begin{proof}
Because $\Spec \bA_F^{\in \Sigma}$ is quasi-compact, the category of quasi-compact hypercovers of $\Spec \bA_F^{\in \Sigma}$ is left filtering in the category of all hypercovers, by the argument of \cite[Proposition 7.1]{fried}. Thus for all simplicial presheaves $\sF$,
\[
 \oR\Gamma^{\bS}_{\et}(\Spec \bA_F^{\in \Sigma},\sF)\simeq\holim_{ \substack{ \lra \\Y'_{\bt} \in \HR(\Spec \bA_F^{\in \Sigma})}} \ho\Lim_{n \in \Delta} \Gamma(Y'_n,\sF) \la \holim_{\substack{ \lra \\Y'_{\bt} \in qc\HR(\Spec \bA_F^{\in \Sigma})}} \ho\Lim_{n \in \Delta} \Gamma(Y'_n,\sF)
\]
is an equivalence,  where $\HR(Y)$ is the category of simplicial hypercovers $Y'_{\bt} \to Y$ and $qc\HR(Y)$  the full subcategory of simplicial hypercovers $Y'_{\bt} \to Y$  with each $Y_n'$ quasi-compact. 

Given a simplicial presheaf $\sF$ for which the map $\sF \to \cosk_m\sF$ is an isomorphism, the map
\[
\holim_{\substack{ \lra \\ Y'_{\bt} \in qc\HR(\Spec \bA_F^{\in \Sigma})}} \ho\Lim_{n \in \Delta} \Gamma(Y'_n,\sF) \la \holim_{\substack{ \lra \\Y'_{\bt} \in qc\HR^{\flat}(\Spec \bA_F)}} \ho\Lim_{n \in \Delta} \Gamma(Y'_n,\sF)
\]
is an equivalence, where $ qc\HR^{\flat}(\Spec \bA_F^{\in \Sigma})$ consists of  quasi-compact hypercovers $Y'$ which are truncated in the sense that  $Y'= \cosk_r(Y'/\bA_F^{\in \Sigma})$ for some $r$ (in fact $r=m$ suffices for the case in hand).   

Given a quasi-compact hypercover $Y'_{\bt} \to \Spec \bA_F^{\in \Sigma}$, write $Y'_{\bt,v}$ for its pullback along $\Spec F_v \to \Spec \bA_F^{\in \Sigma}$. Thus each $Y'_{i,v}$ is the spectrum of a finite product of finite field extensions of $F_v$. Because $Y'_i$ is of finite type over $\bA_F^{\in \Sigma}$, it is defined over $(\prod_{v\in \Sigma} \sO_v)\ten_{\Z}\Z[S^{-1}_i]$ for some finite set $S_i\subset \Sigma$ of primes. For $v \in S_i$, it then follows that $Y'_{i,v}$ is the spectrum of a finite product of finite unramified field extensions of $F_v$.  When the hypercover $Y'_{\bt}$ is $r$-truncated, we can set $S= \bigcup_{i \le r} S_i$, and then see that 
\[
 \{Y'_{\bt,v}\}_v \in (\prod_{v \in \Sigma -S} qc\HR^{\nr}(\Spec F_v) )\by (\prod_{v \in S} qc\HR(\Spec F_v) ) \subset \prod_v \HR(\Spec F_v),
\]
where $qc\HR^{\nr} $ consists of quasi-compact hypercovers built from  unramified field extensions.

Writing 
\[
 \prod'_v qc\HR(\Spec F_v):= \bigcup_{S \subset \Sigma \text{ finite }} (\prod_{v \in \Sigma-S} qc\HR^{\nr}(\Spec F_v) )\by (\prod_{v \in S} qc\HR(\Spec F_v) ),
\]
we then get a map
\[
 \holim_{\substack{ \lra \\Y'_{\bt} \in qc\HR^{\flat}(\Spec \bA_F^{\in \Sigma})}} \ho\Lim_{n \in \Delta} \Gamma(Y'_n,\sF) \to \holim_{\substack{ \lra \\Y'_{\bt} \in \prod'_vqc\HR^{\flat}(\Spec F_v)}} \ho\Lim_{n \in \Delta} \Gamma(Y'_n,\sF).
\]

Returning to the statement of the lemma, since both functors send filtered inverse limits to homotopy limits, Lemma \ref{profgplemma} allows us to restrict to the case where $G \in s^{\flat}\gpd$. Thus  the map $G \to\cosk_{m-1}G$  is an isomorphism for some $m$, so $\bar{W}G\cong \cosk_m\bar{W}G $ and satisfies the conditions for $\sF$ above. Then we have
\[
 \oR\Gamma^{\bS}_{\et}(\Spec \bA_F^{\in \Sigma},\bar{W}G) \to \holim_{\substack{ \lra \\Y'_{\bt} \in \prod'_vqc\HR^{\flat}(\Spec F_v)}} \prod_v \ho\Lim_{n \in \Delta} \Gamma((Y'_v)_n,\bar{W}G).
\]
Now,  we can rewrite the right-hand side as
\[
 \LLim_{S\subset \Sigma  \text{ finite }} \ho\Lim_{n \in \Delta}((\prod_{v \in \Sigma-S} \holim_{\substack{ \lra \\(Y'_v)_{\bt} \in qc\HR^{\nr}(\Spec F_v)}} \Gamma((Y'_v)_n,\bar{W}G)) \by (\prod_{v \in S}  \holim_{\substack{ \lra \\(Y'_v)_{\bt} \in qc\HR(\Spec F_v)}} \Gamma((Y'_v)_n,\bar{W}G)).
\]
Since $(\Spec F_v)_{\et}^{\wedge} \simeq BG_v$ and $(\Spec \sO_{F,v})_{\et}^{\wedge} \simeq B(G_v/I_v)$ this is weakly equivalent to
\[
 \LLim_{S\subset \Sigma  \text{ finite }} \prod_{v \in \Sigma-S}\map(B(G_v/I_v), \bar{W}G) \by \prod_{v \in S}\map(BG_v, \bar{W}G),
\]
which is just $\map( BG_{\bA_F^{\in \Sigma}} , \bar{W}G)$, as required.
\end{proof}

\begin{corollary}\label{adelemapcor}
There is a canonical morphism
\[
BG_{\bA_F^{\in \Sigma}} \to (\Spec \bA_F^{\in \Sigma})_{\et}^{\wedge}
\]
in the homotopy category of pro-simplicial sets, where $\wedge$ denotes profinite completion, and $X_{\et}$  the \'etale  topological type   as in \cite[Definition 4.4]{fried}.
\end{corollary}
\begin{proof}
Since simplicial  profinite groupoids model profinite homotopy types by \cite[Proposition \ref{weiln-cohochar}]{weiln}, it suffices to show that we have natural morphisms
\[
 \map((\Spec \bA_F^{\in \Sigma})_{\et}, \bar{W}G) \to \map( BG_{\bA_F^{\in \Sigma}} , \bar{W}G)
\]
for  simplicial  profinite groupoids $G$, and  this is precisely the content of Lemma \ref{adelemap}.
\end{proof}

\bibliographystyle{alphanum}
\bibliography{references.bib}

\end{document}